\documentclass[aop]{imsart}

\RequirePackage{amsthm,amsmath,amsfonts,amssymb}
\RequirePackage[numbers,sort&compress]{natbib}
\RequirePackage[colorlinks,citecolor=blue,urlcolor=blue]{hyperref}
\RequirePackage{graphicx}

\usepackage[T1]{fontenc}
\usepackage{calc}
\usepackage{gensymb}
\usepackage{mathrsfs}  
\usepackage{bigints}
\usepackage{indentfirst}
\usepackage{enumerate}
\usepackage{hyperref}
\usepackage{mathtools}

\startlocaldefs
\theoremstyle{plain}

\newtheorem{theorem}{Theorem}[section]
\newtheorem{lemma}[theorem]{Lemma}
\theoremstyle{definition}

\theoremstyle{plain}
\newtheorem{prop}[theorem]{Proposition} 
\newtheorem{cor}[theorem]{Corollary} 

\theoremstyle{definition}
\newtheorem{defi}[theorem]{Definition} 
\newtheorem{exe}[theorem]{Example} 
\newtheorem{rem}[theorem]{Remark}

\DeclareMathOperator{\tr}{Tr}
\DeclareMathOperator{\Nb}{Nb}
\DeclareMathOperator{\ts}{tr}
\DeclareMathOperator{\cov}{Cov}
\DeclareMathOperator{\var}{Var}

\DeclareMathOperator{\supp}{Supp}
\DeclareMathOperator{\dep}{depth}
\DeclareMathOperator{\id}{id}
\DeclareMathOperator{\modulo}{mod}
\let\limsup\relax
\DeclareMathOperator*{\limsup}{limsup}
\let\liminf\relax
\DeclareMathOperator*{\liminf}{liminf}
\newcommand{\norm}[1]{\left\Vert #1\right\Vert}

\renewcommand\P{\mathbb{P}}
\newcommand\E{\mathbb{E}}
\newcommand\N{\mathbb{N}}
\newcommand\Z{\mathbb{Z}}
\newcommand\1{\mathbf{1}}
\newcommand\C{\mathbb{C}}
\newcommand\CC{\mathcal{C}}
\newcommand\D{\mathcal{D}}
\newcommand\M{\mathbb{M}}
\newcommand\R{\mathbb{R}}
\newcommand\U{\mathcal{U}}
\newcommand\V{\mathcal{V}}
\newcommand\X{\mathcal{X}}
\newcommand\A{\mathcal{A}}
\newcommand\B{\mathcal{B}}

\newcommand\F{\mathcal{F}}
\newcommand\G{\mathcal{G}}
\renewcommand\i{\mathbf{i}}

\newcommand\PP{\mathcal{A}}

\numberwithin{equation}{section}

\endlocaldefs

\begin{document}

\begin{frontmatter}
\title{Asymptotic expansion of smooth functions in deterministic and iid Haar unitary matrices, and application to tensor products of matrices}
\runtitle{Asymptotic expansions and Haar unitary matrices}

\begin{aug}
\author[A]{\fnms{F\'elix}~\snm{Parraud} \ead[label=e1]{felix.parraud@gmail.com}},

\address[A]{KTH Royal Institute of Technology \printead[presep={ ,\ }]{e1}}
\end{aug}

\begin{abstract}
Let $U^N$ be a family of $N\times N$ independent Haar unitary random matrices and their adjoints, $Z^N$ a family of deterministic matrices, and $P$ a self-adjoint noncommutative polynomial, i.e. for any $N$, $P(U^N,Z^N)$ is self-adjoint, $f$ a smooth function. We prove that for any $k$, if $f$ is smooth enough, there exist deterministic constants $\alpha_i^P(f,Z^N)$ such that
$$ \mathbb{E}\left[\frac{1}{N}\text{Tr}\left( f(P(U^N,Z^N)) \right)\right]\ =\ \sum_{i=0}^k \frac{\alpha_i^P(f,Z^N)}{N^{2i}}\ +\ \mathcal{O}(N^{-2k-2}) .$$
Besides, the constants $\alpha_i^P(f,Z^N)$ are built explicitly with the help of free probability. As a corollary, we prove that given $\alpha<1/2$, for $N$ large enough, every eigenvalue of $P(U^N,Z^N)$ is $N^{-\alpha}$-close to the spectrum of $P(u,Z^N)$ where $u$ is a $d$-tuple of free Haar unitaries. We also prove the convergence of the norm of any polynomial $P(U^N\otimes I_M, I_N\otimes Y^M)$ as long as the family $Y^M$ converges strongly and that $M\ll N \ln^{-3}(N)$.
\end{abstract}



\end{frontmatter}
\tableofcontents

\section{Introduction}

Asymptotic expansions has a long history in Random Matrix Theory. The first result of this kind was obtained by Harer and Zagier in \cite{harerzag} in 1986. They proved that the expectation of the moments of a Gaussian matrix of size $N$ was a polynomial in the inverse of the dimension whose coefficients are given by enumerating graphs of a certain type and genus. More generally, graph enumeration was used extensively to describe the coefficients of asymptotic expansions of different random matrices. Notably when studying the so-called matrix models which can be viewed as a generalization of Gaussian matrices. For works linking matrix models to statistical models on random graphs, see for example the seminal works of t'Hooft \cite{T_ouf} and Br\'ezin, Parisi, Itzykson and Zuber \cite{parisi}, but also \cite{macl, cherbin, segala1, segala2,segala3}, as well as in \cite{segalaU1,segalaU2,Novak} for the unitary case. This was also extended to the so-called $\beta$-ensembles in \cite{cherbin2,betaens1, borot1,borot2,borot4,borot5}. Among other objects, these works study correlation functions and the so-called free energy and show that they expand as power series in the inverse of the dimension. On a different note, computing precise asymptotic of integral of polynomial in the entries of a random unitary matrix was a recurring problem in theoretical physics. More precisely, this problem arose in the 1970s in physics, see \cite{intiweig}. This gave rise to the Weingarten calculus which allowed to compute those integral for random matrices whose law is the Haar measure on a compact group. This theory has a long history starting with Weingarten in the paper mentioned previously, however, significant progress was made in the last two decades, see notably \cite{weing,alsjxam,dlmossc}. For an introduction to the general theory we refer to \cite{weing}.

In this paper, we study polynomials of independent Haar unitary matrices and deterministic matrices. In particular, when this random matrix is self-adjoint, we give an asymptotic expansion of the trace of any sufficiently smooth functions evaluated in this random matrix. The main difference with the papers previously mentioned is that we consider smooth functions whereas they usually work with polynomials or exponential of polynomials. Until recently, this approach was rarely considered due to the difficulties that come with working with non-analytic functions, although there are some previous results, see \cite{macl} and \cite{precurso}. More recently though, in \cite{un} we introduced a new approach which consists in interpolating our random matrices with free operators. This approach was refined in \cite{trois} where we proved an asymptotic expansion for polynomials in independent GUE matrices and deterministic matrices. In \cite{deux}, we used the heuristics of \cite{un} to study polynomials of independent Haar unitary matrices and deterministic matrices. Thus by combining the different tools used in those papers, we prove an asymptotic expansion in the unitary case.

One of the main motivations to prove such an expansion with non-analytic functions is to study the spectrum of polynomials of independent Haar unitary matrices and deterministic matrices. The case of a single Haar unitary matrix is well-known: we even have an explicit formula for the joint law of the eigenvalues, see Proposition 4.1.6 of \cite{alice}. However, there exists no such result for general polynomials. In order to explain how to tackle the multivariable case, let us introduce some notations. Given $A_N$ a self-adjoint matrix of size $N$, one defines the empirical measure of its (real) eigenvalues by 
$$ \mu_{A_N} = \frac{1}{N} \sum_{ i=1}^N \delta_{\lambda_i},$$
\noindent where $\delta_{\lambda}$ is the Dirac mass in $\lambda$ and $\lambda_1,\dots ,\lambda_N$ are the eigenvalue of $A_N$. Besides, for any functions $f$,
$$ \int f\ d\mu_{A_N} = \frac{1}{N} \tr_N\left(f(A_N)\right).$$
In \cite{Vo98}, Voiculescu proved that almost surely the trace of any polynomials of independent Haar unitary matrices converges. This result was in the continuity of his seminal paper \cite{Vo91} where he proved similar results for GUE matrices. Hence he deduced the convergence in law of any empirical measure associated to a self-adjoint polynomial, i.e. such that for any $N$, $U_1^N,\dots,U_d^N$ unitary matrices, $P(U_1^N,\dots,U_d^N, \allowbreak  {U_1^N}^*, \dots,{U_d^N}^*)$ is self-adjoint. Besides, the limit measure $\mu_P$ is defined with the help of free probability. Consequently, assuming we can apply the Portmanteau theorem, the proportion of eigenvalues of $A_N = P(X_1^N,\dots,X_d^N)$ in the interval $[a,b]$, that is $\mu_{A_N}([a,b])$, converges towards $\mu_P([a,b])$. However, Voiculescu's work does not allow us to quantify the speed of the convergence. It also does not prove or disprove the existence of outliers, i.e. eigenvalues of $A_N$ which are not close from the support of the limiting measure. To deal with those question, we consider the following inequality. Let $f$ be a non-negative function such that $f$ is equal to $1$ on the interval $[a,b]$, then if $\sigma(A_N)$ is the spectrum of $A_N$,
$$ \P\Big( \sigma(A_N)\cap [a,b] \neq \emptyset \Big) \leq \P\Big( \tr_{N}\left( f(A_N) \right)\geq 1 \Big) \leq \E\Big[ \tr_{N}\left( f(A_N) \right) \Big] .$$

\noindent Thus if one can show that the right-hand side of this inequality converges towards zero when $N$ goes to infinity, then asymptotically there is no eigenvalue in the segment $[a,b]$. We did so in \cite{deux} where we showed that given a smooth function $f$, there is a constant $\alpha_0^P(f)$, which can be computed explicitly with the help of free probability, such that
\begin{equation}
	\label{socosmocds}
	\E\Big[ \frac{1}{N} \tr_{N}\left( f\left(P\left(U_1^N,\dots,U_d^N, \allowbreak  {U_1^N}^*, \dots,{U_d^N}^*\right)\right) \right) \Big] = \alpha_0^P(f) + \mathcal{O}(N^{-2}) .
\end{equation}
Note that Collins and Male had previously found a strategy in \cite{collins_male} to study the outliers which does not rely on proving Equation \eqref{socosmocds} by using results from Haagerup and Thorbj\o rnsen in \cite{HT}. More precisely, Collins and Male proved that for $P$ a self-adjoint polynomial, almost surely, for any $\varepsilon>0$ and $N$ large enough, 
\begin{equation}
	\label{3spec}
	\sigma\left( P\left(U_1^N,\dots,U_d^N, \allowbreak  {U_1^N}^*, \dots,{U_d^N}^*\right) \right) \subset \supp \mu_P + (-\varepsilon,\varepsilon) ,
\end{equation}

\noindent where $\supp \mu_P$ is the support of the measure $\mu_P$. Given the important consequences that studying the first two orders had in Equation \eqref{socosmocds}, one can wonder what happens at the next order. In this paper, we prove that this expectation has a finite order Taylor expansion, i.e. that for any $k$, if $f$ is smooth enough, there exist deterministic constants $\alpha_i^P(f)$ such that
$$ \mathbb{E}\left[\frac{1}{N}\tr_N\left( f\left(P\left(U_1^N,\dots,U_d^N, \allowbreak  {U_1^N}^*, \dots,{U_d^N}^*\right)\right) \right)\right]\ =\ \sum_{i=0}^k \frac{\alpha_i^P(f)}{N^{2i}}\ +\ \mathcal{O}(N^{-2k-2}).$$

As previously mentioned, up until recently, all the results on asymptotic expansion for non-analytic functions can be summed up in the paper \cite{macl} of Ercolani and McLaughlin, as well as \cite{precurso} from Haagerup and Thorbj\o rnsen. However, the proofs rely on the explicit formula for the law of the eigenvalues of the random matrix considered. Since there exist no such formula for the eigenvalues of a polynomial in independent Haar unitary matrices, we cannot adapt this proof. Instead we rely on the strategy developed in \cite{un,deux,trois}. The main idea is to interpolate Haar unitary matrices and free Haar unitaries with the help of a free unitary Brownian motion. This object can be seen as the large $N$ limit of the Haar unitary Brownian motion. We refer to \cite{bianebr,freebrlaw} for the construction of the free unitary Brownian motion, \cite{ref1} for its use in free probability, and \cite{ref2} for its link with its matrix counterpart. The main tool to do so is the free stochastic calculus, although in this paper we rely on Proposition 3.3 of \cite{deux} to circumvent most of those computations. Once they are done, we are left with Equation \eqref{compderivat}, which is strongly reminiscent of the Schwinger-Dyson Equation for the unitary group. The relationship between this type of equations and asymptotic expansions has a long history in Random Matrix Theory. We refer to \cite{ADGuion} for a very complete introduction. Then we use the invariance of the Haar measure under the group operation, which is a staple in most computations involving the Haar measure, notably in the field of Weingarten calculus. After computing and carefully estimating the remainder term, this yields the following theorem.

\begin{theorem}
	\label{3lessopti}
	Given the following objects,
	\begin{itemize}
		\item $U^N = (U_1^N,\dots,U_d^N)$ independent Haar unitary matrices of size $N$,
		\item $Z^N = (Z_1^N,\dots,Z_q^N, {Z_1^N}^*,\dots,{Z_q^N}^*)$ deterministic matrices of size $N$ and their adjoints,
		\item $P$ a self-adjoint polynomial that can be written as a linear combination of $\mathbf{m}$ monomials of degree at most $n$ and coefficients with an absolute value of at most $c_{\max}$,
		\item $f:\R\mapsto\R$ a function of class $\mathcal{C}^{4k+7}$. We define $\norm{f}_{\mathcal{C}^i}$ the sum of the supremum on $\R$ of the first $i$-th derivatives of $f$.
		
	\end{itemize}
	
	\noindent Then there exist deterministic coefficients $(\alpha_i^P(f,Z^N))_{0\leq i\leq k}$ and a constants $C$ independent of $P,f,N$ or $k$, such that with $K_N = \max \{ \norm{Z^N_1}, \dots, \norm{Z^N_q}, 1\}$, $C_{\max}(P) = \max \{1, c_{\max}\}$, for any $N$ and $k$,
	\begin{align}
		\label{3mainresu0}
		&\left| \E\left[ \frac{1}{N}\tr_N\Big(f(P(U^N,{U^N}^*,Z^N))\Big)\right] - \sum_{0\leq i\leq k} \frac{1}{N^{2i}} \alpha_i^P(f,Z^N) \right| \\
		&\leq \frac{1}{N^{2k+2}} \norm{f}_{\mathcal{C}^{4k+7}} \times \Big(C\times K_N^{n+1} C_{\max} \mathbf{m}\times n(n+1)\Big)^{4k+6}\times k^{15k} . \nonumber
	\end{align}
	
	\noindent Moreover for any $i$,
	\begin{equation}
		\label{3mainresu02}
		\left| \alpha_i^P(f,Z^N) \right| \leq \norm{f}_{\mathcal{C}^{4i+3}} \times \Big(C\times K_N^{n+1} C_{\max} \mathbf{m}\times n(n+1) \Big)^{4i+2}\times i^{15i} .
	\end{equation}
	Finally, if $f$ and $g$ are functions of class $\mathcal{C}^{4k+7}$ equal on a neighborhood of the spectrum of $P(u,u^*,Z^N)$, where $u$ is a $d$-tuple of free Haar unitaries free from $\M_N(\C)$, then for any $i\leq k$, $\alpha_i^P(f,Z^N) = \alpha_i^P(g,Z^N)$. In particular if the support of $f$ and the spectrum of $P(u,u^*,Z^N)$ are disjoint, then for any $i\leq k$, $\alpha_i^P(f,Z^N)=0$.
	
\end{theorem}

This theorem should be compared with Theorem 1.1 of \cite{trois} which proves a similar result but with GUE matrices instead of Haar unitary matrices. This is not entirely unexpected since there are some links between those two type of random matrices, indeed, the law of the matrix whose columns are the $N$ eigenvectors of a GUE random matrix of size $N$ is the one of a Haar unitary matrix. However, it is still quite surprising how close the formulas are. Notably if we compare Theorem \ref{3TTheo} of this paper and Theorem 3.4 of \cite{trois}, which are respectively slightly more general version of Theorem \ref{3lessopti} of this paper and Theorem 1.1 of \cite{trois}, then the explicit formulas that they give for the coefficients $\alpha_i^P(f,Z^N)$ almost have the same definition with the only major difference being that we do not use the same interpolation process between our random matrices and the free operators. That being said, while the heuristic of the proofs have some similarities, the proof themselves do not have much in common. Notably, the proof of Proposition \ref{3intercoef} which ensures that the coefficients $\alpha_i^P(f,Z^N)$ are well-defined is a major difficulty of this paper which we did not have to deal with in the Hermitian case. 

That being said, the above theorem calls for a few remarks.

\begin{itemize}
	\item  In Theorem \ref{3lessopti}, we only considered a single function $f$ evaluated in a self-adjoint polynomial $P$. However, one could easily adapt the proof of Theorem \ref{3TTheo} to consider a product of functions $f_i$ evaluated in self-adjoint polynomials $P_i$ and get a similar result. The main difference would be that instead of $\norm{f}_{\mathcal{C}^{4k+7}}$ one would have $ \max_i \norm{f_i}_{\mathcal{C}^{4k+7}}$. One could also adapt the proof to deal with the case of a product of traces. We give more details about those two situations in Remark \ref{produf}.
	\item Thanks to Proposition \ref{3intercoef}, by taking $Z^N=(E_{i,j})_{1\leq i,j\leq N}$ where $E_{i,j}$ is the matrix whose coefficient $(i,j)$ is equal to $1$ and every other coefficient is equal to $0$, one can compute the expectation of any product of entry of Haar unitary matrices as a power series in $N^{-2}$. Hence giving a solution to the original problem that led to the emergence of the Weingarten calculus which was to compute such integrals. The formula for the coefficients of the power series obtained with our method is different from the one given by the Weingarten calculus and it would be interesting to further compare them.
	\item The coefficients $(\alpha_i^P(f,Z^N))_{1\leq i\leq k}$ are continuous with respect to all of their parameters, $f,Z^N$ and $P$. We give a precise statement in Corollary \ref{continucoeff}. In particular if $Z^N$ converges in distribution when $N$ goes to infinity (as defined in Definition \ref{3freeprob}) towards a family of noncommutative random variables $Z$, then for every $i$, $\alpha_i^P(f,Z^N)$ converges towards $\alpha_i^P(f,Z)$.
	\item We assumed that the matrices $Z^{N}$ are deterministic, but thanks to Fubini's theorem we can assume that they are random matrices as long as they are independent from $U^N$. In this situation though, $K_N$ in the right side of the inequality is a random variable (and thus we need some additional assumptions if we want its expectation to be finite for instance).
	\item Since the probability that there is an eigenvalue of $P(U^N,{U^N}^*,Z^N)$ outside of a neighborhood of $P(u,u^*,Z^N)$ is exponentially small as $N$ goes to infinity, the smoothness assumption on $f$ only needs to be verified on a neighborhood of $P(u,u^*,Z^N)$ for such an asymptotic expansion to exist.
\end{itemize}

As we said earlier, by studying the trace of a smooth function evaluated in a random matrix, one can study the asymptotic behavior of the spectrum. In their seminal paper \cite{HT} in 2005, Haagerup and Thorbj\o rnsen were the first one to study the case of polynomials in independent random matrices. By doing so, they introduced the notion of strong convergence (see Definition \ref{3freeprob}). For a detailed history of this type of results, we refer to the introduction of \cite{un}. In 2012, Collins and Male used those results to prove that the spectrum of $P(U^N,{U^N}^*,Z^N)$ converges for the Hausdorff distance towards an explicit subset of $\R$. We summarized this result in Equation \eqref{3spec}. However, the tools used in this proof did not yield quantitative estimates. On the contrary, by using the finite order Taylor expansion with $f: x\to g(N^{\alpha}x)$ where $g$ is a well-chosen smooth function, one can show the following proposition.

\begin{cor}
	\label{3voisinage}
	Let $U^N$ be independent Haar unitary matrices of size $N$, $Z^N = ( Z_1^N, \dots, Z_q^N, {Z_1^N}^*, \allowbreak \dots, {Z_q^N}^*) $ a family of deterministic matrices whose norm is uniformly bounded over $N$ and their adjoints, $u$ a family of free Haar unitaries and $P$ a self-adjoint polynomial. Given $\alpha< 1/2$, almost surely for $N$ large enough, 
	$$ \sigma\left( P(U^N,{U^N}^*,Z^N) \right) \subset \sigma\left( P(u,u^*,Z^N) \right) + (-N^{-\alpha},N^{-\alpha}) ,$$
	where $\sigma(X)$ is the spectrum of $X$, and $u$ is free from $\M_N(\C)$.
\end{cor}

Recently, there has been rising interest for the following question. If instead of considering a polynomial with scalar coefficients, we take a polynomial with matrix coefficients, how does the norm behave? More precisely, we consider $X_i^N\otimes I_M$, random matrices tensorized with the identity matrix of size $M$, as well as $I_N\otimes Y_j^M$ where $Y_j^M$ are deterministic matrices of size $M$, then does the norm of a polynomial in those matrices converge? The case where $M$ is constant is always true as long as it is true for $M$ equal to $1$, see Proposition 7.3 of \cite{male}. However, if we let $M$ fluctuate with $N$, then the answer is much less straightforward. In the case where every $X_i^N$ is a GUE random matrix, then the convergence of the norm was proved for $M\ll N^{1/4}$ in \cite{slmcls}, it was improved to $M\ll N^{1/3}$ in \cite{un}, and to $M\ll N/\ln^3(N)$ in \cite{tensorsc}. Those results were motivated by \cite{petom} a paper of Ben Hayes which proved that the strong convergence of the family $(X_i^N\otimes I_N, I_N\otimes Y_j^N)_{i,j}$ when $(Y_j^N)_j$ are also independent GUE random matrices of size $N$ implies some important result on the structure of certain finite von Neumann algebras, the so-called Peterson-Thom conjecture. This was proved in \cite{bitid}. However, if we assume that the matrices $Y_j^M$ are deterministic, then it is still unknown how large one can assume $M$ to be with respect to $N$. If we assume that $X_i^N$ are Haar unitary matrices, then we proved in \cite{deux} that one had to assume that $M\ll N^{1/3}/\ln^{2/3}(N)$. In the following theorem we improve this bound to $M\ll N/\ln^{5/2}(N)$.

\begin{cor}
	\label{3boundednormrenm}
	Given the following objects,
	\begin{itemize}
		\item $U^N = (U_1^N,\dots,U_d^N)$ independent Haar unitary matrices of size $N$,
		\item $u = (u_1,\dots,u_d)$ free Haar unitaries,
		\item $Z^N = (Z_1^N,\dots,Z_r^N, {Z_1^N}^*,\dots,{Z_r^N}^*)$ deterministic matrices of size $N$ and their adjoints,
		\item $Y^M = (Y_1^M,\dots,Y_s^M, {Y_1^M}^*,\dots,{Y_s^M}^*)$ deterministic matrices of size $M$ and their adjoints,
		\item $X^{N,M}= (U^N\otimes I_M,{U^N}^*\otimes I_M,Z^N\otimes I_M,I_N\otimes Y^M)$,
		\item $x^{N,M}= (u\otimes I_M,u^*\otimes I_M,Z^N\otimes I_M,I_N\otimes Y^M)$,
		\item $P$ a non-commutative polynomial.		
	\end{itemize}
	If we assume that the families $Z^N$ and $Y^M$ are uniformly bounded over $N$ and $M$ for the operator norm, then there exists a constant $C_P$ such that for any $\delta>0$,	
	\begin{align*}
		&\P\Bigg(\norm{P\left( X^{N,M}\right)} \geq  \norm{P\left( x^{N,M} \right)} + \delta + C_P\left(\frac{M}{N}\right)^{1/2}\ln(NM)^{3/2} \left(1+\frac{1}{\norm{P\left( x^{N,M}\right)}^2}\right) \Bigg) \\
		&\leq  e^{-K_P \delta^2 (N-2)}.
	\end{align*}
	Moreover, if $M\ll N/\ln^{3}(N)$ and that the family $Y^M$ converges strongly in distribution towards a family of non-commutative variable $y$, then the family $(U^N\otimes I_M,I_N\otimes Y^M)$ also converges strongly towards $(u\otimes 1, 1\otimes y)$.
\end{cor}

Note that since this paper was first released, several papers have strongly improved the dimension $M$ of the matrices that one could consider in the second component of the tensor. In \cite{bordocol}, Bordenave and Collins proved that one could consider $M$ much larger than $N$, with a cut-off at $M = \exp(N^{\alpha})$ with $\alpha = (32d + 160)^{-1}$. Very recently, in \cite{rep1}, the authors developed a new approach to prove strong convergence results, and although in this paper, they only consider matrix coefficients of size $o(N)$, with a refinement of their method, and with the help of the asymptotic expansion proved in this paper (more precisely, Proposition \ref{3intercoef}), it was proved in \cite{mage}  as a corollary of the main result that one could take matrix coefficients of size $\exp(N^{1/2} (\log N)^{-4})$, which was improved to $\exp(o(N))$ in \cite{rep2}. In between those two papers, we also proved in \cite{neuf} with different methods that one could consider the case $M= \exp(o(N^{2/3}))$. All of these results imply the Peterson-Thom conjecture thanks to the paper of Ben Hayes \cite{petom}. Finally, note that Pisier found a counter-example in \cite{slmcls} for $M$ of order $\exp(C N^2)$ with $C$ a constant.

Besides, this corollary as well as the previous one shows that the fluctuations of the largest eigenvalue are at most of size $N^{-1/2+o(1)}$. We expect this to be optimal when considering polynomials in both random and deterministic matrices, as in the case of the BBP transition, see Theorem 1.2 of \cite{peche}. However, when only considering polynomials of random matrices, it would be possible for the fluctuations of the largest eigenvalue to be of order $N^{-2/3+o(1)}$, see for example \cite{nemish} where the authors showed that the largest eigenvalue of most polynomials in Wigner matrices of degree two have fluctuations of this size. Optimally, one could even hope that once rescaled, the largest eigenvalue would converge to the Tracy-Widom law, see \cite{tracy} the paper of Tracy and Widom for a definition and a proof in the case of a GUE random matrix.

Finally, by using Lemma \ref{3apparition}, which is a key lemma of the proof of Theorem \ref{3lessopti}, we prove the asymptotic freeness of deterministic matrices conjugated by certain random matrices generated with a polynomial of Haar unitary matrices. This corollary is similar to Theorem 1.2 of \cite{FK} and Corollary 2.12 of \cite{thermalidsl}, with the major difference that since we work with Haar unitary matrices instead of Wigner matrices, our concentration estimates are much easier to prove, hence the proof is considerably shorter.

\begin{cor}
	\label{units}
	Let $A^N=(A_1^N,\dots,A_q^N)$ be deterministic matrices and $U^N= (U_1^N,\dots,U_d^N)$ be Haar unitary matrices. Moreover we assume that for every $i$, $A_i^N$ converges in distribution towards a non-commutative random variable $a_i$ (see Definition~\ref{3freeprob}). Further, let  $y_1^N\leq\dots\leq y_k^N\in\R$ be such that for any $i<k$, $1 \ll y_{i+1}^N-y_i^N \ll N^{1/2}\ln(N)^{-1/2}$, and let $P$ be a non-constant self-adjoint non-commutative polynomial in $d$ variables. Then with 
	$$a_i^N:= e^{\i y_i^N P(U^N)} A_i^N e^{-\i y_i^N P(U^N)},$$ 
	almost surely the family of non-commutative random variables $a^N=(a_1^N,\dots,a_k^N)$ converges jointly in distribution towards $a=(a_1,\dots,a_k)$ where the $(a_i)$ are free.
\end{cor}

The paper is organized as follows, in Section \ref{3definit} we introduce basic definitions of Free Probability and Random Matrix Theory, in Section \ref{prelimw} we prove some propositions that we will use repeatedly in the rest of the paper. Section \ref{3mainsec} is dedicated to the proof of Theorem \ref{3lessopti} and \ref{3TTheo}, notably by first proving Lemma \ref{3imp2}, which gives a first rough formulation of the coefficients. Finally, in Section \ref{conswlf} we prove the different corollaries.

\section*{List of notations}

We finish the introduction with a list of every important notation that will be used in the rest of the paper.

\begin{itemize}
	\item $\norm{\cdot}$: the operator norm.
	\item $x$: system of free semicircular variables, Definition \ref{3freeprob}.
	\item $u$: family of free Haar unitaries, Definition \ref{3freeprob}.
	\item $u_t$: family of free unitary Brownian motion at times $t$, Definition \ref{3freeprob}.
	\item $\A_N$: the free product of $\M_N(\C)$ and the $\CC^*$-algebra generated by a system of free semicircular variables, Definition \ref{3tra}.
	\item $\tr_N$: the non-normalized trace on $\M_N(\C)$, Definition \ref{3tra}.
	\item $\ts_N$: the normalized trace on $\M_N(\C)$, Definition \ref{3tra}.
	\item $E_{r,s}$: the matrix with $1$ in the $(r,s)$ entry and zeros in all the other entries, Definition \ref{3tra}.
	\item $\id_N\otimes \ts_k$: the conditional expectation from $\M_N(\C)\otimes \M_k(\C)$  to $\M_N(\C)$, Definition \ref{3tra}.
	\item $\PP_{d,q}$: the set of noncommutative polynomials in $2(d+q)$ variables, Subsection \ref{3poly}.
	\item $\norm{\cdot}_L$: the norm defined in Equation \eqref{3normA}.
	\item $\#$, $\widetilde{\#}$, $m$: operators defined in Equation \eqref{3defperdu}.
	\item $\delta_i$: noncommutative derivative on $\A_{d,q}$, Definition \ref{3application}.
	\item $\D_i$: cyclic derivative on $\A_{d,q}$, Definition \ref{3application}.
	\item $\F_{d,q}$: the set of noncommutative polynomials in $2(d+q)$ variables and exponentials of those polynomials, Definition \ref{skjdncksn}.
	\item $\otimes_{\text{min}}$: minimal tensor product, Definition \ref{3mini}.
	\item $\delta_{\alpha,i}$: noncommutative derivative on $\F_{d,q}$, Definition \ref{3technicality}.
	\item $\boxtimes$: operator defined in Definition \ref{3sweedler}.
	\item $	F_{n}^{j,1},\widetilde{F}_{n}^{j,1},F_{n}^{j,2},\widetilde{F}_{n}^{j,2}$: functions on sets of integer, Definition \ref{3biz}.
	\item $J_n$: collection of sets of integers defined by induction in \eqref{sldjvncskn}.
	\item $\A_{d,q}^n, \F_{d,q}^n, \A_{d,q}^{i_0,\dots,i_{n-1}}, \F_{d,q}^{i_0,\dots,i_{n-1}}, \G_{d,q}^n$: Spaces of polynomials and their exponentials defined in \ref{sidcosc}.
	\item $\delta_{i,I}, \D_{i,I}$: noncommutative and cyclic derivative on $\A_{d,q}^n$, Definition \ref{3biz2}.
	\item $\delta_{i,I,\alpha}, \D_{i,I,\alpha}$: noncommutative and cyclic derivative on $\F_{d,q}^n$ or $\G_{d,q}^n$, Definition \ref{3biz2}.
	\item $\dep^n(s)$: position of the integer $s$ in elements of $J_n$, Lemma \ref{3detail}.
	\item $U^N$: family of independent Haar unitary matrices, Definition \ref{2Haardef}.
	\item $U^N_t$: family of independent unitary Brownian motions, Definition \ref{2UBdef}.

\end{itemize}

\section{Framework and standard properties}
\label{3definit}

\subsection{Usual definitions in free probability}
\label{3deffree}

In order to be self-contained, we begin by recalling the following definitions from free probability.

\begin{defi}~
	\label{3freeprob}
	\begin{itemize}
		\item A \textbf{$\mathcal{C}^*$-probability space} $(\A,*,\tau,\norm{.}) $ is a unital $\mathcal{C}^*$-algebra $ (\A,*,\norm{.})$ endowed with a \textbf{state} $\tau$, i.e. a linear map $\tau : \A \to \C$ satisfying $\tau(1_{\A})=1$ and $\tau(a^*a)\geq 0$ for all $a\in \A$. In this paper we always assume that $\tau$ is a \textbf{trace}, i.e. that it satisfies $\tau(ab) = \tau(ba) $ for any $a,b\in\A$. An element of $\A$ is called a 
		\textbf{noncommutative random variable}. We will always work with a faithful trace, namely, for $a\in\A$, $\tau(a^*a)=0$ if and only if $a=0$.
		
		\item Let $\A_1,\dots,\A_n$ be unital $*$-subalgebras of $\A$. They are said to be \textbf{free} if for all $k$, for all $a_i\in\A_{j_i}$ such that $j_1\neq j_2$, $j_2\neq j_3$, \dots , $j_{k-1}\neq j_k$:
		\begin{equation}
			\label{kddkdxkfl}
			\tau\Big( (a_1-\tau(a_1))(a_2-\tau(a_2))\dots (a_k-\tau(a_k)) \Big) = 0.
		\end{equation}
		Families of noncommutative random variables are said to be free if the $*$-subalgebras they generate are free.
		
		\item Let $ A = (a_1,\ldots ,a_k)$ be a $k$-tuple of random variables. The \textbf{joint $*$-distribution} of the family $A$ is the linear form $\mu_A : P \mapsto \tau\big[ P(A, A^*) \big]$ on the set of polynomials in $2k$ noncommutative variables. By \textbf{convergence in distribution}, for a sequence of families of variables $(A_N)_{N\geq 1} = (a_{1}^{N},\ldots ,a_{k}^{N})_{N\geq 1}$ 
		in $\mathcal C^*$-algebras $\big( \mathcal A_N, ^*, \tau_N, \norm{.} \big)$,
		we mean the pointwise convergence of
		the map 
		$$ \mu_{A_N}: P \mapsto \tau_N \big[ P(A_N, A_N^*) \big],$$
		and by \textbf{strong convergence in distribution}, we mean convergence in distribution, and pointwise convergence
		of the map
		$$ P \mapsto \big\| P(A_N, A_N^*) \big\|.$$
		
		\item A family of noncommutative random variables $ x=(x_1,\dots ,x_d)$ is called a \textbf{free semicircular system} when the noncommutative random variables are free, self-adjoint ($x_i=x_i^*$), and for all $k$ in $\N$ and $i$, one has
		\begin{equation*}
			\tau( x_i^k) =  \int_{\R} t^k d\sigma(t),
		\end{equation*}
		with $ d\sigma(t) = \frac 1 {2\pi} \sqrt{4-t^2} \ \mathbf 1_{|t|\leq2} \ dt$ the semicircle distribution.
		
		\item A noncommutative random variable $u$ is called a Haar unitary if it is a unitary, i.e. $u^*u=uu^*=1_{\A}$, and for all $k$ in $\Z$, one has
		$$
		\tau( u^k) = \left\{
		\begin{array}{ll}
			1 & \mbox{if } n=0 \\
			0 & \mbox{else.}
		\end{array}
		\right.
		$$
		
		\item We refer to subsection 2.3 of \cite{deux} for notions of free stochastic calculus, and notably for defining $\# dS_s$ the integral with respect to a free Brownian motion. However, it is not necessary to understand this theory to read this paper. Indeed, for the sake of completeness we define the free unitary Brownian motion below, but we will not use this definition directly in this paper. Let $(S_t)_{t\geq 0}$ be a free Brownian motion adapted to a filtered $W^*$-probability space $(\A,(\A_t)_{t\geq 0},\tau)$, the free unitary Brownian motion $(u_t)_{t\geq 0}$ is the unique solution to the equation
		\begin{equation}
			\label{2defunbro}
			\forall t\geq 0, \quad  u_t = 1_{\A} - \int_0^t \frac{u_s}{2}\ ds + \i \int_0^t (u_s\otimes 1_{\A}) \# dS_s .
		\end{equation}
		In particular, for any $t\geq 0$, $u_t$ is unitary, that is $u_t u_t^* = u_t^* u_t = 1_{\A}$. 
		
	\end{itemize}
	
\end{defi}

It is important to note that thanks to \cite[Theorem 7.9]{nica_speicher_2006}, which we recall below, one can consider free copies of any noncommutative random variable.

\begin{theorem}
	\label{3freesum}
	
	Let $(\A_i,\phi_i)_{i\in I}$ be a family of $\mathcal{C}^*$-probability spaces such that the functionals $\phi_i : \A_i\to\C$, $i\in I$, are faithful traces. Then there exist a $\mathcal{C}^*$-probability space $(\A,\phi)$ with $\phi$ a faithful trace, and a family of norm-preserving unital $*$-homomorphism $W_i: \A_i\to\A$, $i\in I$, such that:
	
	\begin{itemize}
		\item $\phi \circ W_i = \phi_i$, $\forall i \in I$.
		\item The unital $\mathcal{C}^*$-subalgebras $W_i(\A_i)$, $i\in I$, form a free family in $(\A,\phi)$.
	\end{itemize}
\end{theorem}

Let us fix a few notations concerning the spaces and traces that we use in this paper.

\begin{defi}
	\label{3tra}
	~
	\begin{itemize}
		\item $(\A_N,\tau_N)$ is the free product $\M_N(\C) * \mathcal{C}_d$ of $\M_N(\C)$ with $\mathcal{C}_d$ the $\CC^*$-algebra generated by a system of $d$ free semicircular variables, that is the $\mathcal{C}^*$-probability space built in Theorem \ref{3freesum}. Note that when restricted to $\M_N(\C)$, $\tau_N$ is just the normalized trace on matrices. The restriction of $\tau_{N}$ to the $\mathcal{C}^*$-algebra generated by the free semicircular system $x$ is denoted by $\tau$. Note that one can view this space as the limit of a matrix space, we refer to \cite[Proposition 3.5]{un}. 
		\item $\tr_N$ is the non-normalized trace on $\M_N(\C)$.
		\item $\ts_N$ is the normalized trace on $\M_N(\C)$.
		\item We denote $E_{r,s}$ the matrix with $1$ in the $(r,s)$ entry and zeros in all the other entries. 
		\item We regularly identify $\M_N(\C)\otimes \M_k(\C)$ with $\M_{kN}(\C)$ through the isomorphism $E_{i,j}\otimes E_{r,s} \mapsto E_{i+rN,j+sN} $, similarly we identify $\tr_N\otimes\tr_k$ with $\tr_{kN}$.
		\item $\id_N\otimes \ts_k$ is the conditional expectation from $\M_N(\C)\otimes \M_k(\C)$  to $\M_N(\C)$. Basically it is the tensor product of the identity map $\id_N:\M_N(\C)\to\M_N(\C)$ and the normalized trace on $\M_k(\C)$.
		\item If $ A^N=(A_1^N,\dots,A_d^N) $ and $B^k=(B_1^k,\dots,B_d^k)$ are two families of random matrices, then we denote $A^N\otimes B^k=(A_1^N\otimes B^k_1,\dots,A_d^N\otimes B^k_d)$. We typically use the notation $X^N\otimes I_k$ for the family $(X^N_1\otimes I_k,\dots,X^N_1\otimes I_k)$.
		\item In the rest of the paper, in order to shorten equations, given a non-commutative polynomial $P$ in $2k$ variables and a family of $k$ non-commutative variables $X$, we will simply write $P(X)$ instead of $P(X,X^*)$ where $X^*$ is the family of the adjoints of $X$.
	\end{itemize}
\end{defi}

\subsection{Noncommutative polynomials and derivatives}
\label{3poly}

Let $\PP_{d,q}=\C\langle U_1,\dots,U_d,V_1,\dots,V_d,\\ Y_1, \dots, Y_{q},Z_1,\dots,Z_q\rangle$ be the set of noncommutative polynomials in $2(d+q)$ variables. We define an involution $*$ on $\PP_{d,q}$ with $U_i^* = V_i$, $Z_i^* = Y_{i}$, and then we extend it to $\PP_{d,q}$ by linearity and the formula $(\alpha P Q)^* = \overline{\alpha} Q^* P^*$. 

$P\in \PP_{d,q}$ is said to be self-adjoint if $P^* = P$. Self-adjoint polynomials have the property that if $u_1,\dots,u_d,z_1,\dots,z_q$ are elements of a $\mathcal{C}^*$-algebra, then $$P(u_1,\dots,u_d,u_1^*,\dots,u_d^*,z_1,\dots,z_r,z_1^*,\dots,z_r^*)$$ is also self-adjoint as an element of the $\CC^*$-algebra. In order to make the computations less heavy in the paper we will use the following notation when evaluating our polynomials. 

\begin{defi}
	Given $P\in\PP_{d,q}$, and $u=(u_1,\dots,u_d), z=(z_1,\dots,z_q)$ elements of a $\mathcal{C}^*$-algebra, we denote
	$$ P(u,z) = P(u_1,\dots,u_d,u_1^*,\dots,u_d^*,z_1,\dots,z_r,z_1^*,\dots,z_r^*).$$
\end{defi}

Besides, for any fixed $L\in\R_+^*$, one defines

\begin{equation}
	\label{3normA}
	\norm{P}_L = \sum_{M \text{ monomial}} |c_M(P)| L^{\deg M} ,
\end{equation}

\noindent where $c_M(P)$ is the coefficient of $P$ for the monomial $M$ and $\deg M$ the total degree of $M$ (that is the sum of its degree in each letter $ U_1,\dots,U_d,V_1,\dots,V_d,Z_1,\dots,Z_q,Y_1,\dots,Y_{q}$). Let us define several maps which we use frequently in the sequel. First, for $A,B,C\in \PP_{d,q}$, let
\begin{equation}
	\label{3defperdu}
	A\otimes B \# C = ACB,\quad A\otimes B \widetilde{\#} C = BCA,\quad m(A\otimes B) = BA.
\end{equation}

Now let us define the noncommutative derivative, it is a widely used tool in the field of probability, see for example the work of Voiculescu, \cite{refdif} and \cite{refdif2}.

\begin{defi}
	\label{3application}
	
	If $1\leq i\leq d$, one defines the \textbf{noncommutative derivative} $\delta_i: \PP_{d,q} \longrightarrow \PP_{d,q} \otimes \PP_{d,q}$  by its value on a monomial  $M\in \PP_{d,q}$  given by
	$$ \delta_i M = \sum_{M=AU_iB} AU_i\otimes B - \sum_{M=AV_iB} A\otimes V_iB ,$$
	and then extend it by linearity to all polynomials. We can also define $\delta_i$ by induction with the formulas,
	\begin{align}
		\label{3leibniz}
		&\forall P,Q\in \mathcal{A}_{d,q},\quad \delta_i (PQ) = \delta_i P \times \left(1\otimes Q\right)	 + \left(P\otimes 1\right) \times \delta_i Q , \\
		&\forall i,j,\quad \delta_i U_j = \1_{i=j}\ U_j\otimes 1,\quad \delta_i V_j = -\1_{i=j}\ 1\otimes V_j,\quad \delta_i Z_j =\delta_i Y_j =0\otimes 0. \nonumber
	\end{align}
	Similarly, with $m$ as in \eqref{3defperdu}, one defines the \textbf{cyclic derivative}  $\D_i: \PP_{d,q} \longrightarrow \PP_{d,q}$ for $P\in \PP_{d,q}$ by
	$$ \D_i P = m\circ \delta_i P \ . $$
	
\end{defi}

In this paper however, we need to work not only with polynomials but also with more general functions, since we work with the Fourier transform we introduce the following space.

\begin{defi}
	\label{skjdncksn}
	We set
	$$\mathcal{F}_{d,q} = \C\big\langle (E_R)_{R\in\PP_{d,q}}, U_1,\dots,U_d,V_1,\dots,V_d,Y_1, \dots,Y_{q},Z_1,\dots,Z_q \big\rangle.$$
	Then given $u=(u_1,\dots,u_d), z=(z_1,\dots,z_q)$ elements of a $\CC^*$-algebra, one can define by induction the evaluation of an element of $\F_{d,q}$ in $(u,z)$ by following the following rules:
	\begin{itemize}
		\item $\forall Q\in\PP_{d,q}$, $Q(u,z)$ is defined as usual,
		\item $\forall R\in\PP_{d,q}$, $E_R(u,z) = e^{R(u,z)}$,
		\item $\forall Q_1,Q_2\in \F_{d,q}$, 
		$$(Q_1+Q_2)(u,z)= Q_1(u,z)+Q_2(u,z),\quad (Q_1Q_2)(u,z)= Q_1(u,z)Q_2(u,z).$$
		
	\end{itemize}
	One can extend the involution $*$ from $\PP_{d,q}$ to $\mathcal{F}_{d,q}$ by setting $(E_R)^* = E_{R^*} $, and we still have that if $Q\in\F_{d,q}$ is self-adjoint, then so is $Q(u,z)$. Finally, in order to make notations more transparent, we will usually write $e^{ R}$ instead of $E_R$.
\end{defi}
Note that for technical reasons (notably due to Definition \ref{3technicality}) that we explain in Remark 2.10 of \cite{trois}, one cannot view $\mathcal{F}_{d,q}$ as a subalgebra of the set of formal power series in $U_1,\dots,U_d,V_1,\dots,V_d,$ $Y_1, \dots,Y_{q},Z_1,\dots,Z_q$. This is why we need to introduce the notation $E_R$. 

As we will see in Proposition \ref{3duhamel}, the natural way of extending the definition of $\delta_i$ (and $\D_i$) to $\F_{d,q}$ is by setting
\begin{equation}
	\label{3ext}
	\delta_i e^{ Q} = \int_0^1 \big(e^{\alpha Q}\otimes 1\big)\ \delta_i Q\ \big(1\otimes e^{(1-\alpha) Q}\big) d\alpha .
\end{equation}

\noindent However, we cannot define the integral properly on $\mathcal{F}_{d,q}\otimes \mathcal{F}_{d,q}$. After evaluating our polynomials in $\CC^*$-algebras, the integral will be well-defined as we will see. Firstly, we need to define properly the operator norm of tensor of $\CC^*$-algebras. We work with the minimal tensor product also named the spatial tensor product. For more information we refer to \cite[Chapter 6]{murphy}.

\begin{defi}
	\label{3mini}
	Let $\A$ and $\B$ be $\CC^*$-algebra with faithful representations $(H_{\A},\phi_{\A})$ and $(H_{\B},\phi_{\B})$, then if $\otimes_2$ is the tensor product of Hilbert spaces, $\A\otimes_{\min}\B$ is the completion of the image of $\phi_{\A}\otimes\phi_{\B}$ in $B(H_{\A}\otimes_2 H_{\B})$ for the operator norm in this space. This definition is independent of the representations that we fixed.
\end{defi}

In particular, it is important to note that if $\A = \M_N(\C)$, then up to isomorphism $\A \otimes_{\min} \A$ is simply $\M_{N^2}(\C)$ with the usual operator norm. The main reason we pick this topology is for the following lemma. It is mainly a consequence of \cite[Lemma 4.1.8]{ozabr}.
\begin{lemma}
	\label{1faith}
	Let $(\A,\tau_{\A})$ and $(\B,\tau_{\B})$ be $\CC^*$-algebra with faithful traces, then $\tau_{\A}\otimes\tau_{\B}$ extends uniquely to a faithful trace $\tau_{\A}\otimes_{\min}\tau_{\B}$ on $\A\otimes_{\min}\B$. 
\end{lemma}

It is not necessary to understand in depth the minimal tensor product to read the rest of the paper. Indeed, we will not directly make use of this property in this paper, however, it is necessary to introduce it to justify that every object in this paper is well-defined. Thus, we define the noncommutative differential on $\F_{d,q}$ as follows.

\begin{defi}
	\label{3technicality}
	For $\alpha\in [0,1]$, let $\delta_{\alpha,i}: \F_{d,q}\to \F_{d,q}\otimes \F_{d,q}$ which satisfies \eqref{3leibniz} and such that for any $P\in \mathcal{A}_{d,q}$,
	$$ \delta_{\alpha,i} e^{ P} =  \big(e^{ \alpha P}\otimes 1\big)\ \delta_i P\ \big(1\otimes e^{ (1-\alpha) P}\big),\quad \D_{\alpha,i} = m \circ \delta_{\alpha,i} . $$
	Then, given $ z = (z_1,\dots, z_{d+q})$ elements of a $\CC^*$-algebra, we define for any $Q\in \F_{d,q}$,
	$$ \delta_i Q(z) = \int_{0}^1 \delta_{\alpha,i} Q(z)\ d\alpha,\quad \D_i Q(z) = \int_{0}^1 \D_{\alpha,i} Q(z)\ d\alpha . $$
\end{defi}

\noindent Note that for any $P\in\PP_{d,q}$, since $\int_0^11d\alpha = 1$, we do also have that with $\delta_i Q$ defined as in Definition \ref{3application}, 
$$ \delta_i Q(z) = \int_{0}^1 \delta_{\alpha,i} Q(z)\ d\alpha .$$
Thus, Definition \ref{3technicality} indeed extends the definition of $\delta_i$ from $\PP_{d,q}$ to $\F_{d,q}$. Besides, it also means that we can define rigorously the composition of those maps. Since the map $\delta_{\alpha,i}$ goes from $\F_{d,q}$ to $\F_{d,q}\otimes \F_{d,q}$ it is very easy to do so. For example one can define the following operator. We will use a similar one later on.

\begin{defi}
	\label{3operatordef}
	Let $Q\in \F_{d,q}$, given $ z = (z_1,\dots, z_{d+q})$ elements of a $\CC^*$-algebra, let $i,j \in [1,d]$, with $\circ$ the composition of operators we define
	$$ (\delta_j\otimes \delta_j)\circ \delta_i\circ \D_i Q(z) = \int_{[0,1]^4} (\delta_{\alpha_4,j}\otimes \delta_{\alpha_3,j})\circ\delta_{\alpha_2,i}\circ \D_{\alpha_1,i} Q(z)\ d\alpha_1 d\alpha_2 d\alpha_3 d\alpha_4 .$$
\end{defi}

Let us now explain why Equation \eqref{3ext} is natural. If $P\in \PP_{d,q}$, $ z = (z_1,\dots, z_{d+q})$ belongs to a $\CC^*$-algebra $\A$, then we naturally have that
$$(\delta_i P^k) (z) = \sum_{l=1}^k \left(P^{l-1}(z)\otimes 1\right) \delta_i P(z) \left(1\otimes P^{k-l}(z)\right), $$
which is an element of $\A \otimes_{\min} \A$. Besides, there exists a constant $C_P(z)$ independent of $k$ such that $\norm{(\delta_i P^k) (z)} \leq C_P(z) k \norm{P(z)}^{k-1}$. Thus, one can set
\begin{equation}
	\label{3extension}
	(\delta_i e^{P}) (z) = \lim\limits_{n\to\infty} \delta_i\left(\sum_{ 1\leq k \leq n} \frac{P^k}{k!} \right)(z) = \sum_{k\in \N} \frac{1}{k!} (\delta_i P^k) (z) ,
\end{equation}

\noindent as an element of $\A \otimes_{\min} \A$. It turns out that this definition is compatible with Definition \ref{3technicality} thanks to the following proposition (see \cite[Proposition 2.2]{deux} for the proof).

\begin{prop}
	\label{3duhamel}
	Let $P\in \PP_{d,q}$, $ z = (z_1,\dots, z_{d+q})$ elements of a $\CC^*$-algebra $\A$, then with $(\delta_i e^{P}) (z)$ defined as in \eqref{3extension},
	$$ \left(\delta_i e^P\right)(z) = \int_0^1 \left(e^{\alpha P(z)}\otimes 1\right)\ \delta_i P(z)\ \left(1\otimes e^{(1-\alpha)P(z)}\right) \ d\alpha . $$
	
\end{prop}

Finally, for the sake of clarity, we introduce the following notation which is close to Sweedler's convention. Its interest will be clear in Section \ref{3mainsec}.
\begin{defi}
	\label{3sweedler}
	Let $Q\in \F_{d,q}$, $\mathcal{C}$ be a $\mathcal{C}^*$-algebra, $\alpha : \F_{d,q}\to \mathcal{C}$ and $\beta : \F_{d,q}\to \mathcal{C}$ be morphisms. We also set $\mathfrak{m} : A\otimes B \in\mathcal{C}\otimes\mathcal{C}\mapsto AB \in \mathcal{C}$. Then we use the following notation,
	$$ \alpha(\delta_i^1 P) \boxtimes \beta(\delta_i^2 P) = \mathfrak{m}\circ((\alpha\otimes\beta)(\delta_i P)) . $$
\end{defi}

\noindent Heuristically, if $\delta_i P$ was a simple tensor, then $\delta_i^1 P$ would represent the left tensorand while  $\delta_i^2 P$ would represent the right one. However, $\delta_i P$ usually is not a simple tensor and one cannot extend this definition by linearity. This notation is especially useful when our maps $\alpha$ and $\beta$ are simply evaluation of $P$ as it is the case in Section \ref{3mainsec}. Indeed, we will typically write $\delta_i^1P (X) \boxtimes \delta_i^2P (Y)$ rather than first defining $h_X: P\to P(X)$ and using the more cumbersome and abstract notation, $ \mathfrak{m}\circ(h_X\otimes h_Y)(\delta_i P)$. We refer to Example 2.14, 2.15 and  2.16 of \cite{trois} to better understand this notation.

\subsection{Combinatorics and noncommutative derivatives}

Now that we have defined the usual noncommutative polynomial spaces, we build a very specific one which we need to define properly the coefficients of the topological expansion. 

The following definitions are not exactly intuitive, however, those constructions will appear naturally in the rest of the paper. We also refer to Remark 2.19 of \cite{trois} for some intuitions.

\begin{defi}
	\label{3biz}
	Let $(c_n)_n$ be the sequence such that $ c_0 = 0$, $c_{n+1} = 6 c_n +6$. Let $X$ be a set whose elements are all sets of integers of length $2n$. Then we define for $n\geq 0$, $j\in[1,2n]$,
	\begin{align*}
		F_{n+1}^{j,1}(X) &= \Big\{ \{I_1+c_n,\dots,I_{j-1}+c_n,I_j+c_n ,I_{j},\dots,I_{2n},3c_n+1\}   \\
		&\quad\quad\quad\quad\quad\quad\quad\quad\quad\quad\quad\quad\quad\quad\quad\quad\quad\quad\quad\quad\quad \Big|\ I=\{I_1,\dots,I_{2n}\}\in X \Big\}, \\
		F_{n+1}^{2n+1,1}(X) &= \Big\{ \{I_1+c_n,\dots,I_{2n}+c_n,3c_n+2,3c_n+1\}   \ \Big|\ I=\{I_1,\dots,I_{2n}\}\in X \Big\}, \\
		F_{n+1}^{j,2}(X) &= \Big\{ \{I_1+2c_n,\dots,I_{j-1}+2c_n,I_j+2c_n,I_{j},\dots,I_{2n},3c_n+1\} \\
		&\quad\quad\quad\quad\quad\quad\quad\quad\quad\quad\quad\quad\quad\quad\quad\quad\quad\quad\quad\quad\quad \Big|\ I=\{I_1,\dots,I_{2n}\}\in X \Big\}, \\
		F_{n+1}^{2n+1,2}(X) &= \Big\{ \{I_1+2c_n,\dots,I_{2n}+2c_n,3c_n+3,3c_n+1\}   \ \Big|\ I=\{I_1,\dots,I_{2n}\}\in X \Big\}.
	\end{align*}
	\noindent We similarly define $\widetilde{F}_{n+1}^{j,1}(X)$ and $\widetilde{F}_{n+1}^{j,2}(X)$ by adding $3c_n+3$ to every integer in every set. Then we define by induction, $J_0 = \{\emptyset\}$ and
	\begin{equation}
		\label{sldjvncskn}
		J_{n+1} = \bigcup_{1\leq j\leq 2n+1} F_{n+1}^{j,1}(J_n)\cup F_{n+1}^{j,2}(J_n)\cup \widetilde{F}_{n+1}^{j,1}(J_n)\cup \widetilde{F}_{n+1}^{j,2}(J_n).
	\end{equation}
	
	\noindent We then divide $J_{n+1}$ into subsets as follows, given $i_0,\dots, i_{n}$ such that $i_j\in [1,2j+1]$, define $J_{i_0,\dots,i_{n}} \subset J_{n+1}$ inductively by the following equation,
	\begin{equation}
		\label{skvdnsklnd}
		J_{i_0,\dots,i_{n}} = F_{n+1}^{i_n,1}(J_{i_0,\dots,i_{n-1}})\cup F_{n+1}^{i_n,2}(J_{i_0,\dots,i_{n-1}}) \cup \widetilde{F}_{n+1}^{i_n,1}(J_{i_0,\dots,i_{n-1}}) \cup \widetilde{F}_{n+1}^{i_n,2}(J_{i_0,\dots,i_{n-1}}).
	\end{equation}
	For the base case, $i_0$ is necessarily $1$ and we set 
	$$ J_{i_0} = J_1 = F_{1}^{1,1}(\{\emptyset\})\cup F_{1}^{1,2}(\{\emptyset\}) \cup \widetilde{F}_{1}^{1,1}(\{\emptyset\}) \cup \widetilde{F}_{1}^{1,2}(\{\emptyset\}).$$
	
\end{defi}

\begin{defi}
	\label{sidcosc}
	Next we define the following spaces,
	\begin{itemize}
		\item $\A_{d,q}^n = \C\langle U_{i,I},V_{i,I},\ 1\leq i\leq d, I\in J_n,Z_1,\dots,Z_q,Y_1,\dots,Y_q \rangle$,
		\item $\F_{d,q}^n$ as the $*$-algebra generated by $\A_{d,q}^n$ and the family $\left\{e^{Q}\ |\ Q\in \mathcal{A}_{d,q}^n \right\}$,
		\item $\A_{d,q}^{i_0,\dots,i_{n-1}} = \C\langle U_{i,I},V_{i,I},\ 1\leq i\leq d, I\in J_{i_0,\dots,i_{n-1}} ,Z_1,\dots,Z_q,Y_1,\dots,Y_q \rangle$,
		\item $\F_{d,q}^{i_0,\dots,i_{n-1}}$ as the $*$-algebra generated by $\A_{d,q}^{i_0,\dots,i_{n-1}}$ and the family $\left\{e^{Q}\ |\ Q\in \mathcal{A}_{d,q}^{i_0,\dots,i_{n}} \right\}$,
		\item $\G_{d,q}^n$ the vector space generated by $\F_{d,q}^{i_0,\dots,i_{n-1}}$ for every $i_j\in [1,2j+1]$, if $n>0$. We also set $\G_{d,q}^0=\F_{d,q}^0=\F_{d,q}$.
	\end{itemize}
\end{defi}

Note that $\G_{d,q}^n\neq \F_{d,q}^n$ for $n>0$, since $\F_{d,q}^n$ is the $*$-algebra generated by $\A_{d,q}^{i_0,\dots,i_{n-1}}$ for every $i_j\in [1,2j+1]$.

\begin{exe}
	For example one has that
	\begin{align*}
		J_1 &= J_1^{1,1}\cup J_1^{1,2}\cup \widetilde{J}_1^{1,1}\cup \widetilde{J}_1^{1,2} \\
		&= \big\{ \{2,1\}, \{3,1\}, \{5,4\}, \{6,4\} \big\}.
	\end{align*}
	We also have that
	\begin{align*}
		J_2 &= J_2^{1,1}\cup J_2^{2,1}\cup J_2^{3,1} \cup J_2^{1,2} \cup J_2^{2,2} \cup J_2^{3,2} \cup \widetilde{J}_2^{1,1}\cup \widetilde{J}_2^{2,1}\cup \widetilde{J}_2^{3,1} \cup \widetilde{J}_2^{1,2} \cup \widetilde{J}_2^{2,2} \cup \widetilde{J}_2^{3,2}.
	\end{align*}
	It would be too long to list every element in $J_2$. However, here are a few subsets:
	$$ J_2^{1,1} = \big\{ \{8,2,1,19\}, \{9,3,1,19\}, \{11,5,4,19\}, \{12,6,4,19\} \big\}, $$
	$$ \widetilde{J}_2^{3,2} = \big\{ \{35,34,42,40\}, \{36,34,42,40\}, \{38,37,42,40\}, \{39,37,42,40\} \big\}. $$
\end{exe}

\begin{defi}
	\label{3biz2}
	Similarly to Definition \ref{3application}, we define $\delta_{i}$ and $\delta_{i,I}$ on $\A_{d,q}^n$ which satisfies \eqref{3leibniz} and $\forall i,j\in [1,d],\ I,K\in J_n$,
	$$ \delta_{i,I} U_{j,K} = \1_{i=j}\1_{I=K}\  U_{j,K}\otimes 1,\quad \delta_{i,I} V_{j,K} = -\1_{i=j}\1_{I=K}\ 1\otimes V_{j,K}, $$
	$$ \delta_{i} U_{j,K} = \1_{i=j}\  U_{j,K}\otimes 1,\quad \delta_{i} V_{j,K} = -\1_{i=j}\ 1\otimes V_{j,K}. $$
	We then define $\D_{i} = m \circ \delta_{i}$ and $\D_{i,I} = m \circ \delta_{i,I}$. We also define $\delta_{i,\alpha}$ and $\delta_{i,I,\alpha}$ on $\F_{d,q}^n$ and $\G_{d,q}^n$ as in Definition \ref{3technicality}.
\end{defi}

In particular, $\G_{d,q}^0 = \F_{d,q}^0 = \F_{d,q}$ and the two definitions of $\delta_i$ coincide. The following lemma will be important for a better estimation of the remainder term in the expansion.

\begin{lemma}
	\label{3detail}
	Given $s\in [1,c_n]$, there exists a unique $l\in [1,n]$ such that for any $ I =\{I_1,\dots,I_{2n}\}\in J_n$, either $I_l = s$ or $s\notin I$. We refer to $l$ as the depth of $s$ in $J_n$, and will denote it $\dep^n(s)$. Besides, if there exist $i_0,\dots,i_{n-1}$ such that $I,K\in J_{i_0,\dots,i_{n-1}}$ and there exists $l$ such that $I_l=K_l$, then for every $k\geq l$, $I_k=K_k$.
\end{lemma}

\begin{proof}
	The first part of the lemma was already proved in Lemma 2.21 of \cite{trois}. As for the second one let us proceed by induction. If this lemma is true for a given $n$, then let us consider $I,K\in J_{i_0,\dots,i_{n}}$ such that for some $l$,  $s:=I_l=K_l$. If $s\leq 3c_n+3$, then by definition $I,K\in F_{n+1}^{i_n,1}(J_{i_0,\dots,i_{n-1}})\cup F_{n+1}^{i_n,2}(J_{i_0,\dots,i_{n-1}})$, consequently let us start with the easier cases.
	\begin{itemize}
		\item If $s=3c_n+1$, since if $I\in J_n$ then for any $p$, $I_p\leq c_n$, we have that $l=2n+2$, hence clearly for every $k\geq 2n+2$, $I_k=K_k$
		\item If $s=3c_n+2\text{ or }3c_n+3$, then similarly we have that $l=2n+1$ and that $I_{2n+2}=K_{2n+2}=3c_n+3$, hence the conclusion.
	\end{itemize}
	
	\noindent Thus, there remains three possibilities:
	\begin{itemize}
		\item If $s\in[1,c_n]$, then $i_n\leq 2n$ and there exist $\bar{I}\in J_{i_0,\dots,i_{n-1}}$ such that if $I\in F_{n+1}^{i_n,1}(J_{i_0,\dots,i_{n-1}})$
		\begin{align*}
			I = \{\bar{I}_1+c_n,\dots,\bar{I}_{i_n-1}+c_n,\bar{I}_{i_n}+c_n ,\bar{I}_{j},\dots,\bar{I}_{2n},3c_n+1\},
		\end{align*}
		and if $I\in F_{n+1}^{i_n,2}(J_{i_0,\dots,i_{n-1}})$
		\begin{align*}
			I= \{\bar{I}_1+2c_n,\dots,\bar{I}_{i_n-1}+2c_n,\bar{I}_{i_n}+2c_n,\bar{I}_{j},\dots,\bar{I}_{2n},3c_n+1\}.
		\end{align*}
		Hence $l>i_n$ and with $\bar{K}$ defined similarly, we have that $\bar{I}_{l-1} = \bar{K}_{l-1}$, thanks to our induction hypothesis we get that for $k\geq l-1$, $\bar{I}_{k} = \bar{K}_{k}$. Consequently for any $k\geq l$, $I_k=K_k$.
		
		\item If $s\in[c_n+1,2c_n]$, then $I,K\in F_{n+1}^{i_n,1}(J_{i_0,\dots,i_{n-1}})$, let $\bar{I}$ and $\bar{K}$ be defined as previously if $i_n\leq 2n$, and otherwise be such that
		$$ I= \{\bar{I}_1+c_n,\dots,\bar{I}_{2n}+c_n,3c_n+2,3c_n+1\}, $$
		and similarly for $\bar{K}$. Then once again $\bar{I}_{l} = \bar{K}_{l}$ and thanks to our induction hypothesis we get that for $k\geq l$, $\bar{I}_{k} = \bar{K}_{k}$. Hence the conclusion.
		
		\item If $s\in[2c_n+1,3c_n]$, then $I,K\in F_{n+1}^{i_n,2}(J_{i_0,\dots,i_{n-1}})$ and we proceed as previously.
		
	\end{itemize}
	
	\noindent The case where $s>3c_n+3$ is identical with the exception that we add $3c_n+3$ to all of the integers considered and that we work with $\widetilde{F}_{n+1}^{i_n,1}(J_{i_0,\dots,i_{n-1}})$ and $\widetilde{F}_{n+1}^{i_n,2}(J_{i_0,\dots,i_{n-1}})$ instead of $F_{n+1}^{i_n,1}(J_{i_0,\dots,i_{n-1}})$ and $F_{n+1}^{i_n,2}(J_{i_0,\dots,i_{n-1}})$.
\end{proof}



\subsection{Random matrix models}

We conclude this section by giving the definition as well as a few properties of the random matrix models that we will use. 

\begin{defi}
	\label{2Haardef}
	A \textbf{Haar unitary matrix} of size $N$ is a random matrix distributed according to the Haar measure on the group of unitary matrices of size $N$.
\end{defi}

\begin{defi}
	\label{2HBdef}
	A \textbf{Hermitian Brownian motion} $(X_t^N)_{t\in\R^+}$ of size $N$ is a self-adjoint matrix whose coefficients are random variables with the following laws:
	\begin{itemize}
		\item For $1\leq i\leq N$, the random variables $\sqrt{N} ((X^N_t)_{i,i})_{t\in\R^+}$ are independent Brownian motions.
		\item For $1\leq i<j\leq N$, the random variables $(\sqrt{2N}\ \Re{(X^N_t)_{i,j}})_{t\in\R^+}$ and $(\sqrt{2N}\ \Im{(X^N_t)_{i,j}})_{t\in\R^+}$ are independent Brownian motions, independent of $\sqrt{N} ((X^N_t)_{i,i})_{t\in\R^+}$.
	\end{itemize}
\end{defi}

To study the free unitary Brownian motion, we will need to study its finite dimensional equivalent, the unitary Brownian motion. Typically it is defined as the Markov process whose infinitesimal generator is the Laplacian operator on the unitary group. However, given the upcoming computations in this paper, it is better to use an equivalent definition as the solution of a stochastic differential equation. We refer to subsection 2.1 of \cite{cls} for a short summary on the different definitions.

\begin{defi}
	\label{2UBdef}
	Let $X^N$ be a Hermitian Brownian motion, then the unitary Brownian motion $(U^N_t)_{t\geq 0}$ is the solution of the following stochastic differential equation:
	\begin{equation}
		\label{2defbrunit}
		dU_t^N = \i U_t^N dX_t^N - \frac{1}{2} U_t^N dt,\quad U_0^N = I_N,
	\end{equation}
	
	\noindent where we formally define $U_t^N dX_t^N$ by simply taking the matrix product $$(U_t^N dX_t^N)_{i,j} = \sum_k (U_t^N)_{i,k} d(X_t^N)_{k,j}.$$
	\noindent In particular, almost surely, for any $t$, $U_t^N$ is a unitary matrix of size $N$. A proof can be found in Section 2.1 of \cite{cls}, or one can simply use the Ito formula to show that $(U_t^N)^*U_t^N=I_N$.
\end{defi}

In particular, the unitary Brownian motion and the free unitary Brownian motion are linked with the following proposition. 

\begin{prop}
	\label{smcwod}
	Let $Z^N$ be a family of deterministic matrices, $U_{i,t_i}^N$ be a Haar unitary Brownian motion of size $N$ at time $t_i$. Assuming that the Brownian motions $U_{i,t_i}^N$ are independent, and that the family $Z^N$ converges in distribution towards a family $z$ of non-commuting random variables, then the family $(Z^N,U_{1,t_1}^N,\dots,U_{k,t_k}^N)$ converges in distribution towards $(z,u_{1,t_1},\dots,u_{k,t_k})$ where the variables $u_{i,t_i}$ are free unitary Brownian motions at time $t_i$, free between each other and with the family $z$.
\end{prop}

It has been known for a long time that the $N\times N$ unitary Brownian motion converges in distribution towards the free unitary Brownian motion when $N$ goes to infinity, see \cite{bianebr}. However, since we also have to consider deterministic matrices we will use Theorem 1.4 of \cite{cls}. That being said, we do not use the convergence of the norm, we only need the convergence in distribution which is way easier to prove through induction and stochastic calculus. However, since we could not find a reference to only the convergence in distribution, we will still refer to \cite{cls} when we need to use this result.

\section{Preliminary work}
\label{prelimw}

\begin{prop}
	\label{2concentration}
	Let $P,Q\in\F_{d,q}$, $(U_t^N)_{t\in\R^+}$, $(V_t^N)_{t\in\R^+}$, $(W_t^N)_{t\in\R^+}$ be independent families of $d$ unitary Brownian motions of size $N$. Let $A^N$ be a family of $q$ deterministic matrices, then with $\cov(X,Y) = \E[XY]-\E[X]\E[Y]$, one has for any $T\geq 0$,
	
	\begin{align*}
		&\cov\left(\tr_N\Big(P(U_T^N,A^N)\Big),\tr_N\left(Q(U_T^N,A^N)\right)\right) \\
		&= - \frac{1}{N} \sum_{1\leq i\leq d} \int_{0}^{T} \E\Big[ \tr_N\Big( \D_iP(V_t^NU_{T-t}^N, A^N) \times \D_iQ(W_t^NU_{T-t}^N, A^N) \Big) \Big] dt .
	\end{align*}
	
\end{prop}

\begin{proof}
	
	First, note that since one has
	$$ \overline{\tr_N\Big(P(U_T^N,A^N)\Big)} = \tr_N\Big(P(U_T^N,A^N)^*\Big), $$
	thanks to the polarization identity, we get that
	\begin{align*}
		&\cov\Big(\tr_N\Big(P(U_T^N,A^N)\Big),\tr_N\Big(Q(U_T^N,A^N)^*\Big)\Big)  \\
		&= \frac{1}{4}\Bigg(\var\left( \tr_N\Big((P+Q)(U_T^N,A^N) \right) - \var\left( \tr_N\Big((P-Q)(U_T^N,A^N) \right) \\
		&\quad\quad - \i \var\left( \tr_N\Big((P+\i Q)(U_T^N,A^N) \right) +\i \var\left( \tr_N\Big((P-\i Q)(U_T^N,A^N) \right)\Bigg),
	\end{align*}
	where $\var(X) = \E[|X|^2] - |\E[X]|^2$. Thanks to Proposition 3.1 of \cite{deux}, one has that 
	\begin{align*}
		&\var\left(\tr_N\Big(P(U_T^N,A^N)\Big)\right) \\
		&= \frac{1}{N} \sum_{1\leq i\leq d} \int_{0}^{T} \E\Big[ \tr_N\Big( \D_iP(V_t^NU_{T-t}^N, A^N) \times \D_iQ(W_t^NU_{T-t}^N, A^N)^* \Big) \Big] dt .
	\end{align*}
	Thus by using again the polarization identity, we get that
	\begin{align*}
		&\cov\left(\tr_N\Big(P(U_T^N,A^N)\Big),\tr_N\left(Q^*(U_T^N,A^N)\right)\right) \\
		&= \frac{1}{N} \sum_{1\leq i\leq d} \int_{0}^{T} \E\Big[ \tr_N\Big( \D_iP(V_t^NU_{T-t}^N, A^N) \times \D_iQ(W_t^NU_{T-t}^N, A^N)^* \Big) \Big] dt .
	\end{align*}
	Thus by replacing $Q$ by $Q^*$ and using the fact that $\D_i(Q^*)^*= -\D_iQ$, we get that
	\begin{align*}
		&\cov\left(\tr_N\Big(P(U_T^N,A^N)\Big),\tr_N\left(Q(U_T^N,A^N)\right)\right) \\
		&= - \frac{1}{N} \sum_{1\leq i\leq d} \int_{0}^{T} \E\Big[ \tr_N\Big( \D_iP(V_{T-t}^N U_t^N, A^N) \times \D_iQ(W_{T-t}^N U_t^N, A^N) \Big) \Big] dt .
	\end{align*}
	Hence the conclusion with the change of variable $t\mapsto T-t$.
\end{proof}

If $u_t$ is a free unitary Brownian motion at time $t$, then thanks to Riesz theorem, there is a measure $\nu_t$ such that for any polynomial $P$ in two commuting variables,
$$ \tau(P(u_t,u_t^*)) = \int_{\C} P(z,z^*)\ d\nu_t(z) .$$

\noindent The measure $\nu_t$ is well-known albeit not explicit. The proof of the following theorem can be found in \cite{freebrlaw}.

\begin{theorem}
	\label{densitetro}
	For every $t>0$, the measure $\nu_t$ is absolutely continuous with respect to the Haar measure on $\mathbb{T} = \{z\in\C\ |\ |z|=1\}$. For $t>4$, the support of $\nu_t$ is equal to $\mathbb{T}$, and its density is positive on $\mathbb{T}$. We set $\kappa(t,\omega)$ the density of $\nu_t$, with respect to the Haar measure, at the point $\omega\in\mathbb{T}$. Then for $t>4$, $\kappa(t,\omega)$ is the real part of the only solution with positive real part of the equation,
	\begin{equation}
		\label{2horreq}
		\frac{z-1}{z+1} e^{\frac{t}{2}z}	= \omega .
	\end{equation}
	
\end{theorem}

The following proposition states that for any $t\geq 5$, one can find a function $f_t$ such that given any unitary Brownian motions $u_t$ at time $t$, $f_t(u_t)$ is a Haar unitary. Besides, this function converges exponentially fast towards the identity when $t$ goes to the infinity. This proposition is a refinement of Proposition 3.2 of \cite{deux} where we were building a specific free Brownian motion instead of a function $f$ which let us work with any given free Brownian motion.

\begin{prop}
	\label{2Browniantounitary}
	Given $t\geq 5$, there exist a continuous function $f_t: \mathbb{T}\to\R$ where $\mathbb{T}$ is the unit circle in $\C$ such that if $u_t$ is a free unitary Brownian motion at time $t$, then $f_t(u_t)$ is a free Haar unitary and besides, $\norm{u_t - f_t(u_t)} \leq 4e^2\pi e^{-\frac{t}{2}}$.
\end{prop}

\begin{proof}
	We  set $g_t:s\mapsto \kappa(t,e^{is})$ and $G_t:s\mapsto \int_{0}^s g_t(u)\ du$. Note that thanks to Theorem \ref{densitetro}, since $g_t$ is the density of $\nu_t$ with respect to the Haar measure, we have that
	$$ G_t(2\pi) = \int_0^{2\pi}g_t(s)ds = 2\pi \nu_t(\mathbb{T}) = 2\pi. $$
	Thus since $g_t$ is positive, $G_t$ is a diffeomorphism of $[0,2\pi]$. Let us now define the function $f_t$. We set
	\begin{itemize}
		\item $h: e^{\i s}\in \mathbb{T}\mapsto s\ \modulo 2\pi$,
		\item $f_t: x\in\mathbb{T}\to e^{\i\ G_t\circ h(x)}$.
	\end{itemize}
	Let us first prove that $f_t$ is actually continuous. Since $G_t$ is continuous on $[0,2\pi]$, and $h$ is continuous on $\mathbb{T}\setminus \{1\}$, we only need to check the continuity of $f_t$ around $1$. Let $\varepsilon>0$, then 
	$$ f_t(e^{\i \varepsilon}) = e^{\i G_t(\varepsilon)} = 1 + \mathcal{O}(\varepsilon) = f_t(1) + \mathcal{O}(\varepsilon), $$
	$$ f_t(e^{\i (2\pi-\varepsilon)}) = e^{\i G_t(2\pi-\varepsilon)} = e^{\i \left(2\pi - \int_{2\pi-\varepsilon}^{2\pi}g_t(s)ds \right)} = f_t(1) + \mathcal{O}(\varepsilon).$$
	Hence $f_t$ is indeed continuous. Besides, if $u_t$ is a free unitary Brownian motion at time $t$ in a $\CC^*$-algebra endowed with a trace $\tau$. Then for any polynomial $P$,
	\begin{align*}
		\tau\left(P(f_t(u_t),f_t(u_t)^*)\right) &= \frac{1}{2\pi} \int_0^{2\pi} P\left( f_t(e^{\i s}),\overline{f_t(e^{\i s})} \right) g_t(s) ds \\
		&= \frac{1}{2\pi} \int_0^{2\pi} P\left( e^{\i G_t(s)},e^{-\i G_t(s)} \right) g_t(s) ds \\
		&= \frac{1}{2\pi} \int_0^{2\pi} P\left( e^{\i u},e^{-\i u} \right) du.
	\end{align*}
	Hence $f_t(u_t)$ is indeed a free Haar unitary. Besides, we have that
	\begin{align*}
		\norm{u_t-f_t(u_t)} &= \sup_{s\in [0,2\pi]} |e^{\i s}-e^{\i G_t(s)}| \\
		&= \sup_{s\in [0,2\pi]} \left| \int_0^1 e^{\i \alpha s} (s-G_t(s)) e^{\i (1-\alpha) G_t(s)} d\alpha \right| \\
		&\leq \sup_{s\in [0,2\pi]} \left| s-G_t(s)\right| \\
		&\leq 2\pi \sup_{s\in [0,2\pi]} \left| 1-g_t(s)\right|.
	\end{align*}
	
	\noindent The rest of the proof follows just like that of Proposition 3.2 of \cite{deux}.
	
\end{proof}

\section{Proof of Theorem \ref{3lessopti}}
\label{3mainsec}

\subsection{A first rough formulation of the coefficients}

\label{3mauvaisepres}

In this subsection we prove the following lemma which will be the backbone of the proof of the topological expansion. The idea is that by interpolating between Haar unitary matrices and free Haar unitaries with the help of free unitary Brownian motions we end up with a remainder term of order $N^{-2}$. But most importantly the remainder term is explicit and consequently one can proceed by induction and reapply the same lemma, which is how we get our expansion.

Note that thanks to the definition of $\A_N$ in Definition \ref{3tra}, it makes sense to consider matrices and free unitary Brownian motions in the same space. One can also assume that those matrices are random thanks to Proposition 2.7 of \cite{un}. Finally, to better understand Equation \eqref{3nececpour}, you can check Examples 2.14, 2.15 and 2.16 of \cite{trois}.

\begin{lemma}
	\label{3imp2}
	Let the following objects be given,
	\begin{itemize}
		\item $U^N = (U_1^N,\dots,U_d^N)$ independent Haar matrices of size $N$,
		\item $u^s=(u_t^s)_{t\geq 0}$ for $s$ from $1$ to $n+1$, families of $d_s$ free unitary Brownian motion with $d_{n+1}=d$, free between each other and from $u$,
		\item $v^s,w^s$ free copies of $u^s$, free between each other,
		\item $Z^N = (Z_1^N,\dots,Z_q^N)$ deterministic matrices of size $N$,
		\item for $s$ from $1$ to $n+1$, $ z_{r}^{1,s} = \Big( v^1_{t_1},\dots,v^{s-1}_{t_{s-1}}, v^s_{r} u^s_{t_s-r},u^{s+1}_{t_{s+1}},\dots,u^{n}_{t_{n}}, u^{n+1}_{t_{n+1}},U^N, Z^{N}\Big)$,
		\item $z_{r}^{2,s}$, defined similarly but with $w_s$ instead of $v_s$,
		\item $\widetilde{z}_{r}^{1,s}$ and $\widetilde{z}_{r}^{2,s}$ defined similarly but where we replaced $u^s,v^s,w^s$ by free copies,
		\item $Q\in \F_{d_1+\dots+d_n+d,q}$.
	\end{itemize}
	
	\noindent Then, let $S\in \F_{d_1+\dots+d_n+2d,q}$ be given by
	$$ S\left(u_{t_1}^1,\dots, u_{t_{n+1}}^{n+1},U^N,Z^N\right) = Q\left(u_{t_1}^1,\dots, u_{t_{n+1}}^{n+1}U^N,Z^N\right),$$
	and let $\delta_{s,j}$ be defined similarly to the noncommutative differential introduced in Definition \ref{3application} but with respect to $u_{j,t_s}^s$ instead of $U_i$. For any $N\in\N$ and $T\in\R^+$,
	\begin{align}
		\label{3nececpour}
		&\E\left[\tau_N\Big(Q\left(u_{t_1}^1,\dots, u_{t_n}^n,U^N,Z^N\right)\Big)\right] - \E\left[\tau_N\Big(Q\left( u_{t_1}^1,\dots, u_{t_n}^n, u_T^{n+1} U^N,Z^N\right)\Big)\right] \nonumber \\
		&= \frac{1}{2N^2}\sum_{\substack{1\leq s\leq n+1\\ 1\leq j\leq d_s \\ 1\leq i\leq d}}\int_{0}^{T}\int_0^{t_s} \E\Big[ \tau_{N}\Big(\left[\Big(\delta_{s,j}^2\left[\delta_i^1 \D_iS \right] \left(z_{r}^{1,s} \right)\Big)\boxtimes \Big( \delta_{s,j}^1\left[\delta_i^1 \D_iS \right] \left(\widetilde{z}_{r}^{1,s} \right)\Big)\right] \\
		&\quad\quad\quad\quad\quad\quad\quad\quad\quad\quad\quad \boxtimes \left[\Big(\delta_{s,j}^2\left[\delta_i^2 \D_iS \right] \left(\widetilde{z}_{r}^{2,s} \right)\Big)\boxtimes \Big(\delta_{s,j}^1\left[\delta_i^2 \D_i \right] \left(z_{r}^{2,s} \right)\Big)\right] \Big) \Big]\ dr\ dt_{n+1}, \nonumber
	\end{align}
	with the notation $\boxtimes$ as in Definition \ref{3sweedler}.
\end{lemma}

Before giving a proof, we need the following technical result which gives an estimate on non-diagonal coefficients of the random matrices we consider.

\begin{lemma}
	\label{2nondiag}
	If $t=(t_1,\dots,t_l)$ and $U^k_t=(U^{i,kN}_{t_i})_{1\leq i\leq l}$ is a family of independent Haar unitary Brownian motions of size $kN$ at time $t_s$, $K^N$ a family of $q$ deterministic matrices, then let
	$$ 	S^k_t = \left( U_t^k, K^N\otimes I_k \right) .$$
	With $P_{1,2} = I_N\otimes E_{1,2}$, $\E_k$ the expectation with respect to $ U^k_t $, given $Q\in \mathcal{F}_{l,q}$, we have that for any $\varepsilon>0$,
	$$ \lim_{k\to \infty} k^{3-\varepsilon} \E_k\left[ \ts_{kN}\left( Q(S^k_t) P_{1,2} \right)\right] =0. $$
\end{lemma}

\begin{proof}
	Given $A_1,\dots,A_r, B_1,\dots, B_r\in \A_{m,q}$, assuming the $B_i$ are self-adjoints, we define the following functions,
	$$ f_{A}^t: \alpha\in [0,1] \mapsto \E_k\left[ \ts_{kN}\left( (A_1 e^{\i \alpha B_1 } \dots A_r e^{\i \alpha B_r }) (S^k_t) P_{1,2} \right)\right] ,$$
	$$ d_n^t: \alpha\in [0,1] \mapsto \sup_{\substack{\sum_i \deg A_i \leq n,\ A_i\text{ monomials} \\ r_s\in [0,t_s]}} \left| f_{A}^r(\alpha) \right| .$$
	
	\noindent Note that the quantities $f_{A}^t$ and $d_n^t$ depends on $B$. However we will not keep track of this dependency in the notations since, unlike the polynomials $A_i$, the polynomials $B_i$ will be fixed in the rest of the proof.
	
	With $D= \max \{1,\norm{K_1^N},\dots,\norm{K_q^N}\}$, we have that for any $t\in(\R^+)^l$ and $\alpha\in [0,1]$, $|d_n(\alpha)| \leq D^n$. Consequently for $a<1/D$, we define 
	$$ g(a,\alpha,t) = \sum_{n\geq 0} d_n^t(\alpha) a^n .$$
	
	\noindent Let $m = \sup_i \deg B_i$ and $A$ be such that $\sum_i \deg A_i \leq n$, there exists a constant $C_B$ which only depends on the coefficients of the $B_i$ such that
	$$ \left| \frac{d f_{A}^t(\alpha)}{d\alpha} \right| \leq C_B\ d_{n+m}^t(\alpha).$$
	
	\noindent Naturally we get that for any $\alpha\in [0,1]$
	$$ \left| f_{A}^t(\alpha) \right| \leq \left| f_{A}^t(0) \right| + C_B  \int_{0}^{\alpha} d^t_{n+m}(\beta) d\beta.$$
	
	\noindent And by taking the supremum over $A$, we get that 
	$$ d_n^t(\alpha) \leq d_n^t(0) + C_B  \int_{0}^{\alpha} d^t_{n+m}(\beta) d\beta.$$
	
	\noindent Hence by summing over $n$, we have for $a$ small enough,
	\begin{align*}
		g(a,\alpha,t) &= \sum_{n\geq 0} d_n^t(\alpha) a^n \\
		&\leq \sum_{n\geq 0} d_n^t(0)a^n + C_B  \int_{0}^{\alpha} \sum_{n\geq 0}	 d_{n+m}^t(\beta) a^n d\beta \\
		&\leq g(a,0,t) + C_B  \int_{0}^{\alpha} \sum_{n\geq m} d_{n}^t(\beta) a^{n-m} d\beta \\
		&\leq g(a,0,t) +  C_B a^{-m}\int_{0}^{\alpha} g(a,\beta,t) d\beta .
	\end{align*}
	
	\noindent Thanks to Gr\"onwall's inequality (see \cite{legallfr}, Lemma 8.4), we get that for any $\alpha\in [0,1]$,
	$$ g(a,\alpha,t) \leq g(a,0,t)\times e^{\alpha C_B a^{-m}}.$$
	
	\noindent Thus for $a<1/D$, we have 
	\begin{equation}
		\label{gronwalled}
		\limsup_{k\to \infty} k^{3-\varepsilon}g(a,\alpha,t) \leq e^{\alpha C_B a^{-m}} \limsup_{k\to \infty} k^{3-\varepsilon} g(a,0,t
		) .
	\end{equation}
	
	\noindent Besides, we have the following formula,
	$$ g(a,0,t) = \sum_{n\geq 0} a^n \sup_{\substack{A \text{ monomial, } \deg A \leq n \\ r_s\in [0,t_s]}} \left| \E_k\left[ \ts_{kN}\left( A(S^k_r) P_{1,2} \right)\right] \right| .$$
	Consequently, we set
	$$ c_n^t = \sup_{\substack{A \text{ monomial, } \deg A \leq n \\ r_s\in [0,t_s]}} \left| \E_k\left[ \ts_{kN}\left( A(S^k_r) P_{1,2} \right)\right] \right| .$$
	
	Let $A$ be a monomial of degree at most $n$, we define $A_t$ as the monomial $A$ evaluated in $S_t^{k}$. Thanks to Proposition 2.4 of \cite{deux}, and with $\deg_l(A)$ the degree of $A$ with respect to the variables $U_{t_l}^{l,kN}$ and ${U_{t_l}^{l,kN}}^*$ we have that 
	
	\begin{align*}
		\frac{d}{dt_l}\E_k\left[\ts_{kN}\left( A_t P_{1,2}\right)\right] = &- \frac{\deg_l(A)}{2}\ \E_k\left[\ts_{kN}\left( A_t P_{1,2}\right) \right] \\
		&- \sum_{ A = B U_l C U_l D} \E_k\left[\ts_{kN}\left(B_t U_{t_l}^{l,kN} D_tP_{1,2}\right) \ts_{kN}\left(C_t U_{t_l}^{l,kN}\right)\right] \\
		&- \sum_{A = B U_l^* C U_l^* D} \E_k\left[\ts_{kN}\left(B_t {U_{t_l}^{l,kN}}^* D_tP_{1,2}\right) \ts_{kN}\left(C_t {U_{t_l}^{l,kN}}^*\right)\right] \\
		&+ \sum_{ A=B U_l C U_l^* D } \E_k\left[\ts_{kN}(B_t D_tP_{1,2}) \ts_{kN}(C_t)\right] \\
		&+ \sum_{ A=B U_l^* C U_l D } \E_k\left[\ts_{kN}(B_t D_tP_{1,2}) \ts_{kN}(C_t)\right],
	\end{align*}
	
	\noindent which is equivalent to
	\begin{align*}
		\frac{d}{dt_l}\Big(e^{\frac{\deg_l(A)}{2} t_l}&\E_k\left[\ts_{kN}\left( A_t P_{1,2}\right)\right]\Big) \\
		= e^{\frac{\deg_l(A)}{2} t_l} \Bigg( &- \sum_{ A = B U_l C U_l D} \E_k\left[\ts_{kN}\left(B_t U_{t_l}^{l,kN} D_tP_{1,2}\right) \ts_{kN}\left(C_t U_{t_l}^{l,kN}\right)\right] \\
		&- \sum_{A = B U_l^* C U_l^* D} \E_k\left[\ts_{kN}\left(B_t {U_{t_l}^{l,kN}}^* D_tP_{1,2}\right) \ts_{kN}\left(C_t {U_{t_l}^{l,kN}}^*\right)\right] \\
		&+ \sum_{ A=B U_l C U_l^* D } \E_k\left[\ts_{kN}(B_t D_tP_{1,2}) \ts_{kN}(C_t)\right] \\
		&+ \sum_{ A=B U_l^* C U_l D } \E_k\left[\ts_{kN}(B_t D_tP_{1,2}) \ts_{kN}(C_t)\right] \Bigg).
	\end{align*}
	
	\noindent Consequently with $\widetilde{t}=(t_1,\dots,t_{l-1})$, we have for any $t$,
	
	\begin{align*}
		\E_k&\left[\ts_{kN}\left( A_t P_{1,2}\right)\right] = \ e^{-\frac{\deg_l(A)}{2} t_l} \E_k\Big[\ts_{kN}\left( A_{(\widetilde{t},0)} P_{1,2}\right)\Big] \\
		+ \int_{0}^{t_l} &e^{- \frac{\deg_l(A)}{2} (t_l-s)} \Bigg( - \sum_{A = B U_l C U_l D} \E_k\left[\ts_{kN}\left(B_{(\widetilde{t},s)} U_{s}^{l,kN} D_{(\widetilde{t},s)}P_{1,2}\right) \ts_{kN}\left(C_{(\widetilde{t},s)} U_{s}^{l,kN}\right)\right] \\
		&\quad\quad\quad\quad\quad- \sum_{ A = B U_l^* C U_l^* D} \E_k\left[\ts_{kN}\left(B_{(\widetilde{t},s)} {U_{s}^{l,kN}}^* D_{(\widetilde{t},s)}P_{1,2}\right) \ts_{kN}\left(C_{(\widetilde{t},s)} {U_{s}^{l,kN}}^*\right)\right] \\
		&\quad\quad\quad\quad\quad + \sum_{ A=B U_l C U_l^* D } \E_k\left[\ts_{kN}(B_{(\widetilde{t},s)} D_{(\widetilde{t},s)}P_{1,2}) \ts_{kN}(C_{(\widetilde{t},s)})\right] \\
		&\quad\quad\quad\quad\quad+ \sum_{ A=B U_l^* C U_l D  } \E_k\left[\ts_{kN}(B_{(\widetilde{t},s)} D_{(\widetilde{t},s)}P_{1,2}) \ts_{kN}(C_{(\widetilde{t},s)})\right] \Bigg) ds .
	\end{align*}
	
	\noindent Thanks to Proposition \ref{2concentration}, we have that
	$$\left|\cov\left(\ts_{kN}\left(B_{(\widetilde{t},s)} D_{(\widetilde{t},s)}P_{1,2}\right), \ts_{kN}\left(C_{(\widetilde{t},s)}\right)\right)\right| \leq \frac{n^2sD^n}{k^3N^2}, $$
	and we have the same inequality for the other three lines. Consequently we have that
	\begin{align*}
		\left| \E_k\left[\ts_{kN}\left( A_t P_{1,2}\right)\right] \right| \leq &\frac{2\deg_l(A) t_l D^nn^2}{k^3N^2}\ + \left|\E_k\Big[\ts_{kN}\left( A_{(\widetilde{t},0)}P_{1,2}\right)\Big]\right|\\
		+ \int_{0}^{t_l} e^{- \frac{\deg_l(A)}{2} (t_l-s)} \Bigg( & \sum_{A = B U_l C U_l D} \left| \E_k\left[\ts_{kN}\left(B_s U_{s}^{l,kN} D_sP_{1,2}\right)\right] \E_k\left[\ts_{kN}\left(C_s U_{s}^{l,kN}\right)\right] \right| \\
		+ & \sum_{A = B U_l^* C U_l^* D} \left| \E_k\left[\ts_{kN}\left(B_s {U_{s}^{l,kN}}^* D_tP_{1,2}\right)\right] \E_k\left[\ts_{kN}\left(C_s {U_{s}^{l,kN}}^*\right)\right] \right| \\
		+&\sum_{A=B U_l C U_l^* D  } \left|\E_k\left[\ts_{kN}(B_s D_sP_{1,2})\right] \E_k\left[\ts_{kN}(C_s)\right] \right| \\
		+&\sum_{ A=B U_l^* C U_l D  } \left|\E_k\left[\ts_{kN}(B_s D_sP_{1,2})\right] \E_k\left[\ts_{kN}(C_s)\right] \right| \Bigg) ds .
	\end{align*}
	
	\noindent This means that,	
	\begin{align*}
		\left| \E_k\left[\ts_{kN}\left( A_t P_{1,2}\right)\right] \right| &\leq \frac{2\deg_l(A) t_l D^nn^2}{k^3N^2}	+ \left|\E_k\Big[\ts_{kN}\left( A_{(\widetilde{t},0)}P_{1,2}\right)\Big]\right|\\
		&\quad+ \int_{0}^{t_l} e^{- \frac{\deg_l(A)}{2} (t_l-s)} ds \Bigg( \sum_{A = B U_l C U_l D} c^t_{\deg(BD)+1} D^{\deg(C)} \\
		&\quad\quad\quad\quad\quad\quad\quad\quad\quad\quad\quad + \sum_{A = B U_l^* C U_l^* D} c^t_{\deg(BD)+1} D^{\deg(C)} \\
		&\quad\quad\quad\quad\quad\quad\quad\quad\quad\quad\quad + \sum_{A=B U_l C U_l^* D  } c^t_{\deg(BD)} D^{\deg(C)} \\
		&\quad\quad\quad\quad\quad\quad\quad\quad\quad\quad\quad + \sum_{ A=B U_l^* C U_l D} c^t_{\deg(BD)} D^{\deg(C)} \Bigg). 
	\end{align*}
	
	\noindent Hence by iterating the process, and since $\E_k\Big[\ts_{kN}\left( A_{(0,\dots,0)}P_{1,2}\right)\Big]=0$, we have that
	\begin{align*}
		\left| \E_k\left[\ts_{kN}\left( A_t P_{1,2}\right)\right] \right| &\leq \frac{2(\deg_1(A)+\dots+\deg_l(A)) D^nn^2\times \max_i t_i}{k^3N^2}\\
		&\quad+ \sum_{i=1}^l\ \frac{2}{\deg_i(A)}\ \Bigg( \sum_{A = B U_i C U_i D} c^t_{\deg(BD)+1} D^{\deg(C)} \\
		&\quad\quad\quad\quad\quad\quad\quad\quad\quad + \sum_{A = B U_i^* C U_i^* D} c^t_{\deg(BD)+1} D^{\deg(C)} \\
		&\quad\quad\quad\quad\quad\quad\quad\quad\quad + \sum_{A=B U_i C U_i^* D  } c^t_{\deg(BD)} D^{\deg(C)} \\
		&\quad\quad\quad\quad\quad\quad\quad\quad\quad + \sum_{ A=B U_i^* C U_i D} c^t_{\deg(BD)} D^{\deg(C)} \Bigg) \\
		&\leq \frac{2n^3 D^n\times \max_i t_i}{k^3N^2} + 4l \sum_{0\leq d\leq n-1} D^d c_{n-1-d}^t \\
		&\leq \frac{2n^3 D^n\times \max_i t_i}{k^3N^2} + 4 l \sum_{0\leq d\leq n-1} D^d c_{n-1-d}^t.
	\end{align*}
	
	\noindent Hence, for any $n\geq 1$,
	$$ c_n^t \leq \frac{2n^3 D^n\times \max_i t_i}{k^3N^2} + 4 l \sum_{0\leq d\leq n-1} D^d c_{n-1-d}^t.$$
	
	\noindent Since the trace of $P_{1,2} = I_N \otimes E_{1,2}$ is equal to $0$, we have $c_0 = 0$. Thus we fix $s:a\mapsto \sum_{n\geq 0} \frac{2\max_i t_i\ n^3 (aD)^n}{N^2}$, and for $a$ small enough, 
	\begin{align*}
		g(a,0,t) &\leq \frac{s(a)}{k^3} + 4l \sum_{n\geq 1} \left(\sum_{0\leq d\leq n-1} D^d c_{n-1-d}^t \right) a^n \\
		&\leq \frac{s(a)}{k^3} + \frac{4l a}{1-aD} g(a,0,t)
	\end{align*}
	
	\noindent Thus for $a$ small enough, $g(a,0,t) \leq k^{-3}$, in which case we have thanks to Equation \eqref{gronwalled} that for any $A$,
	\begin{align*}
		\limsup_{k\to \infty} k^{3-\varepsilon}&\ \left| \E_k\left[ \ts_{kN}\left( (A_1 e^{\i B_1 } \dots A_r e^{\i B_r }) (S_t^k) P_{1,2} \right)\right] \right| \\
		&\leq a^{-\sum_i \deg A_i} \limsup_{k\to \infty} k^{3-\varepsilon}g(a,1,t)  \\
		&\leq a^{-\sum_i \deg A_i} e^{C_B a^{-m}} \limsup_{k\to \infty} k^{3-\varepsilon}g(a,0,t) \\
		&= 0
	\end{align*}
	
	\noindent Hence the conclusion.
	
\end{proof}

\begin{proof}[Proof of Lemma \ref{3imp2}] We divide the proof in three steps in order to make it easier to read. In the first step we define the quantity $\Lambda_{N,t_{n+1}}$. In the second one we reformulate this quantity as a covariance and in the last one we use Proposition \ref{2concentration} to finish the computations.
	
	\textbf{Step 1:} With $U_{t_s}^{s,kN}$ a family of $d_s$ Haar unitary Brownian motions at time $t_s$ and size $kN$, we set
	$$ Y_{t_{n+1}}^N = \left(u_{t_1}^1,\dots, u_{t_n}^n, u_{t_{n+1}}^{n+1} U^N,Z^{N}\right),$$
	$$ Y_{t_{n+1}}^{k,N} = \left( U_{t_1}^{1,kN}\dots, U_{t_n}^{n,kN}, U_{t_{n+1}}^{n+1,kN} U^N\otimes I_k,Z^{N}\otimes I_k\right),$$
	we have, 
	\begin{align*}
		&\E\left[\tau_N\Big(Q\left(u_{t_1}^1,\dots, u_{t_n}^n,U^N,Z^N\right)\Big)\right] - \E\left[\tau_N\Big(Q\left(u_{t_1}^1,\dots, u_{t_n}^n, u_T^{n+1}U^N,Z^N\right)\Big)\right] \\
		&= -\int_{0}^{T} \E\left[ \frac{d}{dt_{n+1}} \tau_N\Big( Q\left(Y_{t_{n+1}}^N\right)\Big) \right] dt_{n+1} .
	\end{align*}
	
	\noindent Thanks to Proposition 3.3 of \cite{deux} with $M=1$, we have that 
	\begin{equation}
		\label{compderivat}
		\frac{d}{dt_{n+1}}\tau_N\Big(Q\left(Y_{t_{n+1}}^N\right)\Big) = -\frac{1}{2} \sum_{1\leq i\leq d} \tau_N\otimes\tau_N\left(\delta_i \D_i P\left(Y_{t_{n+1}}^N\right)\right).
	\end{equation}
	
	\noindent Since all of our random variables are unitary matrices, thanks to Proposition \ref{smcwod} and the dominated convergence theorem,
	\begin{align}
		\label{2premiertr}
		\Lambda_{N,t_{n+1}} &:= \E\left[\tau_N\otimes\tau_N\left(\delta_i \D_i P(Y_{t_{n+1}}^N)\right)\right] \\
		&= \lim_{k\to\infty} \E\left[ (\E_k\circ\ts_{kN}) \otimes (\E_k\circ\ts_{kN})\left(\delta_i\D_i \left(Y_{t_{n+1}}^{k,N}\right)\right) \right] , \nonumber
	\end{align}

	\noindent where $(\E_k\circ\ts_{kN})^{\otimes 2}\left( A\otimes B \left(Y_{t_{n+1}}^{k,N}\right)\right) = \E_k[\ts_{kN}(A(Y_{t_{n+1}}^{k,N}))]\E_k[\ts_{kN}(B(Y_{t_{n+1}}^{k,N}))]$, $\E_k$ being the expectation with respect to $\left(U_{t_s}^{s,kN}\right)_{1\leq s\leq n+1}$. 
	
	\textbf{Step 2:} Then since given $V\in \mathbb{U}_N$,  $U^N_i$ has the same law as $U^N_i V$, we get that $Y_{t_{n+1}}^{k,N}$ has the same law as
	\begin{align*}
		\Big(U_{t_1}^{1,kN},\dots, U_{t_n}^{n,kN}, U_{t_{n+1},1}^{n+1,kN}U^N_1\otimes I_k,\dots\\
		\dots, U_{t_{n+1},i}^{n+1,kN}& (U^N_iV)\otimes I_k,\dots, U_{t_{n+1},d}^{n+1,kN}U^N_d\otimes I_k,Z^{N}\otimes I_k\Big),
	\end{align*}
	
	\noindent Consequently given $q\in \F_{d_1+\dots+d_n+d,q}$, we have that
	
	\begin{align*}
		\E\left[ q\left(Y_{t_{n+1}}^{k,N}\right) \right] = \E\Big[ q\Big( &U_{t_1}^{1,kN},\dots, U_{t_n}^{n,kN}, U_{t_{n+1},1}^{n+1,kN}U^N_1\otimes I_k,\dots, \\
		&U_{t_{n+1},i}^{n+1,kN} (U^N_iV)\otimes I_k,\dots, U_{t_{n+1},d}^{n+1,kN}U^N_d\otimes I_k,Z^{N}\otimes I_k\Big) \Big] .
	\end{align*}
	
	\noindent Hence let $H$ be a skew-Hermitian matrix, then for any $s\in \R$, $e^{sH}\in \mathbb{U}_N$, thus by taking $V$ this matrix and differentiating with respect to $s$ we get that for any $i$,
	$$\E\Big[ \delta_i q(Y_{t_{n+1}}^{k,N}) \# (H\otimes I_{k}) \Big] =0 .$$
	
	\noindent Since every matrix is a linear combination of skew-Hermitian matrices (indeed, if $A\in\M_N(\C)$, then $2A = (A-A^*) + \i \times (-\i)(A^*+A)\ $), the previous equality is true for any matrix $H\in \M_N(\C)$. Thus with $(\tr_N\otimes I_{k})^{\bigotimes 2} = m\circ (\tr_N\otimes I_{k})^{\otimes 2}$, we get that for any $i$,
	\begin{align}
		\label{2SDeq}
		&\E\Big[ (\tr_N\otimes \id_{k})^{\bigotimes 2}\left(\delta_i q\left(Y_{t_{n+1}}^{k,N}\right)\right) \Big] \\
		&= \sum_{ 1\leq r,s \leq N} (g_r^*\otimes I_{k})\ \E\Big[ \delta_i q\left(Y_{t_{n+1}}^{k,N}\right) \# (E_{r,s}\otimes I_{k}) \Big] (g_s\otimes I_{k}) = 0, \nonumber
	\end{align}
	where $(g_i)_{1\leq i\leq N}$ the canonical basis of $\C^N$. Let $S,T\in \M_{kN}(\C)$ be deterministic matrices, then with $(f_i)_{1\leq i\leq k}$ the canonical basis of $\C^k$, by using the fact that
	$$ \tr_N\otimes \id_{k}(S) = \sum_{1\leq n\leq N} g_n^*\otimes I_k\ T\ g_n\otimes I_k,\quad\quad I_k = \sum_{1\leq l\leq k} f_lf_l^*,$$
	we get that
	\begin{align*}
		& \tr_{k}\left( (\tr_N\otimes \id_{k})^{\bigotimes 2}\left( S\otimes T \right) \right) \\
		&= \sum_{1\leq l,l'\leq k} \sum_{1\leq m\leq N} g_m^*\otimes f_l^*\ S\ g_m\otimes f_{l'} \sum_{1\leq n\leq N} g_n^*\otimes f_{l'}^*\ T\ g_n\otimes f_l  \\
		&= \sum_{1\leq l,l'\leq k} \tr_N( I_N\otimes f_l^*\ S\ I_N\otimes f_{l'}) \tr_N(I_N\otimes f_{l'}^*\ T\ I_N\otimes f_l) \\
		&= \sum_{1\leq l,l'\leq k} \tr_{kN}\big( S\ I_N\otimes E_{l',l}\big) \tr_{kN}\big( T\ I_N\otimes E_{l,l'}\big) .
	\end{align*}
	
	\noindent Thus by using equation \eqref{2SDeq}, with $P_{l,l'} = I_N\otimes E_{l,l'}$, we have for any $i$,
	\begin{equation*}
		\sum_{1\leq l,l'\leq k} \E\Big[ \tr_{kN}^{\otimes 2}\left(\delta_i q\left(Y_{t_{n+1}}^{k,N}\right) \times P_{l',l}\otimes P_{l,l'} \right) \Big] = 0 .
	\end{equation*}
	
	\noindent And consequently,
	\begin{align}
		\label{2deuxiemetr}
		&\sum_{1\leq l,l'\leq k} \E\Big[ \tr_{kN}^{\otimes 2}\left(\delta_i q\left(Y_{t_{n+1}}^{k,N}\right) \times P_{l',l}\otimes P_{l,l'} \right) \Big] \nonumber \\
		&\quad\quad - \E\left[ (\E_k\circ\tr_{kN})^{\otimes 2}\left(\delta_i q\left(Y_{t_{n+1}}^{k,N}\right) \times P_{l',l}\otimes P_{l,l'} \right) \right] \\
		&= - \sum_{1\leq l,l'\leq k} \E\left[ (\E_k\circ\tr_{kN})^{\otimes 2}\left(\delta_i q\left(Y_{t_{n+1}}^{k,N}\right) \times P_{l',l}\otimes P_{l,l'} \right) \right] . \nonumber
	\end{align}
	
	\noindent Let $V,W\in \M_k(\C)$ be permutation matrices. Since $I_{N}\otimes V$ commutes with $Z^{N}\otimes I_k$ and $U^N\otimes I_{k}$, and that the law of $U^{s,kN}_{t_s,j}$ is invariant by conjugation by a unitary matrix, it follows that the law of every matrix of $Y_{t_{n+1}}^{k,N}$ is invariant by conjugation by $I_{N}\otimes V$ or $I_{N}\otimes W$. Thus,
	\begin{align*}
		&(\E_k\circ\tr_{kN})^{\otimes 2}\left(\delta_i q\left(Y_{t_{n+1}}^{k,N}\right) \times P_{l',l}\otimes P_{l,l'} \right) \\
		&= (\E_k\circ\tr_{kN})^{\otimes 2}\left(\delta_i q\left(Y_{t_{n+1}}^{k,N}\right) \times VP_{l',l}V^*\otimes WP_{l,l'}W^* \right) .
	\end{align*}
	
	\noindent Thus by using well-chosen matrices, we get
	\begin{itemize}
		\item if $l= l'$, 
		\begin{align}
			\label{besoindunlabel1}
			&(\E_k\circ\tr_{kN})^{\otimes 2}\left(\delta_i q\left(Y_{t_{n+1}}^{k,N}\right) \times P_{l',l}\otimes P_{l,l'} \right) \\
			&= (\E_k\circ\tr_{kN})^{\otimes 2}\left(\delta_i q\left(Y_{t_{n+1}}^{k,N}\right) \times P_{1,1}\otimes P_{1,1} \right), \nonumber
		\end{align}
		\item if $l\neq l'$, 
		\begin{align}
			\label{besoindunlabel2}
			&(\E_k\circ\tr_{kN})^{\otimes 2}\left(\delta_i q\left(Y_{t_{n+1}}^{k,N}\right) \times P_{l',l}\otimes P_{l,l'} \right) \\
			&= (\E_k\circ\tr_{kN})^{\otimes 2}\left(\delta_i q\left(Y_{t_{n+1}}^{k,N}\right) \times P_{1,2}\otimes P_{1,2} \right). \nonumber
		\end{align}
	\end{itemize}
	
	\noindent Consequently, we have that
	\begin{itemize}
		\item Equation \eqref{2deuxiemetr} simplifies into
		\begin{align*}
			&\sum_{1\leq l,l'\leq k} \E\Big[ \tr_{kN}^{\otimes 2}\left(\delta_i q\left(Y_{t_{n+1}}^{k,N}\right) \times P_{l',l}\otimes P_{l,l'} \right) \Big] \\
			&\quad\quad - \E\left[ (\E_k\circ\tr_{kN})^{\otimes 2}\left(\delta_i q\left(Y_{t_{n+1}}^{k,N}\right) \times P_{l',l}\otimes P_{l,l'} \right) \right] \\
			&= - k \E\left[(\E_k\circ\tr_{kN})^{\otimes 2}\left(\delta_i q\left(Y_{t_{n+1}}^{k,N}\right) \times P_{1,1}\otimes P_{1,1} \right)\right] \\
			&\quad - k(k-1) \E\left[(\E_k\circ\tr_{kN})^{\otimes 2}\left(\delta_i q\left(Y_{t_{n+1}}^{k,N}\right) \times P_{1,2}\otimes P_{1,2} \right) \right] .
		\end{align*}
		
		\item Whereas we have that
		
		\begin{align*}
			&\E\left[(\E_k\circ\ts_{kN}) \otimes (\E_k\circ\ts_{kN})\left(\delta_iq \left(Y_{t_{n+1}}^{k,N}\right)\right)\right] \\
			&= \frac{1}{(kN)^2} \sum_{ 1\leq l,l' \leq k} \E\left[(\E_k\circ\tr_{kN})^{\otimes 2} \left(\delta_iq \left(Y_{t_{n+1}}^{k,N}\right) \times P_{l,l}\otimes P_{l',l'}\right)\right] \\
			&= \frac{1}{N^2} \E\left[(\E_k\circ\tr_{kN})^{\otimes 2}\left(\delta_iq \left(Y_{t_{n+1}}^{k,N}\right) \times P_{1,1}\otimes P_{1,1}\right)\right] .
		\end{align*}
	\end{itemize}
	
	\noindent Thus with $q=\D_iQ$, by combining the last two equations, we have that
	\begin{align*}
		\Lambda_{N,t_{n+1}} = \lim_{k\to\infty} \E&\left[ (\E_k\circ\ts_{kN}) \otimes (\E_k\circ\ts_{kN}) \left(\delta_i\D_iQ \left(Y_{t_{n+1}}^{k,N}\right)\right) \right]  \\
		= \lim_{k\to\infty} - &\frac{1}{k N^2}  \sum_{1\leq l,l'\leq k} \E\Big[ \tr_{kN}^{\otimes 2}\left(\delta_i \D_iQ\left(Y_{t_{n+1}}^{k,N}\right) \times P_{l',l}\otimes P_{l,l'} \right) \Big] \\
		&\quad\quad\quad\quad\quad\ - \E\left[(\E_k\circ\tr_{kN})^{\otimes 2}\left(\delta_i \D_iQ\left(Y_{t_{n+1}}^{k,N}\right) \times P_{l',l}\otimes P_{l,l'} \right)\right]  \\
		-& \frac{k-1}{N^2} \E\left[ (\E_k\circ\tr_{kN})^{\otimes 2}\left(\delta_i \D_iQ\left(Y_{t_{n+1}}^{k,N}\right) \times P_{1,2}\otimes P_{1,2} \right) \right] .
	\end{align*}
	
	\noindent Thanks to Lemma \ref{2nondiag} which we use with $K^N=(U^N,Z^N)$, the last term converges towards $0$, hence
	\begin{align*}
		\Lambda_{N,t_{n+1}}	= \lim_{k\to\infty} - &\frac{1}{k N^2}  \sum_{1\leq l,l'\leq k} \E\Big[ \tr_{kN}^{\otimes 2}\left(\delta_i \D_iQ\left(Y_{t_{n+1}}^{k,N}\right) \times P_{l',l}\otimes P_{l,l'} \right) \Big] \\
		&\quad\quad\quad\quad\quad\ - \E\left[ (\E_k\circ\tr_{kN})^{\otimes 2}\left(\delta_i \D_iQ\left(Y_{t_{n+1}}^{k,N}\right) \times P_{l',l}\otimes P_{l,l'} \right)\right] 
	\end{align*}
	
	\textbf{Step 3:}  We have by definition that 
	$$ \delta_i\D_iS\left(U_{t_1}^{1,kN}\dots, U_{t_{n+1}}^{n+1,kN}, U^N\otimes I_k,Z^{N}\otimes I_k \right) = \delta_i\D_iQ\left(Y^{k,N}_{t_{n+1}}\right),$$
	
	\noindent thus we set for $s$ from $1$ to $n+1$,
	\begin{align*}
		Z_{r,k}^{1,s} = \Big( V^{1,kN}_{t_1},\dots,V ^{s-1,kN}_{t_{s-1}}, V^{s,kN}_{r} U^{s,kN}_{t_s-r} ,U^{s+1,kN}_{t_{s+1}}, &\dots \\
		\dots,U^{n,kN}_{t_{n}},& U^{n+1,kN}_{t_{n+1}}, U^N\otimes I_k, Z^{N}\otimes I_k\Big).
	\end{align*}
	We also define $Z_{r,k}^{2,s}$ a copy of $Z_{r,k}^{1,s}$ where we replaced each Brownian motion $(V^{i,kN}_{t})_{t\geq 0}$ by an independent copy $(W^{i,kN}_{t})_{t\geq 0}$. Thus since we have that $Z_{t_s,k}^{1,s}=Z_{0,k}^{1,s+1}$ and similarly $Z_{t_s,k}^{2,s}=Z_{0,k}^{2,s+1}$, we get that $\Lambda_{N,t_{n+1}}$ is equal to
	\begin{align*}
		\lim_{k\to\infty} -&\frac{1}{kN^2}  \sum_{1\leq l,l'\leq k} \sum_{1\leq s\leq n+1}\E\Big[ \tr_{kN}^{\otimes 2}\left(\delta_i^1 \D_iS\left(Z_{0,k}^{1,s}\right) \otimes \delta_i^2 \D_iS\left(Z_{0,k}^{2,s}\right) \times P_{l',l}\otimes P_{l,l'} \right) \Big] \\
		&\quad\quad\quad\quad\quad\quad\quad\quad\ - \E\left[\tr_{kN}^{\otimes 2}\left(\delta_i^1 \D_iS\left(Z_{t_s,k}^{1,s}\right) \otimes \delta_i^2 \D_iS\left(Z_{t_s,k}^{2,s}\right) \times P_{l',l}\otimes P_{l,l'} \right)\right].
	\end{align*}
	
	\noindent Thus by using Proposition \ref{2concentration}, we have that
	
	\begin{align}
		\label{limitelambda}
		\Lambda_{N,t_{n+1}} = \lim_{k\to\infty} \frac{1}{k^2N^3}  \sum_{1\leq l,l'\leq k} \sum_{\substack{1\leq s\leq n+1\\ 1\leq j\leq d_s}} \int_0^{t_s} \E\Big[ \tr_{kN}\Big(&\left[\delta_{s,j}\left[\delta_i^1 \D_iS \right] \left(Z_{r,k}^{1,s} \right)\widetilde{\#} P_{l',l}\right] \\
		\boxtimes &\left[\delta_{s,j}\left[\delta_i^2 \D_iS \right] \left(Z_{r,k}^{2,s} \right)\widetilde{\#} P_{l,l'}\right] \Big) \Big]\ dr. \nonumber
	\end{align}
	
	\noindent Besides, we have that for $A,B,C,D\in\F_{d_1+\dots+d_{s+1},q}$, for all $k>0$,
	\begin{align*}
		&\frac{1}{k^2N}\sum_{1\leq l,l'\leq k} \tr_{kN}\left(A(Z_{r,k}^{1,s})  P_{l',l} B(Z_{r,k}^{1,s}) C(Z_{r,k}^{2,s})  P_{l,l'} D(Z_{r,k}^{2,s})\right) \\
		&= \ts_N\left(\id_N\otimes\ts_k\left(D(Z_{r,k}^{2,s})A(Z_{r,k}^{1,s})\right) \id_N\otimes\ts_k\left(B(Z_{r,k}^{1,s}) C(Z_{r,k}^{2,s})\right)\right) \\
		&= \frac{1}{N}\sum_{ 1\leq u,v \leq N} \left(\id_N\otimes\ts_k\left(D(Z_{r,k}^{2,s})A(Z_{r,k}^{1,s})\right)\right)_{u,v} \left(\id_N\otimes\ts_k\left(B(Z_{r,k}^{1,s}) C(Z_{r,k}^{2,s})\right)\right)_{v,u} \\
		&= N\sum_{ 1\leq u,v \leq N} \ts_{kN}\left(D(Z_{r,k}^{2,s})A(Z_{r,k}^{1,s}) \times E_{v,u}\otimes I_k \right) \ts_{kN}\left(B(Z_{r,k}^{1,s}) C(Z_{r,k}^{2,s}) \times E_{u,v}\otimes I_k\right).
	\end{align*}
	
	\noindent But thanks to Proposition \ref{smcwod}, we know that almost surely the family $$\left(Z_{r,k}^{1,s},Z_{r,k}^{2,s},E_{v,u}\otimes I_k, E_{u,v}\otimes I_k\right)$$ converges in distribution towards $\left(z_{r}^{1,s},z_{r}^{2,s},E_{v,u},E_{u,v}\right)$ as $k$ goes to $\infty$. Consequently, we have that almost surely
	\begin{align}
		&\lim_{k\to \infty}\ \frac{1}{k^2N}\sum_{1\leq l,l'\leq k} \tr_{kN}\left(A(Z_{r,k}^{1,s})  P_{l',l} B(Z_{r,k}^{1,s}) C(Z_{r,k}^{2,s})  P_{l,l'} D(Z_{r,k}^{2,s})\right) \\
		&= N\sum_{ 1\leq u,v \leq N} \tau_{N}\left(D(z_{r}^{2,s})A(z_{r}^{1,s}) \times E_{v,u} \right) \tau_{N}\left(B(z_{r}^{1,s}) C(Z_{r}^{2,s}) \times E_{u,v}\right). \nonumber
	\end{align}
	But then, once again thanks to Proposition \ref{smcwod}, with $\widetilde{Z}_{r,k}^{1,s}$ and $\widetilde{Z}_{r,k}^{2,s}$  copies of ${Z}_{r,k}^{1,s}$ and ${Z}_{r,k}^{2,s}$ where we replaced each Brownian motion by an independent copy, we have that
	\begin{align*}
		&\lim_{k\to \infty}\ \frac{1}{k^2N}\sum_{1\leq l,l'\leq k} \tr_{kN}\left(A(Z_{r,k}^{1,s})  P_{l',l} B(Z_{r,k}^{1,s}) C(Z_{r,k}^{2,s})  P_{l,l'} D(Z_{r,k}^{2,s})\right) \\
		&= \lim_{k\to \infty}\ N\sum_{ 1\leq u,v \leq N} \E_k\Big[ \ts_{kN}\left(D(Z_{r,k}^{2,s})A(Z_{r,k}^{1,s}) \times E_{v,u}\otimes I_k \right) \\
		&\quad\quad\quad\quad\quad\quad\quad\quad\quad\quad \ts_{kN}\left(B(\widetilde{Z}_{r,k}^{1,s}) C(\widetilde{Z}_{r,k}^{2,s}) \times E_{u,v}\otimes I_k\right) \Big].
	\end{align*}
	
	\noindent But then thanks to the same argument that we used in Equations \eqref{besoindunlabel1} and \ref{besoindunlabel2}, and combined with the fact that $I_k=\sum_l E_{l,l}$, we get that
	\begin{align*}
		&N\sum_{ 1\leq u,v \leq N} \E_k\left[ \ts_{kN}\left(D(Z_{r,k}^{2,s})A(Z_{r,k}^{1,s}) \times E_{v,u}\otimes I_k \right) \ts_{kN}\left(B(\widetilde{Z}_{r,k}^{1,s}) C(\widetilde{Z}_{r,k}^{2,s}) \times E_{u,v}\otimes I_k\right) \right] \\
		&= \frac{1}{N}\sum_{ 1\leq u,v \leq N} \E_k\Big[ \tr_{kN}\left(D(Z_{r,k}^{2,s})A(Z_{r,k}^{1,s}) \times E_{v,u}\otimes E_{1,1} \right) \\
		&\quad\quad\quad\quad\quad\quad\quad\ \ \tr_{kN}\left(B(\widetilde{Z}_{r,k}^{1,s}) C(\widetilde{Z}_{r,k}^{2,s}) \times E_{u,v}\otimes E_{1,1}\right) \Big].
	\end{align*}
	
	\noindent And similarly we have that
	\begin{align*}
		&\E_k\left[ \ts_{kN}\left(A(Z_{r,k}^{1,s}) B(\widetilde{Z}_{r,k}^{1,s}) C(\widetilde{Z}_{r,k}^{2,s})D(Z_{r,k}^{2,s})\right) \right] \\
		&=\frac{1}{N k}\sum_{\substack{ 1\leq u,v \leq N\\ 1\leq l,l'\leq k}} \E_k\Big[ \tr_{kN}\left(D(Z_{r,k}^{2,s})A(Z_{r,k}^{1,s}) \times E_{v,u}\otimes E_{l',l} \right) \\
		&\quad\quad\quad\quad\quad\quad\quad\quad \tr_{kN}\left(B(\widetilde{Z}_{r,k}^{1,s}) C(\widetilde{Z}_{r,k}^{2,s}) \times E_{u,v}\otimes E_{l,l'}\right) \Big] \\
		&= \frac{1}{N}\sum_{ 1\leq u,v \leq N} \E_k\Big[ \tr_{kN}\left(D(Z_{r,k}^{2,s})A(Z_{r,k}^{1,s}) \times E_{v,u}\otimes E_{1,1} \right) \\
		&\quad\quad\quad\quad\quad\quad\quad\ \ \tr_{kN}\left(B(\widetilde{Z}_{r,k}^{1,s}) C(\widetilde{Z}_{r,k}^{2,s}) \times E_{u,v}\otimes E_{1,1}\right) \Big] \\
		&\quad +  \frac{k-1}{N}\sum_{ 1\leq u,v \leq N} \E_k\Big[ \tr_{kN}\left(D(Z_{r,k}^{2,s})A(Z_{r,k}^{1,s}) \times E_{v,u}\otimes E_{1,2} \right) \\
		&\quad\quad\quad\quad\quad\quad\quad\quad\quad\ \tr_{kN}\left(B(\widetilde{Z}_{r,k}^{1,s}) C(\widetilde{Z}_{r,k}^{2,s}) \times E_{u,v}\otimes E_{1,2}\right) \Big].
	\end{align*}
	
	\noindent Hence it turns out that
	\begin{align*}
		&\lim_{k\to \infty}\ \frac{1}{k^2N}\sum_{1\leq l,l'\leq k} \tr_{kN}\left(A(Z_{r,k}^{1,s})  P_{l',l} B(Z_{r,k}^{1,s}) C(Z_{r,k}^{2,s})  P_{l,l'} D(Z_{r,k}^{2,s})\right) \\
		&= \lim_{k\to \infty}\ \E_k\left[ \ts_{kN}\left(A(Z_{r,k}^{1,s}) B(\widetilde{Z}_{r,k}^{1,s}) C(\widetilde{Z}_{r,k}^{2,s})D(Z_{r,k}^{2,s})\right) \right] \\
		&\quad - \lim_{k\to \infty}\ \frac{k-1}{N}\sum_{ 1\leq u,v \leq N} \E_k\left[ \tr_{kN}\left(D(Z_{r,k}^{2,s})A(Z_{r,k}^{1,s}) \times E_{v,u}\otimes E_{1,2} \right) \right] \\
		&\quad\quad\quad\quad\quad\quad\quad\quad\quad\quad\quad\times \E_k\left[\tr_{kN}\left(B(\widetilde{Z}_{r,k}^{1,s}) C(\widetilde{Z}_{r,k}^{2,s}) \times E_{u,v}\otimes E_{1,2}\right) \right].
	\end{align*}
	
	\noindent Finally, thanks to Lemma \ref{2nondiag} which we use with $K^N=(U^N,Z^N,E_{u,v},E_{v,u})$, we get that the last limit converges towards $0$, and with Proposition \ref{smcwod}, we have in conclusion that
	\begin{align*}
		&\lim_{k\to \infty}\ \frac{1}{k^2N}\sum_{1\leq l,l'\leq k} \tr_{kN}\left(A(Z_{r,k}^{1,s})  P_{l',l} B(Z_{r,k}^{1,s}) C(Z_{r,k}^{2,s})  P_{l,l'} D(Z_{r,k}^{2,s})\right) \\
		&= \tau_{N}\left(A(z_{r}^{1,s}) B(\widetilde{z}_{r}^{1,s}) C(\widetilde{z}_{r}^{2,s})D(z_{r}^{2,s})\right).
	\end{align*}
	
	\noindent Thus by plugging this equality back into Equation \eqref{limitelambda}, we have that
	\begin{align}
		\Lambda_{N,t_{n+1}} = \frac{1}{N^2} \sum_{\substack{1\leq s\leq n+1\\ 1\leq j\leq d_s}} \int_0^{t_s} \E\Big[ \tau_{N}\Big(&\left[\Big(\delta_{s,j}^2\left[\delta_i^1 \D_iS \right] \left(z_{r}^{1,s} \right)\Big)\boxtimes \Big( \left(\delta_{s,j}^1\left[\delta_i^1 \D_iS \right] \left(\widetilde{z}_{r}^{1,s} \right)\right)\Big)\right] \\
		\boxtimes &\left[\Big(\delta_{s,j}^2\left[\delta_i^2 \D_iS \right] \left(\widetilde{z}_{r}^{2,s} \right)\Big)\boxtimes \Big(\delta_{s,j}^1\left[\delta_i^2 \D_iS \right] \left(z_{r}^{2,s} \right)\Big)\right] \Big) \Big]\ dr. \nonumber
	\end{align}
	
	\noindent And finally we can plug this equality back into Equation \eqref{compderivat}, and we get that
	\begin{align*}
		&\E\left[\tau_N\Big(Q\left(u_{t_1}^1,\dots, u_{t_n}^n,U^N\right)\Big)\right] - \E\left[\tau_N\Big(Q\left( u_{t_1}^1,\dots, u_{t_n}^n, u_T^{n+1} U^N\right)\Big)\right] \nonumber \\
		&= \frac{1}{2N^2}\sum_{\substack{1\leq s\leq n+1\\ 1\leq j\leq d_s \\ 1\leq i\leq d}}\int_{0}^{T}\int_0^{t_s} \E\Big[ \tau_{N}\Big(\left[\Big(\delta_{s,j}^2\left[\delta_i^1 \D_iS \right] \left(z_{r}^{1,s} \right)\Big)\boxtimes \Big( \left(\delta_{s,j}^1\left[\delta_i^1 \D_iS \right] \left(\widetilde{z}_{r}^{1,s} \right)\right)\Big)\right] \\
	\end{align*}
	
\end{proof}

\subsection{Proof of Theorem \ref{3lessopti}}

\label{3technical}

In this section we focus on proving Theorem \ref{3lessopti} from which we deduce all of the important corollaries. It will mainly be a corollary of the following theorem, which is slightly stronger but less explicit. We refer to  Lemma \ref{3apparition} for the definition of $L^{T_{i}}$ and $u^{T_i}$, and to Proposition \ref{3intercoef} for the one of $A_i$. To fully understand how the coefficients $\alpha_i^P(f,Z^N)$ are built we also refer to those propositions.

\begin{theorem}
	\label{3TTheo}
	Let the following objects be given,
	\begin{itemize}
		\item $U^N = (U_1^N,\dots,U_d^N)$ independent Haar unitary matrices of size $N$,
		\item $Z^N = (Z_1^N,\dots,Z_q^N, {Z_1^N}^*,\dots,{Z_q^N}^*)$ deterministic matrices of size $N$ and their adjoints,
		\item $P\in \A_{d,q}$ a polynomial that we assume to be self-adjoint,
		\item $f:\R\mapsto\R$ such that there exists a complex-valued measure on the real line $\mu$ with $$\int (1+|y|^{4k+5})\  d|\mu|(y)\ < +\infty,$$ and for any $x\in\R$,
		\begin{equation}
			\label{3hypoth}
			f(x) = \int_{\R} e^{\i x y}\ d\mu(y) .
		\end{equation}
		
	\end{itemize}
	
	\noindent Then with notations as in Lemma \ref{3apparition} and Proposition \ref{3intercoef} if we set,
	\begin{align}
		\label{exprdescoeff}
		\alpha_i^P(f,Z^N) = \int_{\R} \int_{A_i} \int_{[0,1]^{4i}} \tau_N\Big( \left(L^{{T}_i}_{\lambda_i,\beta_i,\gamma_i,\delta_i} \dots L^{{T}_1}_{\lambda_1,\beta_1,\gamma_1,\delta_1}\right)&(e^{\i y P}) (u^{{T}_i},Z^N) \Big)\\
		&  d\lambda d\beta d\gamma d\delta\ dt\ d\mu(y), \nonumber
	\end{align} 
	
	\noindent and that we write $P = \sum_{1\leq i\leq \Nb(P)} c_i M_i$ where the $M_i$ are monomials and $c_i\in\C$ (i.e. $P$ is a a sum of at most $\Nb(P)$ monomials), if we set $C_{\max}(P) = \max \{1, \max_i |c_i|\}$, then there exists a constant $C$ independent of $P$ such that with $K_N = \max \{ \norm{Z^N_1}$ $,\dots, \norm{Z^N_q}, 1\}$, for any $N$ and $k$,
	\begin{align}
		\label{3mainresu}
		&\left| \E\left[ \tau_{N}\Big(f(P(U^N,{U^N}^*,Z^N))\Big)\right] - \sum_{0\leq i\leq k} \frac{1}{N^{2i}} \alpha_i^P(f,Z^N) \right| \\
		&\leq \frac{1}{N^{2k+2}} \int_{\R} |y|(1+ y^{4(k+1)}) d|\mu|(y) \nonumber \\
		&\quad\quad\quad\quad\quad\times \Big(C\times K_N^{\deg P+1} C_{\max}(P) \Nb(P)(\deg P)(\deg P+1)\Big)^{4k+5}\times k^{6k}  . \nonumber
	\end{align}
	
	\noindent Besides, we also have that for any $j\in\N^*$,
	\begin{align}
		\label{3mainresu2}
		&\left| \alpha_j^P(f,Z^N) \right| \leq \int_{\R} |y|(1+ y^{4j}) d|\mu|(y) \\
		&\quad\quad\quad\quad\quad\quad\quad  \times \Big(C\times K_N^{\deg P} C_{\max}(P) \Nb(P)(\deg P)(\deg P+1)\Big)^{4j+1}\times j^{6j} . \nonumber
	\end{align}
	Finally, if $f$ and $g$ both satisfy \eqref{3hypoth} for some complex measures $\mu_f$ and $\mu_g$, then if they are bounded functions equal on a neighborhood of the spectrum of $P(u,u^*,Z^N)$, where $u$ is a $d$-tuple of free Haar unitaries free from $\M_N(\C)$, then for any $i$, $\alpha_i^P(f,Z^N) = \alpha_i^P(g,Z^N)$. In particular if $f$ is a bounded function such that its support and the spectrum of $P(u,u^*,Z^N)$ are disjoint, then for any $i$, $\alpha_i^P(f,Z^N)=0$.
	
\end{theorem}

Note that it is quite important to allow $\mu$ to be a complex-valued measure. Indeed, this means that one can use the Fourier inversion theorem and thus consider pretty much any functions smooth enough, as we will see in the proof of Theorem \ref{3lessopti}.

\begin{rem}
	\label{produf}
	It is worth noting that if one wanted, one could consider a product of functions $f_i$ evaluated in self-adjoint polynomials $P_i\in\PP_{d,q}$ instead of a single function $f$ evaluated in $P$. Indeed, the proof of Theorem \ref{3TTheo} consists in first using Proposition \ref{3intercoef} and then estimating the remainder term. However, Proposition \ref{3intercoef} can be used in more general situations. If we assume that for any $i$ and $x\in\R$,
	\begin{equation*}
		f_i(x) = \int_{\R} e^{\i x y}\ d\mu_i(y),
	\end{equation*}
	for some complex-valued measure $\mu_i$.  Then given $R_i\in\PP_{d,q}, y_i\in\R$, 
	$$Q= e^{\i y_1 P_1}R_1\dots e^{\i y_k P_k}R_k$$ belongs to $\F_{d,q}$. Consequently, one can apply Proposition \ref{3intercoef} to $Q$ and since 
	\begin{align*}
		&\E\left[ \ts_N\Big(f_1(P_1(U^N,Z^N))R_1(U^N,Z^N)\dots f_k(P_k(U^N,Z^N))R_k(U^N,Z^N)\Big) \right] \\
		&= \int_{\R^k} \E\left[ \ts_N\Big(Q(U^N,Z^N)\Big) \right] d\mu_1(y_1)\dots d\mu_k(y_k),
	\end{align*}
	one can obtain an asymptotic expansion for any products of smooth functions.
	
	One can also study the case where we have a product of traces, to do so we use the Schwinger-Dyson equations to reduce the problem to the case of a single trace. Given matrices $A,B\in\M_N(\C)$, one has thanks to Equation (5.4.29) of \cite{alice}, that with $V$ a Haar unitary matrix of size $N$,
	$$ \ts_{N}(A) \ts_{N}(B) = \E\left[ \ts_{N}\left(V^*AVB\right) \right]. $$
	Consequently, given $Q_1,\dots,Q_k\in\F_{d,q}$, $V_1^N,\dots,V_{k-1}^N$ independent Haar unitary matrices, independent from $U^N$, one has that
	\begin{align*}
		&\E\left[ \ts_N\Big(Q_1(U^N,Z^N))\Big)\dots \ts_N\Big(Q_k(U^N,Z^N)\Big) \right] \\
		&= \E\left[ \ts_N\Big(V_{k-1}^N\dots V_1^N Q_1(U^N,Z^N))V^N_1Q_2(U^N,Z^N)) \dots V_{k-1}^N Q_k(U^N,Z^N)\Big) \right].
	\end{align*}
	Hence once again one can use Proposition \ref{3intercoef} to get an asymptotic expansion.
\end{rem}

The following lemma is the first step of the proof of Theorem \ref{3TTheo} and allows us to define the coefficients of the topological expansion by induction. It is basically a reformulation of Lemma \ref{3imp2} with the notations of Definitions \ref{3biz} and \ref{3biz2}. Although the notations in this formula are a bit heavy, they are necessary in order to get a better upper bound on the remainder term.

\begin{lemma}
	\label{3apparition}
	Let $(v_t)_{t\geq 0},(u^1_t)_{t\geq 0},\dots,(u^{c_{n}}_t)_{t\geq 0}$ be families of $d$ free unitary Brownian motions, free between each other and free from $u$. Then with $T_n = \{t_1,\dots,t_{2n}\}$ a sequence of non-negative number,  $\{\widetilde{t}_1,\dots,\widetilde{t}_{2n}\}$ the same set but ordered by increasing order, and $I = \{I_1,\dots,I_{2n}\}\in J_n$, with $t_0=0$ and $t^*\geq \widetilde{t}_{2n}$, we set
	$$ U_{i,I}^{N,T_{n}} = \left(\prod_{l=1}^{2n} u^{I_{l}}_{i,\widetilde{t}_l-\widetilde{t}_{l-1}}\right) U_i^N , $$
	$$ u_{i,I}^{T_{n},t^*} = \left(\prod_{l=1}^{2n} u^{I_{l}}_{i,\widetilde{t}_l-\widetilde{t}_{l-1}}\right) v_{i,t^*-\widetilde{t}_{2n}} U_i^N. $$
	
	\noindent We define for $s\in [1,2n+1]$ the following subfamilies of $(U_{i,I})_{i\in [1,d], I\in J_{n+1}}$ (the variables of $\A^{n+1}_{d,q}$ defined in the first bullet point of Definition \ref{sidcosc}),
	$$U_{s,1} = \left(U_{i,I}\right)_{i\in [1,d], I\in F_{n+1}^{s,1}(J_n)}, U_{s,2} = \left(U_{i,I}\right)_{i\in [1,d],I\in F_{n+1}^{s,2}(J_n)}, $$
	$$\widetilde{U}_{s,1} = \left(U_{i,I}\right)_{i\in [1,d],I\in \widetilde{F}_{n+1}^{s,1}(J_n)}, \widetilde{U}_{s,2} = \left(U_{i,I}\right)_{i\in [1,d],I\in \widetilde{F}_{n+1}^{s,2}(J_n)}.$$
	
	\noindent One defines similarly $V_{s,1},V_{s,2},\widetilde{V}_{s,1}$ and $\widetilde{V}_{s,2}$. Since by construction there is a bijection between $J_n$ and $F_{n+1}^{s,1}(J_n)$ (see Definition \ref{3biz}), one can evaluate an element of $\F_{d,q}^{n}$ in $X_{s,1}=(U_{s,1},V_{s,1},Z,Y)$ where $Z=(Z_1,\dots,Z_{q})$ and $Y=(Y_1,\dots,Y_{q})$ as in Definition \ref{sidcosc}, and similarly for $X_{s,2},\widetilde{X}_{s,1}$ and $\widetilde{X}_{s,2}$. 
	Then we define the following operators (with the help of Definition \ref{3biz2}) from $\G_{d,q}^{n}$ to $\G_{d,q}^{n+1}$, for $s$ from $1$ to $2n+1$,
	\begin{align*}
		&L^{n,s}_{\rho_{n+1},\beta_{n+1},\gamma_{n+1},\delta_{n+1}}(Q) := \frac{1}{2} \sum_{1\leq i,j\leq d}\ \sum_{\substack{I,J\in J_n \text{ such that } \\ \forall l\geq s, I_l=J_l}} \\
		&\quad\quad\quad\quad \Big(\delta_{\delta_{n+1},j,I}^2\left( \delta_{\beta_{n+1},i}^1 \D_{\rho_{n+1},i} Q\right)\left(X_{s,1}\right) \boxtimes \delta_{\delta_{n+1},j,I}^1 \left( \delta_{\beta_{n+1},i}^1 \D_{\rho_{n+1},i} Q\right)(\widetilde{X}_{s,1}) \Big) \\
		&\quad\quad\quad \boxtimes \Big(\delta_{\gamma_{n+1},j,J}^2\left( \delta_{\beta_{n+1},i}^2 \D_{\rho_{n+1},i} Q\right)(\widetilde{X}_{s,2}) \boxtimes \delta_{\gamma_{n+1},j,J}^1 \left( \delta_{\beta_{n+1},i}^2 \D_{\rho_{n+1},i} Q\right)\left(X_{s,2}\right) \Big). \\
	\end{align*}
	\noindent Note that since $I\in J_n$ only has $2n$ elements, the condition ``$I,J\in J_n$, such that $\forall l\geq 2n+1, I_l=J_l$'' is satisfied for any $I,J$. Finally,, if $T_{n+1}$ is a set of $2n+2$ numbers, with $\widetilde{T}_{n} = \{\widetilde{t}_1,\dots,\widetilde{t}_{2n}\}$ the set which contains the first $2n$ elements of $T_{n+1}$ but  sorted by increasing order, we set 
	\begin{align}
		\label{3fullop}
		L^{T_{n+1}}_{\rho_{n+1},\beta_{n+1},\gamma_{n+1},\delta_{n+1}}(Q) := & \1_{[\widetilde{t}_{2n},t_{2n+2}]}(t_{2n+1}) L^{n,2n+1}_{\rho_{n+1},\beta_{n+1},\gamma_{n+1},\delta_{n+1}}(Q) \\
		&+ \sum_{1\leq s\leq 2n} \1_{[\widetilde{t}_{s-1},\widetilde{t}_s]}(t_{2n+1}) L^{n,s}_{\rho_{n+1},\beta_{n+1},\gamma_{n+1},\delta_{n+1}}(Q). \nonumber
	\end{align}
	
	\noindent Then, given $Q\in \G_{d,q}^{n}$, for $t^*\geq \widetilde{t}_{2n}$,
	
	\begin{align*}
		&\E\left[\tau_N\Big(Q(U^{N,T_n},Z^N)\Big)\right] - \E\left[\tau_N\Big(Q(u^{T_n,t^*},Z^N)\Big)\right]  \\
		&= \frac{1}{N^2}\int_{\widetilde{t}_{2n}}^{t^*} \int_0^{t_{2n+2}} \int_{[0,1]^4}\tau_N\left( L^{T_{n+1}}_{\rho_{n+1},\beta_{n+1},\gamma_{n+1},\delta_{n+1}}(Q)\left( U^{N,T_{n+1}},Z^N \right) \right)\\
		&\quad\quad\quad\quad\quad\quad\quad\quad\quad\quad\quad\quad\quad\quad\quad\quad\quad\quad\quad d\rho_{n+1} d\beta_{n+1} d\gamma_{n+1} d\delta_{n+1}\ dt_{2n+1} dt_{2n+2}.
	\end{align*}

\end{lemma}

\begin{proof}	
	
	We want to apply Lemma \ref{3imp2} with well-chosen families of free unitary Brownian motions. To avoid a conflict of notation, we will apply Lemma \ref{3imp2} with the family $u^s$ replaced by $y^s$ defined as follows,
	$$y^s = \left\{u^k_{\widetilde{t}_s-\widetilde{t}_{s-1}}\ \middle|\ \dep^{n}(k)=s \right\}. $$
	Let $R\in\F_{d(1+c_n), q}$ be such that 
	$$ Q\left(U^{N,T_n},Z^N\right) = R\left(y_1,\dots,y_{2n},U^N,Z^N\right).$$
	Then with notations as in Lemma \ref{3imp2}, we have after renaming $t_{2n+1}$ into $t$,
	\begin{align}
		\label{oisjclakmsc}
		&\E\left[\tau_N\Big(R\left(y_1,\dots,y_{2n},U^N,Z^N\right)\Big)\right] - \E\left[\tau_N\Big(R\left(y_1,\dots,y_{2n},v_{t^*-\widetilde{t}_{2n}}U^N,Z^N\right)\Big)\right]  \\
		&= \frac{1}{2N^2}\sum_{\substack{1\leq s\leq 2n\\ 1\leq j\leq d_s \\ 1\leq i\leq d}}\int_{0}^{t^*-\widetilde{t}_{2n}}\int_0^{\widetilde{t}_s-\widetilde{t}_{s-1}} \E\Big[ \tau_{N}\Big(\left[\Big(\delta_{s,j}^2\left[\delta_i^1 \D_iS \right] \left(z_{r}^{1,s} \right)\Big)\boxtimes \Big( \delta_{s,j}^1\left[\delta_i^1 \D_iS \right] \left(\widetilde{z}_{r}^{1,s} \right)\Big)\right] \nonumber \\
		&\quad\quad\quad\quad\quad\quad\quad\quad\quad\quad\quad\quad \boxtimes \left[\Big(\delta_{s,j}^2\left[\delta_i^2 \D_iS \right] \left(\widetilde{z}_{r}^{2,s} \right)\Big)\boxtimes \Big(\delta_{s,j}^1\left[\delta_i^2 \D_iS \right] \left(z_{r}^{2,s} \right)\Big)\right] \Big) \Big]\ dr\ dt \nonumber \\
		&+\frac{1}{2N^2}\sum_{1\leq i,j\leq d}\int_{0}^{t^*-\widetilde{t}_{2n}}\int_0^{t} \nonumber \\
		&\quad\quad\quad\quad\quad \E\Big[ \tau_{N}\Big(\left[\Big(\delta_{2n+1,j}^2\left[\delta_i^1 \D_iS \right] \left(z_{r}^{1,2n+1} \right)\Big)\boxtimes \Big( \delta_{2n+1,j}^1\left[\delta_i^1 \D_iQ \right] \left(\widetilde{z}_{r}^{1,2n+1} \right)\Big)\right] \nonumber\\
		&\quad\quad\quad\quad\quad \boxtimes \left[\Big(\delta_{2n+1,j}^2\left[\delta_i^2 \D_iS \right] \left(\widetilde{z}_{r}^{2,2n+1} \right)\Big)\boxtimes \Big(\delta_{2n+1,j}^1\left[\delta_i^2 \D_iQ \right] \left(z_{r}^{2,2n+1} \right)\Big)\right] \Big) \Big]\ dr\ dt. \nonumber
	\end{align}
	
	\noindent Then with $\delta_i\D_iQ$ as in Definition \ref{3biz2} and $\delta_i\D_iS$ as in Lemma \ref{3imp2}, we have that 
	$$ \delta_i\D_iQ\left(u^{T_n,t+\widetilde{t}_{2n}},Z^N\right) = \delta_i\D_iS\left(y_1,\dots,y_{2n},v_t,U^N,Z^N\right).$$
	
	\noindent Besides, for $s\in [1,2n+1]$ and  $I\in J_{i_0,\dots,i_{n-1}}$, we define
	$$ I^{1,s}=F_{n+1}^{s,1}(I),\quad I^{2,s}=F_{n+1}^{s,2}(I),\quad  \widetilde{I}^{1,s}=\widetilde{F}_{n+1}^{s,1}(I),\quad \widetilde{I}^{2,s}=\widetilde{F}_{n+1}^{s,2}(I), $$
	as in Definition \ref{3biz}. Given $j$ such that $\delta_{s,j}$ is the differential with respect to $u^{I_{s}}_{i,\widetilde{t}_s-\widetilde{t}_{s-1}}$ (or to $v_{i,t^*}$ in the case where $s=2n+1$), we have with the notations of Lemma \ref{3imp2} that for $s\leq 2n$,
	
	\begin{align*}
		&\left(\delta_{s,j}^1 U_{i,I}\right) (\widetilde{z}_{r}^{1,s}) \otimes \left(\delta_{s,j}^2 U_{i,I}\right) (z_{r}^{1,s}) \\
		&= \left(\prod_{l=1}^{s-1} u^{\widetilde{I}_{l}^{1,s}}_{i,\widetilde{t}_l-\widetilde{t}_{l-1}}\right) u^{\widetilde{I}_{s}^{1,s}}_{i,r} u^{\widetilde{I}_{s+1}^{1,s}}_{i,\widetilde{t}_s-\widetilde{t}_{s-1}-r} \otimes \left(\prod_{l=s+1}^{2n} u^{I_{l+1}^{1,s}}_{i,\widetilde{t}_l-\widetilde{t}_{l-1}}\right) u^{I_{2n+2}^{1,s}}_{i,t} U_i^N ,
	\end{align*}
	
	\begin{align*}
		&\left(\delta_{s,j}^1 U_{i,I}\right) (z_{r}^{2,s}) \otimes \left(\delta_{s,j}^2 U_{i,I}\right) (\widetilde{z}_{r}^{2,s}) \\
		&= \left(\prod_{l=1}^{s-1} u^{I_{l}^{2,s}}_{i,\widetilde{t}_l-\widetilde{t}_{l-1}}\right) u^{I_{s}^{2,s}}_{i,r} u^{I_{s+1}^{2,s}}_{i,\widetilde{t}_s-\widetilde{t}_{s-1}-r} \otimes \left(\prod_{l=s+1}^{2n} u^{\widetilde{I}_{l+1}^{2,s}}_{i,\widetilde{t}_l-\widetilde{t}_{l-1}}\right) u^{\widetilde{I}_{2n+2}^{2,s}}_{i,t} U_i^N,
	\end{align*}
	
	\begin{align*}
		&\left(\delta_{s,j}^1 U_{i,I}^*\right) (\widetilde{z}_{r}^{1,s}) \otimes \left(\delta_{s,j}^2 U_{i,I}^*\right) (z_{r}^{1,s}) \\
		&= - \left(U_i^N\right)^* \left(u^{\widetilde{I}_{2n+2}^{1,s}}_{i,t}\right)^* \left(\prod_{l=s+1}^{2n} u^{\widetilde{I}_{l+1}^{1,s}}_{i,\widetilde{t}_l-\widetilde{t}_{l-1}}\right)^* \otimes \left(u^{I_{s+1}^{1,s}}_{i,\widetilde{t}_s-\widetilde{t}_{s-1}-r}\right)^* \left(u^{I_{s}^{1,s}}_{i,r}\right)^* \left(\prod_{l=1}^{s-1} u^{I_{l}^{1,s}}_{i,\widetilde{t}_l-\widetilde{t}_{l-1}}\right)^*,
	\end{align*}
	
	\begin{align*}
		&\left(\delta_{s,j}^1 U_{i,I}^*\right) (z_{r}^{2,s}) \otimes \left(\delta_{s,j}^2 U_{i,I}^*\right) (\widetilde{z}_{r}^{2,s}) \\
		&= - \left(U_i^N\right)^* \left(u^{I_{2n+2}^{2,s}}_{i,t}\right)^* \left(\prod_{l=s+1}^{2n} u^{I_{l+1}^{2,s}}_{i,\widetilde{t}_l-\widetilde{t}_{l-1}}\right)^* \otimes \left(u^{\widetilde{I}_{s+1}^{2,s}}_{i,\widetilde{t}_s-\widetilde{t}_{s-1}-r}\right)^* \left(u^{\widetilde{I}_{s}^{2,s}}_{i,r}\right)^* \left(\prod_{l=1}^{s-1} u^{\widetilde{I}_{l}^{2,s}}_{i,\widetilde{t}_l-\widetilde{t}_{l-1}}\right)^*.
	\end{align*}
	
	\noindent And for $s=2n+1$, we have that
	$$ \left(\delta_{2n+1,j}^1 U_{i,I}\right) (\widetilde{z}_{r}^{1,2n+1}) \otimes \left(\delta_{2n+1,j}^2 U_{i,I}\right) (z_{r}^{1,2n+1}) = \left(\prod_{l=1}^{2n} u^{\widetilde{I}_{l}^{1,2n+1}}_{i,\widetilde{t}_l-\widetilde{t}_{l-1}}\right) u^{\widetilde{I}_{2n+1}^{1,2n+1}}_{i,r} u^{\widetilde{I}_{2n+2}^{1,2n+1}}_{i,t-r} \otimes U_i^N , $$
	
	$$ \left(\delta_{2n+1,j}^1 U_{i,I}\right) (z_{r}^{2,2n+1}) \otimes \left(\delta_{2n+1,j}^2 U_{i,I}\right) (\widetilde{z}_{r}^{2,2n+1}) = \left(\prod_{l=1}^{2n} u^{I_{l}^{2,2n+1}}_{i, \widetilde{t}_l-\widetilde{t}_{l-1}}\right) u^{I_{2n+1}^{2,2n+1}}_{i,r} u^{I_{2n+2}^{2,2n+1}}_{i,t-r} \otimes U_i^N . $$
	
	\begin{align*}
		&\left(\delta_{2n+1,j}^1 U_{i,I}^*\right) (\widetilde{z}_{r}^{1,2n+1}) \otimes \left(\delta_{2n+1,j}^2 U_{i,I}^*\right) (z_{r}^{1,2n+1}) \\
		&= - (U_i^N)^* \otimes \left( u^{I_{2n+2}^{1,2n+1}}_{i,t-r}\right)^* \left(u^{I_{2n+1}^{1,2n+1}}_{i,r}\right)^* \left(\prod_{l=1}^{2n} u^{I_{l}^{1,2n+1}}_{i,\widetilde{t}_l-\widetilde{t}_{l-1}}\right)^*,
	\end{align*}
	
	\begin{align*}
		&\left(\delta_{2n+1,j}^1 U_{i,I}^*\right) (z_{r}^{2,2n+1}) \otimes \left(\delta_{2n+1,j}^2 U_{i,I}^*\right) (\widetilde{z}_{r}^{2,2n+1}) \\
		&= - \left(U_i^N\right)^* \otimes \left(u^{\widetilde{I}_{2n+2}^{2,2n+1}}_{i,t-r}\right)^* \left(u^{\widetilde{I}_{2n+1}^{2,2n+1}}_{i,r}\right)^* \left(\prod_{l=1}^{2n} u^{\widetilde{I}_{l}^{2,2n+1}}_{i,\widetilde{t}_l-\widetilde{t}_{l-1}}\right)^*.
	\end{align*}
	
	\noindent Let us now take $A,B\in\A_N$ where $\A_N$ is defined as in \ref{3tra}. If $s\leq 2n$, let $I,J\in  J_{i_0,\dots,i_{n-1}}$ such that $I_s=J_s$. Hence by construction, we have that 
	$$I^1_{s+1}= I_s=J_s =J^2_{s+1},$$
	$$\widetilde{I}^1_{s+1}= I_s+3c_n+3=J_s+3c_n+3 =\widetilde{J}^2_{s+1}.$$
	Then thanks to Lemma \ref{3detail}, we have that for every $l>s$, $I^1_l=J^2_l$ and $\widetilde{I}^1_l=\widetilde{J}^2_l$, consequently,
	\begin{align}
		\label{dojs}
		&\tau_N\left(\left(\delta_{s,j}^1 U_{i,I}\right) (\widetilde{z}_{r}^{1,s}) \otimes \left(\delta_{s,j}^2 U_{i,I}\right) (z_{r}^{1,s}) \widetilde{\#} A \times \left(\delta_{s,j}^1 U_{i,J}\right) (z_{r}^{2,s}) \otimes \left(\delta_{s,j}^2 U_{i,J}\right) (\widetilde{z}_{r}^{2,s}) \widetilde{\#} B\right) \nonumber \\
		&= \tau_N\Bigg(\left(\prod_{l=s+1}^{2n} u^{I_{l+1}^{1,s}}_{i,\widetilde{t}_l-\widetilde{t}_{l-1}}\right) u^{I_{2n+2}^{1,s}}_{i,t} U_i^N A \left(\prod_{l=1}^{s-1} u^{\widetilde{I}_{l}^{1,s}}_{i,\widetilde{t}_l-\widetilde{t}_{l-1}}\right) u^{\widetilde{I}_{s}^{1,s}}_{i,r} u^{\widetilde{I}_{s+1}^{1,s}}_{i,\widetilde{t}_s-\widetilde{t}_{s-1}-r} \\
		&\quad\quad\quad\times \left(\prod_{l=s+1}^{2n} u^{\widetilde{J}_{l+1}^{2,s}}_{i,\widetilde{t}_l-\widetilde{t}_{l-1}}\right) u^{\widetilde{J}_{2n+2}^{2,s}}_{i,t} U_i^N B \left(\prod_{l=1}^{s-1} u^{J_{l}^{2,s}}_{i,\widetilde{t}_l-\widetilde{t}_{l-1}}\right) u^{J_{s}^{2,s}}_{i,r} u^{J_{s+1}^{2,s}}_{i,\widetilde{t}_s-\widetilde{t}_{s-1}-r} \Bigg) \nonumber \\
		&= \tau_N\left( A\ U_{i,\widetilde{I}^{1,s}}^{N,\{T_{n},r+\widetilde{t}_{s-1},t+\widetilde{t}_{2n}\}} B\ U_{i,J^{2,s}}^{N,\{T_{n},r+\widetilde{t}_{s-1},t+\widetilde{t}_{2n}\}} \right). \nonumber
	\end{align}
	
	\begin{align}
		\label{dojs2}
		&\tau_N\left(\left(\delta_{s,j}^1 U_{i,I}\right) (\widetilde{z}_{r}^{1,s}) \otimes \left(\delta_{s,j}^2 U_{i,I}\right) (z_{r}^{1,s}) \widetilde{\#} A \times \left(\delta_{s,j}^1 U_{i,I}^*\right) (z_{r}^{2,s}) \otimes \left(\delta_{s,j}^2 U_{i,I}^*\right) (\widetilde{z}_{r}^{2,s}) \widetilde{\#} B\right) \nonumber\\
		&= - \tau_N\Bigg(\left(\prod_{l=s+1}^{2n} u^{I_{l+1}^{1,s}}_{i,\widetilde{t}_l-\widetilde{t}_{l-1}}\right) u^{I_{2n+2}^{1,s}}_{i,t} U_i^N A \left(\prod_{l=1}^{s-1} u^{\widetilde{I}_{l}^{1,s}}_{i,\widetilde{t}_l-\widetilde{t}_{l-1}}\right) u^{\widetilde{I}_{s}^{1,s}}_{i,r} u^{\widetilde{I}_{s+1}^{1,s}}_{i,\widetilde{t}_s-\widetilde{t}_{s-1}-r} \\
		&\times \left(u^{\widetilde{J}_{s+1}^{2,s}}_{i,\widetilde{t}_s-\widetilde{t}_{s-1}-r}\right)^* \left(u^{\widetilde{J}_{s}^{2,s}}_{i,r}\right)^* \left(\prod_{l=1}^{s-1} u^{\widetilde{J}_{l}^{2,s}}_{i,\widetilde{t}_l-\widetilde{t}_{l-1}}\right)^* B \left(U_i^N\right)^* \left(u^{J_{2n+2}^{2,s}}_{i,t}\right)^* \left(\prod_{l=s+1}^{2n} u^{J_{l+1}^{2,s}}_{i,\widetilde{t}_l-\widetilde{t}_{l-1}}\right)^* \Bigg) \nonumber\\
		&= -\tau_N\left( A\ U_{i,\widetilde{I}^{1,s}}^{N,\{T_{n},r+\widetilde{t}_{s-1},t+\widetilde{t}_{2n}\}} \left(U_{i,\widetilde{J}^{2,s}}^{N,\{T_{n},r+\widetilde{t}_{s-1},t+\widetilde{t}_{2n}\}}\right)^* B \right). \nonumber
	\end{align}
	
	\noindent And similarly,
	\begin{align}
		\label{dojs3}
		&\tau_N\left(\left(\delta_{s,j}^1 U_{i,I}^*\right) (\widetilde{z}_{r}^{1,s}) \otimes \left(\delta_{s,j}^2 U_{i,I}^*\right) (z_{r}^{1,s}) \widetilde{\#} A \times \left(\delta_{s,j}^1 U_{i,I}^*\right) (z_{r}^{2,s}) \otimes \left(\delta_{s,j}^2 U_{i,I}^*\right) (\widetilde{z}_{r}^{2,s}) \widetilde{\#} B\right) \\
		&= -\tau_N\left( \left(U_{i,I^{1,s}}^{N,\{T_{n},r+\widetilde{t}_{s-1},t+\widetilde{t}_{2n}\}}\right)^* A\ \left(U_{i,\widetilde{J}^{2,s}}^{N,\{T_{n},r+\widetilde{t}_{s-1},t+\widetilde{t}_{2n}\}}\right)^* B \right). \nonumber
	\end{align}
	\begin{align}
		\label{dojs4}
		&\tau_N\left(\left(\delta_{s,j}^1 U_{i,I}^*\right) (\widetilde{z}_{r}^{1,s}) \otimes \left(\delta_{s,j}^2 U_{i,I}^*\right) (z_{r}^{1,s}) \widetilde{\#} A \times \left(\delta_{s,j}^1 U_{i,J}\right) (z_{r}^{2,s}) \otimes \left(\delta_{s,j}^2 U_{i,J}\right) (\widetilde{z}_{r}^{2,s}) \widetilde{\#} B\right) \\
		&= -\tau_N\left( \left(U_{i,I^{1,s}}^{N,\{T_{n},r+\widetilde{t}_{s-1},t+\widetilde{t}_{2n}\}}\right)^* A B \ U_{i,J^{2,s}}^{N,\{T_{n},r+\widetilde{t}_{s-1},t+\widetilde{t}_{2n}\}} \right). \nonumber
	\end{align}
	
	\noindent The case where $s=2n+1$ also gives the same formula. Besides, we also have the following formulas,
	
	\begin{align*}
		U_{i,I}(z_{r}^{1,s}) &= \left(\prod_{l=1}^{s-1} u^{I_{l}^{1,s}}_{i,\widetilde{t}_l-\widetilde{t}_{l-1}}\right) u^{I_{s}^{1,s}}_{i,r} u^{I_{s+1}^{1,s}}_{i,\widetilde{t}_s-\widetilde{t}_{s-1}-r} \left(\prod_{l=s+1}^{2n} u^{I_{l+1}^{1,s}}_{i,\widetilde{t}_l-\widetilde{t}_{l-1}}\right) u^{I_{2n+2}^{1,s}}_{i,t} U_i^N \\
		&= U_{i,I^{1,s}}^{N,\{T_{n},r+\widetilde{t}_{s-1},t+\widetilde{t}_{2n}\}},
	\end{align*}
	
	\begin{align*}
		U_{i,I}(z_{r}^{2,s}) &= \left(\prod_{l=1}^{s-1} u^{I_{l}^{2,s}}_{i,\widetilde{t}_l-\widetilde{t}_{l-1}}\right) u^{I_{s}^{2,s}}_{i,r} u^{I_{s+1}^{2,s}}_{i,\widetilde{t}_s-\widetilde{t}_{s-1}-r} \left(\prod_{l=s+1}^{2n} u^{I_{l+1}^{2,s}}_{i,\widetilde{t}_l-\widetilde{t}_{l-1}}\right) u^{I_{2n+2}^{2,s}}_{i,t} U_i^N  \\ 
		&= U_{i,I^{2,s}}^{N,\{T_{n},r+\widetilde{t}_{s-1},t+\widetilde{t}_{2n}\}},	
	\end{align*}
	
	\begin{align*}
		U_{i,I}(\widetilde{z}_{r}^{1,s}) &= \left(\prod_{l=1}^{s-1} u^{\widetilde{I}_{l}^{1,s}}_{i,\widetilde{t}_l-\widetilde{t}_{l-1}}\right) u^{\widetilde{I}_{s}^{1,s}}_{i,r} u^{\widetilde{I}_{s+1}^{1,s}}_{i,\widetilde{t}_s-\widetilde{t}_{s-1}-r} \left(\prod_{l=s+1}^{2n} u^{\widetilde{I}_{l+1}^{1,s}}_{i,\widetilde{t}_l-\widetilde{t}_{l-1}}\right) u^{\widetilde{I}_{2n+2}^{1,s}}_{i,t} U_i^N  \\
		&= U_{i,\widetilde{I}^{1,s}}^{N,\{T_{n},r+\widetilde{t}_{s-1},t+\widetilde{t}_{2n}\}},
	\end{align*}
	
	\begin{align*}
		U_{i,I}(\widetilde{z}_{r}^{2,s}) &= \left(\prod_{l=1}^{s-1} u^{\widetilde{I}_{l}^{2,s}}_{i,\widetilde{t}_l-\widetilde{t}_{l-1}}\right) u^{\widetilde{I}_{s}^{2,s}}_{i,r} u^{\widetilde{I}_{s+1}^{2,s}}_{i,\widetilde{t}_s-\widetilde{t}_{s-1}-r} \left(\prod_{l=s+1}^{2n} u^{\widetilde{I}_{l+1}^{2,s}}_{i,\widetilde{t}_l-\widetilde{t}_{l-1}}\right) u^{\widetilde{I}_{2n+2}^{2,s}}_{i,t} U_i^N  \\
		&= U_{i,\widetilde{I}^{2,s}}^{N,\{T_{n},r+\widetilde{t}_{s-1},t+\widetilde{t}_{2n}\}}.
	\end{align*}
	
	\noindent And for $s=2n+1$, we have that
	$$ U_{i,I}(z_{r}^{1,2n+1}) = \left(\prod_{l=1}^{2n} u^{I_{l}^{1,2n+1}}_{i,\widetilde{t}_l-\widetilde{t}_{l-1}}\right) u^{I_{2n+1}^{1,2n+1}}_{i,r} u^{I_{2n+2}^{1,2n+1}}_{i,t-r} U_i^N = U_{i,I^{1,2n+1}}^{N,\{T_{n},r+\widetilde{t}_{2n},t+\widetilde{t}_{2n}\}}, $$
	
	$$ U_{i,I}(z_{r}^{2,2n+1}) = \left(\prod_{l=1}^{2n} u^{I_{l}^{2,2n+1}}_{i,\widetilde{t}_l-\widetilde{t}_{l-1}}\right) u^{I_{2n+1}^{2,2n+1}}_{i,r} u^{I_{2n+2}^{2,2n+1}}_{i,t-r} U_i^N = U_{i,I^{2,2n+1}}^{N,\{T_{n},r+\widetilde{t}_{2n},t+\widetilde{t}_{2n}\}}, $$
	
	$$ U_{i,I}(\widetilde{z}_{r}^{1,2n+1}) = \left(\prod_{l=1}^{2n} u^{\widetilde{I}_{l}^{1,2n+1}}_{i,\widetilde{t}_l-\widetilde{t}_{l-1}}\right) u^{\widetilde{I}_{2n+1}^{1,2n+1}}_{i,r} u^{\widetilde{I}_{2n+2}^{1,2n+1}}_{i,t-r} U_i^N = U_{i,\widetilde{I}^{1,2n+1}}^{N,\{T_{n},r+\widetilde{t}_{2n},t+\widetilde{t}_{2n}\}}, $$
	
	$$ U_{i,I}(\widetilde{z}_{r}^{2,2n+1}) = \left(\prod_{l=1}^{2n} u^{\widetilde{I}_{l}^{2,2n+1}}_{i,\widetilde{t}_l-\widetilde{t}_{l-1}}\right) u^{\widetilde{I}_{2n+1}^{2,2n+1}}_{i,r} u^{\widetilde{I}_{2n+2}^{2,2n+1}}_{i,t-r} U_i^N = U_{i,\widetilde{I}^{2,2n+1}}^{N,\{T_{n},r+\widetilde{t}_{2n},t+\widetilde{t}_{2n}\}}. $$	
	
	\noindent From there on, for a given $s$ we set 
	$$U^{N,\{T_{n},r+\widetilde{t}_{l-1},t+\widetilde{t}_{2n}\}}_{s,1} = \left(U^{N,\{T_{n},r+\widetilde{t}_{s-1},t+\widetilde{t}_{2n}\}}_{i,I}\right)_{i\in [1,d],\ I\in F_{n+1}^{s,1}(J_n)},$$
	$$U^{N,\{T_{n},r+\widetilde{t}_{l-1},t+\widetilde{t}_{2n}\}}_{s,2} = \left(U^{N,\{T_{n},r+\widetilde{t}_{s-1},t+\widetilde{t}_{2n}\}}_{i,I}\right)_{i\in [1,d],\ I\in F_{n+1}^{s,2}(J_n)}, $$
	$$\widetilde{U}^{N,\{T_{n},r+\widetilde{t}_{l-1},t+\widetilde{t}_{2n}\}}_{s,1} = \left(U^{N,\{T_{n},r+\widetilde{t}_{s-1},t+\widetilde{t}_{2n}\}}_{i,I}\right)_{i\in [1,d],\ I\in \widetilde{F}_{n+1}^{s,1}(J_n)},$$
	$$\widetilde{U}^{N,\{T_{n},r+\widetilde{t}_{l-1},t+\widetilde{t}_{2n}\}}_{s,2} = \left(\widetilde{U}^{N,\{T_{n},r+\widetilde{t}_{s-1},t+\widetilde{t}_{2n}\}}_{i,I}\right)_{i\in [1,d],\ I\in \widetilde{F}_{n+1}^{s,2}(J_n)}.$$
	
	\noindent Consequently for $s\leq 2n$, by using the fact that $Q\in \G_{d,q}^{n}$, we get that
	\begin{align}
		\label{dxifvles}
		&\sum_{1\leq j\leq d_s} \tau_{N}\Big(\left[\Big(\delta_{s,j}^2\left[\delta_i^1 \D_iS \right] \left(z_{r}^{1,s} \right)\Big)\boxtimes \Big( \delta_{s,j}^1\left[\delta_i^1 \D_iS \right] \left(\widetilde{z}_{r}^{1,s} \right)\Big)\right] \nonumber\\
		&\quad\quad\quad\quad\quad \boxtimes \left[\Big(\delta_{s,j}^2\left[\delta_i^2 \D_iS \right] \left(\widetilde{z}_{r}^{2,s} \right)\Big)\boxtimes \Big(\delta_{s,j}^1\left[\delta_i^2 \D_iS \right] \left(z_{r}^{2,s} \right)\Big)\right] \Big) \nonumber \\
		&= \sum_{1\leq i\leq d} \sum_{\substack{I,J\in J_n\\\text{such that }\forall l\geq s, I_l=J_l}} \tau_{N}\Bigg(\Bigg[\Big(\delta_{i,I}^2\left[\delta_i^1 \D_iQ \right] \left( U^{N,\{T_{n},r+\widetilde{t}_{s-1},t+\widetilde{t}_{2n}\}}_{s,1}, Z^N \right)\Big) \nonumber \\
		&\quad\quad\quad\quad\quad\quad\quad\quad\quad\quad\quad\quad\quad\quad\quad \boxtimes \Big( \delta_{i,I}^1\left[\delta_i^1 \D_iQ \right] \left( \widetilde{U}^{N,\{T_{n},r+\widetilde{t}_{s-1},t+\widetilde{t}_{2n}\}}_{s,1}, Z^N \right)\Big)\Bigg] \\
		&\quad\quad\quad\quad\quad\quad\quad\quad\quad\quad\quad\quad\quad \boxtimes \Bigg[\Big(\delta_{i,J}^2\left[\delta_i^2 \D_iQ \right] \left( \widetilde{U}^{N,\{T_{n},r+\widetilde{t}_{s-1},t+\widetilde{t}_{2n}\}}_{s,2}, Z^N \right)\Big) \nonumber \\
		&\quad\quad\quad\quad\quad\quad\quad\quad\quad\quad\quad\quad\quad\quad\quad \boxtimes \Big(\delta_{i,J}^1\left[\delta_i^2 \D_iQ \right] \left( U^{N,\{T_{n},r+\widetilde{t}_{s-1},t+\widetilde{t}_{2n}\}}_{s,2}, Z^N \right)\Big)\Bigg] \Bigg) \nonumber \\
		&= 2\int_{[0,1]^4}\tau_N\left( L^{n,s}_{\rho_{n+1},\beta_{n+1},\gamma_{n+1},\delta_{n+1}}(Q) \left(U^{N,\{T_{n},r+\widetilde{t}_{s-1},t+\widetilde{t}_{2n}\}}, Z^N\right)\right) \nonumber
	\end{align}
	
	\noindent Besides, since $I\in J_n$ only has $2n$ elements, the condition ``$I,J\in J_n$, such that $\forall l\geq 2n+1, I_l=J_l$'' is satisfied for any $I,J$. Consequently the above formula still stands for $s=2n+1$. Thus, in combination with Equation \eqref{oisjclakmsc}, we get that
	\begin{align*}
		&\E\left[\tau_N\Big(Q(U^{N,T_n},Z^N)\Big)\right] - \E\left[\tau_N\Big(Q(u^{T_n,t^*},Z^N)\Big)\right] \\
		&= \frac{1}{N^2}\sum_{1\leq s\leq 2n}\int_{0}^{t^*-\widetilde{t}_{2n}}\int_0^{\widetilde{t}_s-\widetilde{t}_{s-1}} \int_{[0,1]^4} \\
		&\quad\quad \E\left[ L^{n,s}_{\rho_{n+1},\beta_{n+1},\gamma_{n+1},\delta_{n+1}}(Q) \left(U^{N,\{T_{n},r+\widetilde{t}_{s-1},t+\widetilde{t}_{2n}\}}, Z^N\right)\right] d\rho_{n+1} d\beta_{n+1} d\gamma_{n+1} d\delta_{n+1}\ dr\ dt \\
		&\quad+\frac{1}{N^2}\int_{0}^{t^*-\widetilde{t}_{2n}}\int_0^{t} \int_{[0,1]^4} \E\left[ L^{n,s}_{\rho_{n+1},\beta_{n+1},\gamma_{n+1},\delta_{n+1}}(Q) \left(U^{N,\{T_{n},r+\widetilde{t}_{2n},t+\widetilde{t}_{2n}\}}, Z^N\right)\right] \\
		&\quad\quad\quad\quad\quad\quad\quad\quad\quad\quad\quad\quad\quad\quad\quad\quad\quad\quad\quad\quad\quad\quad\quad\quad d\rho_{n+1} d\beta_{n+1} d\gamma_{n+1} d\delta_{n+1}\ dr\ dt.
	\end{align*}
	Hence after a change of variable,
	\begin{align*}
		&\E\left[\tau_N\Big(Q(U^{N,T_n},Z^N)\Big)\right] - \E\left[\tau_N\Big(Q(u^{T_n,t^*},Z^N)\Big)\right] \\
		&= \frac{1}{N^2}\sum_{1\leq s\leq 2n}\int_{\widetilde{t}_{2n}}^{t^*}\int_{\widetilde{t}_{s-1}}^{\widetilde{t}_s} \int_{[0,1]^4} \E\left[ L^{n,s}_{\rho_{n+1},\beta_{n+1},\gamma_{n+1},\delta_{n+1}}(Q) \left(U^{N,\{T_{n},r,t\}}, Z^N\right)\right] \\
		&\quad\quad\quad\quad\quad\quad\quad\quad\quad\quad\quad\quad\quad\quad\quad\quad\quad\quad\quad\quad d\rho_{n+1} d\beta_{n+1} d\gamma_{n+1} d\delta_{n+1}drdt \\
		&\quad+\frac{1}{N^2}\int_{\widetilde{t}_{2n}}^{t^*}\int_{\widetilde{t}_{2n}}^{t} \int_{[0,1]^4} \E\left[ L^{n,s}_{\rho_{n+1},\beta_{n+1},\gamma_{n+1},\delta_{n+1}}(Q) \left(U^{N,\{T_{n},r,t\}}, Z^N\right)\right] \\
		&\quad\quad\quad\quad\quad\quad\quad\quad\quad\quad\quad\quad\quad\quad\quad\quad\quad\quad\quad\quad d\rho_{n+1} d\beta_{n+1} d\gamma_{n+1} d\delta_{n+1}\ dr\ dt.
	\end{align*}
	And after renaming $r,t$ into $t_{2n+1},t_{2n+2}$, we get that 
	\begin{align*}
		&\E\left[\tau_N\Big(Q(U^{N,T_n},Z^N)\Big)\right] - \E\left[\tau_N\Big(Q(u^{T_n,t^*},Z^N)\Big)\right] \\
		&= \frac{1}{N^2}\int_{\widetilde{t}_{2n}}^{t^*} \int_0^{t} \int_{[0,1]^4}\tau_N\left( L^{T_{n+1}}_{\rho_{n+1},\beta_{n+1},\gamma_{n+1},\delta_{n+1}}(Q)\left( U^{N,T_{n+1}},Z^N \right) \right)\\
		&\quad\quad\quad\quad\quad\quad\quad\quad\quad\quad\quad\quad\quad\quad\quad\quad\quad\quad\quad\quad d\rho_{n+1} d\beta_{n+1} d\gamma_{n+1} d\delta_{n+1}\ dr dt.
	\end{align*}
	Hence the conclusion.
	
\end{proof}

\noindent Thus we get the following proposition by iterating Lemma \ref{3apparition}, coupled with a lengthy argument to justify that each quantity is well-defined.

\begin{prop}
	\label{3intercoef}
	Let $u$ be a $d$-tuple of free Haar unitaries, and $U^N$ be independent Haar matrices. We define $U^{N,T_n}$ as in Lemma \ref{3apparition}. We also define $u^{T_n}$ similarly to $U^{N,T_n}$ but with $u_i$ instead of $U_i^N$. We also set  
	$$A_i = \{ t_{2i}\geq t_{2i-2}\geq \dots \geq t_2\geq 0 \}\cap\{\forall s\in [1,i], t_{2s} \geq t_{2s-1} \geq 0\} \subset \R^{2i},$$
	then for any $Q\in \F_{d,q}$,
	\begin{align*}
		&\E\left[ \tau_N\Big( Q(U^N,Z^N) \Big) \right] \\
		&= \begin{multlined}[t]\sum_{0\leq i\leq k}\ \frac{1}{N^{2i}} \int_{A_i } \int_{[0,1]^{4i}} \tau_N\Big( \left(L^{{T}_i}_{\alpha_i,\beta_i,\gamma_i,\delta_i} \dots L^{{T}_1}_{\alpha_1,\beta_1,\gamma_1,\delta_1}\right)(Q) (u^{T_i},Z^N) \Big)\\
			\quad\quad\quad\quad\quad\quad\quad\quad\quad\quad\quad\quad\quad\quad\quad\quad\quad\quad\quad\quad\quad\quad\quad\quad d\alpha\ d\beta\ d\gamma\ d\delta\  dt_1\dots dt_{2i} \end{multlined} \\
		&\quad \begin{multlined}[t] + \frac{1}{N^{2(k+1)}} \int_{A_{k+1}} \int_{[0,1]^{4(k+1)}} \E\Big[\tau_N\Big( \Big(L^{{T}_{k+1}}_{\alpha_{k+1},\beta_{k+1},\gamma_{k+1},\delta_{k+1}} \dots \\
			\quad\quad\quad\quad\quad\quad\quad\quad \dots L^{{T}_1}_{\alpha_1,\beta_1,\gamma_1,\delta_1}\Big)(Q) (U^{N,T_{k+1}},Z^N) \Big)\Big] d\alpha\ d\beta\ d\gamma\ d\delta\ dt_1\dots dt_{2(k+1)} .\end{multlined}
	\end{align*}
	In particular, the functions
	\begin{multline}
		\label{dijvcs}
		T_{k+1}\in A_{k+1} \mapsto \int_{[0,1]^{4(k+1)}} \E\Big[\tau_N\Big( \Big(L^{{T}_{k+1}}_{\alpha_{k+1},\beta_{k+1},\gamma_{k+1},\delta_{k+1}} \dots \\
		\dots L^{{T}_1}_{\alpha_1,\beta_1,\gamma_1,\delta_1}\Big)(Q) (U^{N,T_{k+1}},Z^N) \Big)\Big] d\alpha\ d\beta\ d\gamma\ d\delta,
	\end{multline}
	\begin{equation}
		\label{dijvcs2}
		T_i\in A_{i}\mapsto \int_{[0,1]^{4i}} \tau_N\Big( \left(L^{{T}_i}_{\alpha_i,\beta_i,\gamma_i,\delta_i} \dots L^{{T}_1}_{\alpha_1,\beta_1,\gamma_1,\delta_1}\right)(Q) (u^{T_i},Z^N) \Big)\ d\alpha\ d\beta\ d\gamma\ d\delta,
	\end{equation}
	are integrable. 
\end{prop}

To prove the main equation of this proposition we only need to use Lemma \ref{3apparition} repeatedly. The hard part of the proof is then to prove that the functions defined in \eqref{dijvcs} and \eqref{dijvcs2} are indeed integrable. In order to do so, we remark that thanks to Proposition \ref{2Browniantounitary}, one can replace a unitary Brownian motion by a free Haar unitary at the cost of an exponentially small perturbation which one can easily show is integrable. Finally, we show that by replacing these free Brownian motions by free Haar unitaries, the resulting quantity is actually equal to $0$. This last part of the proof is actually by far the longest and most technical.

\begin{proof}
	
	First let us prove the following formula, for any $t^*>0$, we set
	$$A_i^{t^*} = \{ t^*\geq t_{2i}\geq t_{2i-2}\geq \dots \geq t_2\geq 0 \}\cap\{\forall s\in [1,i], t_{2s} \geq t_{2s-1} \geq 0\} \subset \R^{2i},$$
	then
	\begin{align*}
		&\E\left[ \tau_N\Big( Q(U^N,Z^N) \Big) \right] \\
		&=\sum_{0\leq i\leq k}\ \frac{1}{N^{2i}} \int_{A_i^{t^*} } \int_{[0,1]^{4i}} \tau_N\Big( \Big(L^{{T}_i}_{\alpha_i,\beta_i,\gamma_i,\delta_i} \dots L^{{T}_1}_{\alpha_1,\beta_1,\gamma_1,\delta_1}\Big)(Q) (u^{T_i,T},Z^N) \Big)\\
		&\quad\quad\quad\quad\quad\quad\quad\quad\quad\quad\quad\quad\quad\quad\quad\quad\quad\quad\quad\quad\quad\quad\quad\quad\quad\quad\quad d\alpha d\beta d\gamma d\delta\  dt_1\dots dt_{2i}\ \\
		&\quad + \frac{1}{N^{2(k+1)}} \int_{A_{k+1}^{t^*}} \int_{[0,1]^{4(k+1)}} \E\Big[\tau_N\Big( \Big(L^{{T}_{k+1}}_{\alpha_{k+1},\beta_{k+1},\gamma_{k+1},\delta_{k+1}} \dots \\
		&\quad\quad\quad\quad\quad\quad\quad\quad\quad\quad \dots L^{{T}_1}_{\alpha_1,\beta_1,\gamma_1,\delta_1}\Big)(Q) (U^{N,T_{k+1}},Z^N) \Big)\Big] d\alpha d\beta d\gamma d\delta\ dt_1\dots dt_{2(k+1)}.
	\end{align*}
	
	\noindent For $k=0$, we only need to apply Lemma \ref{3apparition} with $n=0$. Then if the formula is true for $k-1$, since $\left(L^{{T}_{k}}_{\alpha_{k},\beta_{k},\gamma_{k},\delta_{k}} \dots L^{{T}_1}_{\alpha_1,\beta_1,\gamma_1,\delta_1}\right)(Q)$ is an element of $\G_{d,q}^k$, one can use Lemma \ref{3apparition} with $n=k$. Besides, for any $T_k=\{t_1,\dots,t_{2k}\}\in A_k^{t^*}$, we have $\widetilde{t}_{2k} = t_{2k}$. Hence the previous formula is true thanks to the fact that
	$$ A_{k+1}^{t^*} = \{A_k^{t^*}\times\R^2\} \cap \{t^*\geq t_{2k+2}\geq t_{2k}, t_{2k+2}\geq t_{2k+1}\geq 0\} .$$
	The hard part of the proof is to show that one can actually let $t^*$ go to infinity. To do so, let us take $T_{k+1}\in A_{k+1}$, with $m\in [1,2k+2]$, if $\widetilde{t}_m-\widetilde{t}_{m-1}\geq 5$ one can use Proposition \ref{2Browniantounitary} on $u^{I_{l}}_{i,\widetilde{t}_m-\widetilde{t}_{m-1}}$ for any $I\in J_n$, thus we set
	$$u^{N,T_{k+1}} = \left(\prod_{l=1}^{m-1} u^{I_{l}}_{i,\widetilde{t}_l-\widetilde{t}_{l-1}}\cdot f_{\widetilde{t}_m-\widetilde{t}_{m-1}}\left(u^{I_{m}}_{i,\widetilde{t}_m-\widetilde{t}_{m-1}}\right) \cdot \prod_{l=m+1}^{2k+2} u^{I_{l}}_{i,\widetilde{t}_l-\widetilde{t}_{l-1}} U_i^N\right)_{i\in [1,d], I\in J_{k+1}}. $$		
	Assuming that
	\begin{equation}
		\label{dskjs1111}
		\int_{[0,1]^{4(k+1)}} \tau_N\Big( \left(L^{{T}_{k+1}}_{\alpha_{k+1},\beta_{k+1},\gamma_{k+1},\delta_{k+1}} \dots L^{{T}_1}_{\alpha_1,\beta_1,\gamma_1,\delta_1}\right)(Q) (u^{N,T_{k+1}},Z^N) \Big) =0.
	\end{equation}
	This implies that
	\begin{align*}
		&\left| \int_{[0,1]^{4(k+1)}} \tau_N\Big( \left(L^{{T}_{k+1}}_{\alpha_{k+1},\beta_{k+1},\gamma_{k+1},\delta_{k+1}} \dots L^{{T}_1}_{\alpha_1,\beta_1,\gamma_1,\delta_1}\right)(Q) (U^{N,T_{k+1}},Z^N) \Big) \right| \\
		&= \Bigg| \int_{[0,1]^{4(k+1)}} \tau_N\Big( \left(L^{{T}_{k+1}}_{\alpha_{k+1},\beta_{k+1},\gamma_{k+1},\delta_{k+1}} \dots L^{{T}_1}_{\alpha_1,\beta_1,\gamma_1,\delta_1}\right)(Q) (U^{N,T_{k+1}},Z^N) \Big) \\
		&\quad  - \int_{[0,1]^{4(k+1)}} \tau_N\Big( \left(L^{{T}_{k+1}}_{\alpha_{k+1},\beta_{k+1},\gamma_{k+1},\delta_{k+1}} \dots L^{{T}_1}_{\alpha_1,\beta_1,\gamma_1,\delta_1}\right)(Q) (u^{N,T_{k+1}},Z^N) \Big) \Bigg| \\
		&\leq K \sup_{i\in [1,d], I\in J_{k+1}} \norm{U^{N,T_{k+1}}_{i,I}- u^{N,T_{k+1}}_{i,I}},
	\end{align*}
	where $K$ does not depend on $T_{k+1}$. Hence thanks to Proposition \ref{2Browniantounitary} there is a constant $C$ independent of $T_{k+1}$ such that if $\left|\widetilde{t}_m-\widetilde{t}_{m-1}\right|> 5$, then
	\begin{equation}
		\label{kdncksmcnlks}
		\left| \int_{[0,1]^{4(k+1)}} \tau_N\Big( \left(L^{{T}_{k+1}}_{\alpha_{k+1},\beta_{k+1},\gamma_{k+1},\delta_{k+1}} \dots L^{{T}_1}_{\alpha_1,\beta_1,\gamma_1,\delta_1}\right)(Q) (U^{N,T_{k+1}},Z^N) \Big) \right| \leq C e^{-\left|\widetilde{t}_m-\widetilde{t}_{m-1}\right|/2} .
	\end{equation}
	By modifying the constant $C$, we get that this inequality remains true if  $\left|\widetilde{t}_m-\widetilde{t}_{m-1}\right|\leq 5$. Hence one can find a constant $C$ such that
	\begin{align*}
		&\left| \int_{[0,1]^{4(k+1)}} \tau_N\Big( \left(L^{{T}_{k+1}}_{\alpha_{k+1},\beta_{k+1},\gamma_{k+1},\delta_{k+1}} \dots L^{{T}_1}_{\alpha_1,\beta_1,\gamma_1,\delta_1}\right)(Q) (U^{N,T_{k+1}},Z^N) \Big) \right| \\
		&\leq C \min_{1\leq m\leq 2k+2} e^{-\left|\widetilde{t}_m-\widetilde{t}_{m-1}\right|/2} \\
		&\leq C e^{-\max_{1\leq m\leq 2k+2} \left|\widetilde{t}_m-\widetilde{t}_{m-1}\right|/2 } \\
		&\leq C e^{-\frac{1}{2k+2}\sum_{1\leq m\leq 2k+2} \left|\widetilde{t}_m-\widetilde{t}_{m-1}\right|/2 } \\
		&\leq C e^{-\frac{t_{2k+2}}{4(k+1)}}.
	\end{align*}
	Thus since the function $T_{k+1}\mapsto e^{-\frac{t_{2k+2}}{4(k+1)}}$ is integrable on $A_{k+1}$, so is \eqref{dijvcs}, and similarly one gets that \eqref{dijvcs2} is also integrable. It remains to prove Equation \eqref{dskjs1111}, which we do in Proposition \ref{pfffffffff}.
\end{proof}

\begin{prop}
	\label{pfffffffff}
	With the notations of Proposition \ref{3intercoef}, given $m\in [1,2k+2]$, we set
	$$u^{N,T_{k+1}} = \left(\prod_{l=1}^{m-1} u^{I_{l}}_{i,\widetilde{t}_l-\widetilde{t}_{l-1}}\cdot f_{\widetilde{t}_m-\widetilde{t}_{m-1}}\left(u^{I_{m}}_{i,\widetilde{t}_m-\widetilde{t}_{m-1}}\right) \cdot \prod_{l=m+1}^{2k+2} u^{I_{l}}_{i,\widetilde{t}_l-\widetilde{t}_{l-1}} U_i^N\right)_{i\in [1,d], I\in J_{k+1}}, $$	
	with $f_{\widetilde{t}_m-\widetilde{t}_{m-1}}$ as in Proposition \ref{2Browniantounitary}, and $\widetilde{t}_m-\widetilde{t}_{m-1} \geq 5$. Then for any $R\in\G_{d,q}^{k}$,
	\begin{equation}
		\label{dskjs}
		\int_{[0,1]^{4(k+1)}} \tau_N\Big( \left(L^{{T}_{k+1}}_{\alpha_{k+1},\beta_{k+1},\gamma_{k+1},\delta_{k+1}}(R)\right) (u^{N,T_{k+1}},Z^N) \Big) =0.
	\end{equation}
\end{prop}

Let us start by proving the following proposition.

\begin{prop}
	\label{twcswcd}
	Let us remind that $\G_{d,q}^{k}$ is the vector space generated by $\F_{d,q}^{i_0,\dots,i_{k-1}}$ for every $i_j\in [1,2j+1]$. Then if Proposition \ref{pfffffffff} holds for $R\in\A_{d,q}^{i_0,\dots,i_{k-1}}$ for every $i_j\in [1,2j+1]$, it also holds for $R\in \G_{d,q}^{k}$.
\end{prop}

The proposition above simply states that the operator $L^{{T}_{k+1}}_{\alpha_{k+1},\beta_{k+1},\gamma_{k+1},\delta_{k+1}}$ is in some sense continuous and that proving Proposition \ref{pfffffffff} for polynomials will also allows us to consider power series, hence any elements of $\G_{d,q}^{k}$. However the set $\G_{d,q}^{k}$ is not the one of power series. Indeed, for example, with the notations of Definition \ref{skjdncksn}, $E_Q=e^Q$ and $E_{\alpha Q} E_{(1-\alpha) Q} = e^{\alpha Q} e^{(1-\alpha) Q}$ are distinct objects. Thus the proof of Proposition \ref{twcswcd} consists in justifying through lengthy but unavoidable computations that applying $L^{{T}_{k+1}}_{\alpha_{k+1},\beta_{k+1},\gamma_{k+1},\delta_{k+1}}$ to an element of $\G_{d,q}^{k}$ and then expanding the result in a power series yields the same results as doing those two operations in the reverse order.

\begin{proof}
	
	Let $R$ be a polynomial, then $L^{{T}_{k+1}}_{\alpha_{k+1},\beta_{k+1},\gamma_{k+1},\delta_{k+1}}(R) \left(u^{N,T_{k+1}},Z^N\right)$ does not depend on $\alpha_{k+1},\beta_{k+1},\gamma_{k+1},\delta_{k+1}$, thus one has that
	$$ \int_{[0,1]^{4(k+1)}} L^{{T}_{k+1}}_{\alpha_{k+1},\beta_{k+1},\gamma_{k+1},\delta_{k+1}}(R) \left(u^{N,T_{k+1}},Z^N\right) = L^{{T}_{k+1}}_{\alpha_{k+1},\beta_{k+1},\gamma_{k+1},\delta_{k+1}}(R) \left(u^{N,T_{k+1}},Z^N\right). $$
	Thus if we assume that Proposition \ref{pfffffffff} holds for polynomials, then for any polynomial $R$, 
	\begin{equation}
		\label{ksucdsk}
		\tau_N\Big( L^{{T}_{k+1}}_{\alpha_{k+1},\beta_{k+1},\gamma_{k+1},\delta_{k+1}}(R) \left(u^{N,T_{k+1}},Z^N\right) \Big) =0.
	\end{equation}
	In order to keep the computations shorter, we will only prove Proposition \ref{twcswcd} for $R = e^Q$ where $Q\in\A_{d,q}^{i_0,\dots,i_{k-1}}$ a polynomial. As we explain at the end of the proof, the general case is handled with the same kind of computations. Thus, if one can prove that
	\begin{align*}
		&\int_{[0,1]^{4(k+1)}} \tau_N\Big( L^{{T}_{k+1}}_{\alpha_{k+1},\beta_{k+1},\gamma_{k+1},\delta_{k+1}}(e^Q) \left(u^{N,T_{k+1}},Z^N\right) \Big) \\
		&= \lim\limits_{l\to\infty} \tau_N\left( L^{{T}_{k+1}}_{\alpha_{k+1},\beta_{k+1},\gamma_{k+1},\delta_{k+1}}\left( \sum_{0\leq n\leq l}\frac{Q^n}{n!}\right) \left(u^{N,T_{k+1}},Z^N\right) \right),
	\end{align*}
	then one can conclude since, thanks to Equation \eqref{ksucdsk}, the right hand side is equal to $0$. 
	
	To begin with, one has that
	$$ \D_{\alpha,i} e^Q = \delta_i Q \widetilde{\#} \left(e^{\alpha Q} e^{(1-\alpha)Q}\right). $$
	Besides, for $A,B$ polynomials, one also has that
	\begin{align*}
		\delta_{\beta,i} (Ae^{\alpha Q} e^{(1-\alpha)Q}B) =& (\delta_i A) e^{\alpha Q} e^{(1-\alpha)Q}B \\
		&+ \alpha A e^{\alpha\beta Q} (\delta_i Q ) e^{\alpha(1-\beta) Q} e^{(1-\alpha)Q} B \\
		&+ (1-\alpha) A e^{\alpha Q} e^{(1-\alpha)\beta Q} (\delta_i Q ) e^{(1-\alpha)(1-\beta)Q} B \\
		&+  Ae^{\alpha Q} e^{(1-\alpha)Q} (\delta_iB). 
	\end{align*}
	And by doing the same with $\delta_{\gamma,j,J}$ and $\delta_{\rho,j,J}$, one has that there exist multilinear maps $\mathcal{L}, \mathcal{M}_j, \mathcal{N}_j, \mathcal{P}_j$ such that
	\begin{align}
		\label{sidchsi}
		&L^{{T}_{k+1}}_{\alpha,\beta,\gamma,\rho}(e^Q)\left(u^{N,T_{k+1}},Z^N\right) \\ \nonumber
		&= \mathcal{L}\left(e^{\alpha Q},e^{(1-\alpha) Q}\right) \\ \nonumber
		&+ \alpha \mathcal{M}_1\left(e^{\alpha\rho Q},e^{\alpha(1-\rho)Q}, e^{(1-\alpha)Q}\right) + (1-\alpha)\mathcal{M}_2\left(e^{\alpha Q}, e^{\rho(1-\alpha) Q},e^{(1-\rho)(1-\alpha)Q}\right) \\ \nonumber
		&+ \alpha \mathcal{M}_3\left(e^{\alpha\gamma Q},e^{\alpha(1-\gamma)Q}, e^{(1-\alpha)Q}\right) + (1-\alpha)\mathcal{M}_4\left(e^{\alpha Q}, e^{\gamma(1-\alpha) Q}, e^{(1-\gamma)(1-\alpha)Q}\right) \\ \nonumber
		&+ \alpha \mathcal{M}_5\left(e^{\alpha\beta Q},e^{\alpha(1-\beta)Q}, e^{(1-\alpha) Q}\right) + (1-\alpha)\mathcal{M}_6\left(e^{\alpha Q}, e^{(1-\alpha)\beta Q},e^{(1-\alpha)(1-\beta) Q}\right) \\ \nonumber
		&+ \alpha^2\beta \mathcal{N}_1\left(e^{\alpha\beta\gamma Q},e^{\alpha\beta(1-\gamma)Q},e^{\alpha(1-\beta)Q}, e^{1-\alpha) Q}\right) \\ \nonumber
		&+ \alpha^2(1-\beta) \mathcal{N}_2\left(e^{\alpha\beta Q}, e^{\alpha(1-\beta)\rho Q},e^{\alpha(1-\beta)(1-\rho)Q}, e^{(1-\alpha) Q}\right) \\ \nonumber
		&+ \alpha(1-\alpha) \mathcal{N}_3\left(e^{\alpha\beta Q}, e^{\alpha(1-\beta)Q}, e^{(1-\alpha)\rho Q}, e^{(1-\alpha)(1-\rho) Q}\right) \\ \nonumber
		&+ (1-\alpha)\alpha \mathcal{N}_4\left( e^{\alpha\gamma Q}, e^{\alpha(1-\gamma)Q}, e^{(1-\alpha)\beta Q}, e^{(1-\alpha)(1-\beta)Q} \right) \\ \nonumber
		&+ (1-\alpha)^2\beta \mathcal{N}_5\left( e^{(\alpha Q}, e^{(1-\alpha)\beta\gamma Q}, e^{(1-\alpha)\beta(1-\gamma) Q}, e^{(1-\alpha)(1-\beta)Q} \right) \\ \nonumber
		&+ (1-\alpha)^2(1-\beta) \mathcal{N}_6\left( e^{\alpha Q}, e^{(1-\alpha)\beta Q}, e^{(1-\alpha)(1-\beta)\gamma Q}, e^{(1-\alpha)(1-\beta)(1-\gamma) Q} \right) \\ \nonumber
		&+ \alpha^3 \beta(1-\beta)\mathcal{P}_1\left( e^{\alpha\beta\gamma Q}, e^{\alpha\beta(1-\gamma)Q}, e^{\alpha(1-\beta)\rho Q}, e^{\alpha(1-\beta)(1-\rho)Q}, e^{(1-\alpha)Q}\right) \\ \nonumber
		&+ \alpha^2 \beta(1-\alpha)\mathcal{P}_2\left(e^{\alpha\beta\gamma Q}, e^{\alpha\beta(1-\gamma)Q}, e^{\alpha(1-\beta)Q}, e^{(1-\alpha)\rho Q}, e^{(1-\alpha)(1-\rho)Q}\right) \\ \nonumber
		&+ (1-\alpha)^2\alpha(1-\beta) \mathcal{P}_3\left( e^{\alpha\gamma Q}, e^{\alpha(1-\gamma)Q}, e^{(1-\alpha)\beta Q},  e^{(1-\alpha)(1-\beta)\rho Q} , e^{(1-\alpha)(1-\beta)(1-\rho)Q} \right) \\ \nonumber
		&+ (1-\alpha)^3\beta(1-\beta) \mathcal{P}_4\left( e^{\alpha Q}, e^{(1-\alpha)\beta \gamma Q}, e^{(1-\alpha)\beta(1-\gamma)Q}, e^{(1-\alpha)(1-\beta)\rho Q}, e^{(1-\alpha)(1-\beta)(1-\rho)Q} \right).
	\end{align}
	
	\noindent On the other hand, one has that
	$$ \D_{\alpha,i} Q^n = \sum_{1\leq n_1\leq n} \delta_i Q \widetilde{\#} (Q^{n_1-1}Q^{n-n_1}). $$
	Besides, for $A,B$ polynomials, one has that
	\begin{align*}
		\delta_{\beta,i} (AQ^{n_1-1}Q^{n-n_1}B) =& (\delta_i A) Q^{n_1-1}Q^{n-n_1}B \\
		&+ \sum_{1\leq n_2\leq n_1-1} A Q^{n_2-1} (\delta_i Q ) Q^{n_1-n_2-1}Q^{n-n_1} B \\
		&+ \sum_{1\leq n_2\leq n-n_1} A Q^{n_1-1}Q^{n_2-1} (\delta_i Q ) Q^{n-n_1-n_2} B \\
		&+  AQ^{n_1-1}Q^{n-n_1} (\delta_iB). 
	\end{align*}
	And by doing the same with $\delta_{\gamma,j,J}$ and $\delta_{\rho,j,J}$, with the same multilinear map as defined previously, one has
	\begin{align}
		\label{dkfnvdkds}
		&L^{{T}_{k+1}}_{\alpha,\beta,\gamma,\rho}(Q^n)\left(u^{N,T_{k+1}},Z^N\right) \\ \nonumber
		&=\ \sum_{1\leq n_1\leq n} \mathcal{L}\left( Q^{n_1-1},Q^{n-n_1}\right) \\ \nonumber
		&+ \sum_{1\leq n_1\leq n} \sum_{1\leq n_4\leq n_1-1}  \mathcal{M}_1\left(Q^{n_4-1},Q^{n_1-n_4-1},Q^{n-n_1}\right) \\ \nonumber
		&+  \sum_{1\leq n_4\leq n-n_1} \mathcal{M}_2\left( Q^{n_1-1}, Q^{n_4-1}, Q^{n-n_1-n_4} \right) \\ \nonumber
		&+ \sum_{1\leq n_1\leq n} \sum_{1\leq n_3\leq n_1-1}  \mathcal{M}_3\left(Q^{n_3-1},Q^{n_1-n_3-1},Q^{n-n_1}\right) \\ \nonumber
		&+  \sum_{1\leq n_3\leq n-n_1} \mathcal{M}_4\left( Q^{n_1-1}, Q^{n_3-1}, Q^{n-n_1-n_3} \right) \\ \nonumber
		&+ \sum_{1\leq n_1\leq n} \sum_{1\leq n_2\leq n_1-1} \mathcal{M}_5\left(Q^{n_2-1}, Q^{n_1-n_2-1},Q^{n-n_1}\right) \\ \nonumber 
		&+ \sum_{1\leq n_2\leq n-n_1} \mathcal{M}_6\left(  Q^{n_1-1}, Q^{n_2-1}, Q^{n-n_1-n_2}\right) \\ \nonumber
		&+ \sum_{1\leq n_1\leq n} \sum_{1\leq n_2\leq n_1-1} \sum_{1\leq n_3\leq n_2-1} \mathcal{N}_1\left( Q^{n_3-1}, Q^{n_2-n_3-1},Q^{n_1-n_2-1},Q^{n-n_1}\right) \\ \nonumber
		&+ \sum_{1\leq n_1\leq n} \sum_{1\leq n_2\leq n_1-1} \sum_{1\leq n_4\leq n_1-n_2-1} \mathcal{N}_2\left( Q^{n_2-1}, Q^{n_4-1},Q^{n_1-n_2-n_4-1}, Q^{n-n_1} \right) \\ \nonumber
		&+ \sum_{1\leq n_1\leq n} \sum_{1\leq n_2\leq n_1-1} \sum_{1\leq n_4\leq n-n_1} \mathcal{N}_3\left( Q^{n_2-1}, Q^{n_1-n_2-1}, Q^{n_4-1}, Q^{n-n_1-n_4} \right) \\ \nonumber
		&+ \sum_{1\leq n_1\leq n} \sum_{1\leq n_2\leq n-n_1} \sum_{1\leq n_3\leq n_1-1} \mathcal{N}_4\left( Q^{n_3-1}, Q^{n_1-n_3-1}, Q^{n_2-1}, Q^{n-n_1-n_2} \right) \\ \nonumber
		&+ \sum_{1\leq n_1\leq n} \sum_{1\leq n_2\leq n-n_1} \sum_{1\leq n_3\leq n_2-1} \mathcal{N}_5\left( Q^{n_1-1}, Q^{n_3-1}, Q^{n_2-n_3-1}, Q^{n-n_1-n_2} \right) \\ \nonumber
		&+ \sum_{1\leq n_1\leq n} \sum_{1\leq n_2\leq n-n_1} \sum_{1\leq n_4\leq n-n_1-n_2} \mathcal{N}_6\left( Q^{n_1-1}, Q^{n_2-1}, Q^{n_4-1}, Q^{n-n_1-n_2-n_4} \right) \\ \nonumber
		&+ \sum_{1\leq n_1\leq n} \sum_{1\leq n_2\leq n-n_1} \sum_{1\leq n_3\leq n_2-1} \sum_{1\leq n_4\leq n_1-n_2-1} \\ \nonumber
		&\quad\quad\quad\quad\quad\quad\quad\quad\quad\quad \mathcal{P}_1\left( Q^{n_3-1}, Q^{n_2-n_3-1}, Q^{n_4-1}, Q^{n_1-n_2-n_4-1}, Q^{n-n_1} \right) \\ \nonumber
		&+ \sum_{1\leq n_1\leq n} \sum_{1\leq n_2\leq n-n_1} \sum_{1\leq n_3\leq n_2-1} \sum_{1\leq n_4\leq n-n_1}\\ \nonumber
		&\quad\quad\quad\quad\quad\quad\quad\quad\quad\quad \mathcal{P}_2\left( Q^{n_3-1}, Q^{n_2-n_3-1}, Q^{n_1-n_2-1}, Q^{n_4-1}, Q^{n-n_1-n_4} \right) \\ \nonumber
		&+ \sum_{1\leq n_1\leq n} \sum_{1\leq n_2\leq n-n_1} \sum_{1\leq n_3\leq n_1-1} \sum_{1\leq n_4\leq n-n_1-n_2}\\ \nonumber
		&\quad\quad\quad\quad\quad\quad\quad\quad\quad\quad \mathcal{P}_3\left( Q^{n_3-1}, Q^{n_1-n_3-1}, Q^{n_2-1}, Q^{n_4-1}, Q^{n-n_1-n_2-n_4} \right) \\ \nonumber
		&+ \sum_{1\leq n_1\leq n} \sum_{1\leq n_2\leq n-n_1} \sum_{1\leq n_3\leq n_2-1} \sum_{1\leq n_4\leq n-n_1-n_2}\\ \nonumber
		&\quad\quad\quad\quad\quad\quad\quad\quad\quad\quad \mathcal{P}_4\left( Q^{n_1-1}, Q^{n_3-1}, Q^{n_2-n_3-1}, Q^{n_4-1}, Q^{n-n_1-n_2-n_4} \right).
	\end{align}
	
	\noindent Then:
	\begin{itemize}
		\item One has $$ \forall n,m\geq 0,\quad \int_0^1 \alpha^n (1-\alpha)^m\ d\alpha = \frac{n! m!}{(n+m+1)!}.$$
		\item One has
		\begin{align*}
			&\left\{ (l_1,l_2,l_3,l_4,l_5)\ \middle|\ l_1 +l_2+l_3+l_4+l_5 = n-4, \forall i, l_i\geq 0 \right\} \\
			&= \{ (n_1-1,n_3-1, n_2-n_3-1,n_4-1,n-n_1-n_2-n_4)\ |\\
			&\quad\quad 1\leq n_1\leq n, 1\leq n_2\leq n-n_1, 1\leq n_3\leq n_2-1, 1\leq n_4\leq n-n_1-n_2 \} 
		\end{align*}
		Indeed, if $n_1,n_2,n_3,n_4$ are picked such that they satisfy the inequalities of the second set, then
		$$ n_1-1 + n_3-1 + n_2-n_3-1+ n_4-1 + n-n_1-n_2-n_4 = n-4,$$
		$$ n_1-1\geq 0, \quad n_3-1 \geq 0, \quad n_2-n_3-1 \geq 0,\quad n_4-1 \geq 0,\quad n-n_1-n_2-n_4 \geq 0.$$
		Whereas if $(l_1,l_2,l_3,l_4,l_5)$ belongs to the first set, then by setting
		$$ n_1 = l_1+1,\quad n_2 = l_2+l_3+2,\quad n_3= l_2+1, \quad n_4 =l_4+1,$$
		one has that
		$$ (l_1,l_2,l_3,l_4,l_5) =  (n_1-1,n_3-1, n_2-n_3-1,n_4-1,n-n_1-n_2-n_4),$$
		and 
		$$ 1\leq n_1\leq n,\quad  1\leq n_2\leq n-n_1,\quad 1\leq n_3\leq n_2-1,\quad 1\leq n_4\leq n-n_1-n_2.$$
		Hence the conclusion. Note in particular that for $n<4$, those sets are empty.
		\item Given a polynomial $Q$, and non-commutative variables $x$ in a $\CC^*$-algebra, $e^{Q(x)} = \sum_{n\geq 0} \frac{Q(x)^n}{n!}$. Consequently since $\mathcal{P}_4(A,B,C,D)$ is merely evaluating $A,B,C,D$ in specific set of non-commutative random variables and then multiplying them with other polynomials, we have that 
		\begin{align*}
			&\mathcal{P}_4\left( e^{\alpha Q}, e^{(1-\alpha)\beta \gamma Q}, e^{(1-\alpha)\beta(1-\gamma)Q}, e^{(1-\alpha)(1-\beta)\rho Q}, e^{(1-\alpha)(1-\beta)(1-\rho)Q} \right) \\
			&= \sum_{n\geq 4} \sum_{l_1+l_2+l_3+l_4+l_5 = n-4} \alpha^{l_1}(1-\alpha)^{l_2+l_3+l_4+l_5} \beta^{l_2+l_3} (1-\beta)^{l_4+l_5} \gamma^{l_2} (1-\gamma)^{l_3} \rho^{l_4} (1-\rho)^{l_5} \\
			&\quad\quad\quad\quad\quad\quad\quad\quad\quad \times \mathcal{P}_4\left( Q^{l_1}, Q^{l_2}, Q^{l_3}, Q^{l_4}, Q^{l_5} \right) \times \left( l_1!\ l_2!\ l_3!\ l_4!\ l_5! \right)^{-1}.
		\end{align*}
	\end{itemize}
	
	\noindent Hence we have,
	\begin{align}
		\label{odijcsokmdlsn}
		&\int_{[0,1]^4}(1-\alpha)^3\beta(1-\beta) \mathcal{P}_4\Big( e^{\alpha Q}, e^{(1-\alpha)\beta \gamma Q}, e^{(1-\alpha)\beta(1-\gamma)Q}, e^{(1-\alpha)(1-\beta)\rho Q}, \\ \nonumber
		&\quad\quad\quad\quad\quad\quad\quad\quad\quad\quad\quad\quad\quad\quad\quad\quad\quad\quad\quad\quad\quad\quad\quad\quad e^{(1-\alpha)(1-\beta)(1-\rho)Q} \Big) d\alpha d\beta d\gamma d\rho \\ \nonumber
		&= \sum_{n\geq 4} \sum_{l_1+\dots+l_5 = n-4} \int_0^1\alpha^{l_1}(1-\alpha)^{3+l_2+l_3+l_4+l_5} d\alpha \int_0^1 \beta^{1+l_2+l_3} (1-\beta)^{1+l_4+l_5} d\beta\ \times\\ \nonumber 
		&\int_0^1 \gamma^{l_2} (1-\gamma)^{l_3} d\gamma \int_0^1 \rho^{l_4} (1-\rho)^{l_5} d\rho \times \mathcal{P}_4\left( Q^{l_1}, Q^{l_2}, Q^{l_3}, Q^{l_4}, Q^{l_5} \right) \times \left( l_1!\ l_2!\ l_3!\ l_4!\ l_5! \right)^{-1} \\ \nonumber
		&= \sum_{n\geq 0}\ \sum_{1\leq n_1\leq n} \sum_{1\leq n_2\leq n-n_1} \sum_{1\leq n_3\leq n_2-1} \sum_{1\leq n_4\leq n-n_1-n_2} \\ \nonumber
		&\quad\quad \mathcal{P}_4\left( Q^{n_1-1}, Q^{n_3-1}, Q^{n_2-n_3-1}, Q^{n_4-1}, Q^{n-n_1-n_2-n_4} \right) \times \int_0^1\alpha^{n_1-1}(1-\alpha)^{n-n_1} d\alpha \\ \nonumber
		& \int_0^1 \beta^{n_2-1} (1-\beta)^{n-n_1-n_2} d\beta \int_0^1 \gamma^{n_3-1} (1-\gamma)^{n_2-n_3-1} d\gamma \int_0^1 \rho^{n_4-1} (1-\rho)^{n-n_1-n_2-n_4} d\rho \\ \nonumber
		&\times \frac{1}{(n_1-1)!\ (n_3-1)!\ (n_2-n_3-1)!\ (n_4-1)!\ (n-n_1-n_2-n_4)!} \\ \nonumber
		&= \sum_{n\geq 0}\ \sum_{1\leq n_1\leq n} \sum_{1\leq n_2\leq n-n_1} \sum_{1\leq n_3\leq n_2-1} \sum_{1\leq n_4\leq n-n_1-n_2} \\ \nonumber
		&\quad\quad \mathcal{P}_4\left( Q^{n_1-1}, Q^{n_3-1}, Q^{n_2-n_3-1}, Q^{n_4-1}, Q^{n-n_1-n_2-n_4} \right) \times \frac{(n_1-1)! (n-n_1)!}{n!} \\ \nonumber
		&\quad\quad \frac{(n_2-1)! (n-n_1-n_2)!}{(n-n_1)!}\times \frac{(n_3-1)! (n_2-n_3-1)!}{(n_2-1)!} \times \frac{(n_4-1)! (n-n_1-n_2-n_4)!}{(n-n_1-n_2)!} \\ \nonumber
		&\quad\quad  \times \frac{1}{(n_1-1)!\ (n_3-1)!\ (n_2-n_3-1)!\ (n_4-1)!\ (n-n_1-n_2-n_4)!} \\ \nonumber
		&= \sum_{n\geq 0}\frac{1}{n!} \sum_{\substack{1\leq n_1\leq n \\ 1\leq n_2\leq n-n_1}} \sum_{\substack{1\leq n_3\leq n_2-1 \\ 1\leq n_4\leq n-n_1-n_2}} \mathcal{P}_4\left( Q^{n_1-1}, Q^{n_3-1}, Q^{n_2-n_3-1}, Q^{n_4-1}, Q^{n-n_1-n_2-n_4} \right).
	\end{align}
	Thus by doing the same computations for the other multilinear maps, we get that
	\begin{align*}
		&\tau_N\Big( L^{{T}_{k+1}}_{\alpha_{k+1},\beta_{k+1},\gamma_{k+1},\delta_{k+1}}(e^Q) \left(u^{N,T_{k+1}},Z^N\right) \Big) \\
		&= \sum_{n\geq 0} \tau_N\left( L^{{T}_{k+1}}_{\alpha_{k+1},\beta_{k+1},\gamma_{k+1},\delta_{k+1}}\left( \frac{Q^n}{n!}\right) \left(u^{N,T_{k+1}},Z^N\right) \right).
	\end{align*}
	Hence the conclusion for the case $R=e^Q$. For the general case, i.e.
	$$ R= P_1e^{Q_1}P_2\dots e^{Q_g}P_{g+1},$$
	one can compute $L^{{T}_{k+1}}_{\alpha,\beta,\gamma,\rho}(R)\left(u^{N,T_{k+1}},Z^N\right)$ exactly as in Equation \eqref{sidchsi}. However, since $R$ contains the exponential of $g$ different polynomials, we will have around $g^4$ multilinear maps that will appear. Besides, if one expands $R$ as a power series, then one can get a formula similar to Equation \eqref{dkfnvdkds} but with the $g^4$ multilinear maps from the previous sentence. 
	
	From there on, the computations are very similar to the case $R=e^Q$. Indeed, since by applying $L^{{T}_{k+1}}_{\alpha,\beta,\gamma,\rho}$ one only differentiates four times, the harder case is still the one that we prove in Equation \eqref{odijcsokmdlsn}.
	
\end{proof}

We can now give a proof of Proposition \ref{pfffffffff}.

\begin{proof}[Proof of Proposition \ref{pfffffffff}]
	
	Thanks to Proposition \ref{twcswcd}, if we fix $i_0,\dots,i_{k-1}$ and take $R\in \A_{d,q}^{i_0,\dots,i_{n-1}}$, it is sufficient to show that
	\begin{equation}
		\label{dskjs22}
		\tau_N\Big( L^{{T}_{k+1}}_{\alpha_{k+1},\beta_{k+1},\gamma_{k+1},\delta_{k+1}}(R) \left(u^{N,T_{k+1}},Z^N\right) \Big) =0.
	\end{equation}
	Then thanks to Equation \eqref{3fullop}, we know that
	\begin{align*}
		L^{T_{k+1}}_{\rho_{k+1},\beta_{k+1},\gamma_{k+1},\delta_{k+1}}(R) = &\ \1_{[\widetilde{t}_{2k},t_{2k+2}]}(t_{2k+1}) L^{k,2k+1}_{\rho_{k+1},\beta_{k+1},\gamma_{k+1},\delta_{k+1}}(R) \\
		&+ \sum_{1\leq s\leq 2k} \1_{[\widetilde{t}_{s-1},\widetilde{t}_s]}(t_{2k+1}) L^{k,s}_{\rho_{k+1},\beta_{k+1},\gamma_{k+1},\delta_{k+1}}(R).
	\end{align*}
	Consequently, depending on $T_{k+1}$, there exists $s$ such that $$L^{T_{k+1}}_{\rho_{k+1}, \beta_{k+1},\gamma_{k+1},\delta_{k+1}} = L^{k,s}_{\rho_{k+1},\beta_{k+1},\gamma_{k+1},\delta_{k+1}}.$$ Thus in order to prove Equation \eqref{dskjs22}, we need to show that for such $s$,
	\begin{equation}
		\label{doijvsed}
		\tau_N\left(  L^{k,s}_{\rho_{k+1},\beta_{k+1},\gamma_{k+1},\delta_{k+1}}(R)(u^{N,T_{k+1}},Z^N)\right) =0.
	\end{equation}
	\textbf{Step 1:} Let us first assume that $s\geq m$. Thanks to Lemma \ref{3detail}, we know that if $I,J\in J_{i_0,\dots,i_{n-1},s}$ and that $I_m=J_m$, then for any $l\geq m$, $I_l=J_l$. Besides, given a family of free Haar unitaries $(u_i)_i$ and another family of unitary operators $(v_i)_i$ (not necessarily free between each other) free from the family $(u_i)_i$, then the families $(u_i)_i$ and $(u_iv_i)_i$ have the same joint distribution in the sense of Definition \ref{3freeprob}. Indeed, by using the moment-cumulant formula, see Proposition 11.4 and Theorem 11.16 of \cite{nica_speicher_2006}, one can prove that those families have the same mixed moments. Consequently, thanks to Lemma \ref{3detail}, the family
	$$ \left(\prod_{l=1}^{m-1} u^{I_{l}}_{i,\widetilde{t}_l-\widetilde{t}_{l-1}}\cdot f_{\widetilde{t}_m-\widetilde{t}_{m-1}}\left(u^{I_{m}}_{i,\widetilde{t}_m-\widetilde{t}_{m-1}}\right) \cdot \prod_{l=m+1}^{2k+2} u^{I_{l}}_{i,\widetilde{t}_l-\widetilde{t}_{l-1}} U_i^N\right)_{i\in [1,d], I\in J_{i_0,\dots,i_{n-1},s} }, $$	
	has the same joint distribution	as
	$$ v^{T_{k+1}} = \left(\prod_{l=1}^{m-1} u^{I_{l}}_{i,\widetilde{t}_l-\widetilde{t}_{l-1}}\cdot f_{\widetilde{t}_m-\widetilde{t}_{m-1}}\left(u^{I_{m}}_{i,\widetilde{t}_m-\widetilde{t}_{m-1}}\right) \right)_{i\in [1,d], I\in J_{i_0,\dots,i_{n-1},s} }. $$		
	Consequently,
	\begin{equation}
		\label{ksncsa}
		\tau_N\left(  L^{k,s}_{\rho_{k+1},\beta_{k+1},\gamma_{k+1},\delta_{k+1}}(R)(u^{N,T_{k+1}},Z^N)\right) = \tau_N\left(  L^{k,s}_{\rho_{k+1},\beta_{k+1},\gamma_{k+1},\delta_{k+1}}(R)(v^{T_{k+1}},Z^N)\right).
	\end{equation}
	Let $\A_{d(1+c_k), q}$ be the set of non-commutative polynomials in the variables $u_i^{c},v_i^{c},U_i,V_i,Z_j$ for $i\in [1,d]$, and $c\in [1,c_k]$. If $m\leq 2k$ we set 
	$$ S = R\left( \left(\prod_{l=1}^{m} u^{I_{l}}_{i}\right)_{i\in [1,d], I\in J_{i_0,\dots,i_{n-1}}} , Z \right),$$	
	and if $m=s=2k+1$, we set
	$$ S = R\left( \left(\prod_{l=1}^{2k} u^{I_{l}}_{i}\ U_i\right)_{i\in [1,d], I\in J_{i_0,\dots,i_{n-1}}} , Z \right).$$	
	
	\noindent Let us remind that we have defined $\delta_i$ on $\G_{d,q}^{k}$ in Definition \ref{3biz2}. 
	
	\begin{defi}
		\label{kdjvnssc}
		With $\A_{d(1+c_k), q}$ the set of non-commutative polynomials in the variables $u_i^{c},v_i^{c},U_i,V_i,Z_j$ for $i\in [1,d]$ and $c\in [1,c_k]$, one defines the non-commutative differentials $\delta_i$ and $\delta_{i,c}$ on $\A_{d(1+c_k), q}$ by
		$$ \forall P,Q\in \mathcal{A}_{d(1+c_k),q},\quad \delta_{i,c} (PQ) = \delta_{i,c} P \times \left(1\otimes Q\right)	 + \left(P\otimes 1\right) \times \delta_{i,c} Q ,$$
		$$ \delta_{i,c} u_j^{c'} = \1_{i=j}\1_{c=c'}\  u_j^{c}\otimes 1,\quad \delta_{i,c} v_j^{c'} = -\1_{i=j}\1_{c=c'}\ 1\otimes v_j^{c},\quad \delta_{i,c}U_j = \delta_{i,c}V_j = \delta_i Z_j =0\otimes 0. $$
		$$ \forall P,Q\in \mathcal{A}_{d(1+c_k),q},\quad \delta_i (PQ) = \delta_i P \times \left(1\otimes Q\right)	 + \left(P\otimes 1\right) \times \delta_i Q ,$$
		$$ \delta_{i} U_j = \1_{i=j}\  U_j\otimes 1,\quad \delta_{i} V_j = -\1_{i=j}\ 1\otimes V_j,\quad \delta_{i} u_j^{c} = \delta_{i} v_j^{c} = \delta_i Z_j = 0\otimes 0. $$
	\end{defi}
	
	With this definition we have that if $m\leq 2k$,
	\begin{equation}
		\label{ijdncsincds}
		\delta_i\D_i R\left( \left(\prod_{l=1}^{m} u^{I_{l}}_{i}\right)_{i\in [1,d], I\in J_{i_0,\dots,i_{n-1}}} , Z \right) = \sum_{\substack{c^1,c^2\in [1,c_k] \\ \dep^k(c^1)=\dep^k(c^2)=m}} \delta_{i,c^1}\D_{i,c^2} S,
	\end{equation}
	and if $m=2k+1$,
	\begin{equation}
		\delta_i\D_i R\left( \left(\prod_{l=1}^{2k} u^{I_{l}}_{i}\ U_i\right)_{i\in [1,d], I\in J_{i_0,\dots,i_{n-1}}} , Z \right) = \delta_{i}\D_{i} S.
	\end{equation}
	
	\noindent Then for $P,Q\in \A_{d(1+c_k), q}$ we set for $s\leq 2k$,
	\begin{align}
		\label{xdkovse}
		&\Theta_{j,c,m,s}(P\otimes Q):= \sum_{\substack{c^3,c^4\in [1,c_k] \\ \dep^k(c^3)=\dep^k(c^4)=m \\ \exists I,J\in J_{i_0,\dots,i_{n-1}} \text{ such that } \\ I_s=J_s=c,\ I_m=c^3,\ J_m=c^4}} \left(\delta_{j,c^3}^2 P(\X_1^m,Z) \boxtimes \delta_{j,c^3}^1 P(\widetilde{\X}_1^m,Z)\right) \\ \nonumber
		&\quad\quad\quad\quad\quad\quad\quad\quad\quad\quad\quad\quad\quad\quad\quad\quad\quad \boxtimes \left(\delta_{j,c^4}^2 Q(\widetilde{\X}_2^m,Z) \boxtimes \delta_{j,c^4}^1 Q(\X_2^m,Z)\right), 
	\end{align}
	for $m\leq 2k, s=2k+1$,
	\begin{align}
		\label{xdkovse3}
		&\Theta_{j,m,2k+1}(P\otimes Q) := \sum_{\substack{c^3,c^4\in [1,c_k] \\ \dep^k(c^3)=\dep^k(c^4)=m}} \left(\delta_{j,c^3}^2 P(\X_1^m,Z) \boxtimes \delta_{j,c^3}^1 P(\widetilde{\X}_1^m,Z)\right) \\ \nonumber
		&\quad\quad\quad\quad\quad\quad\quad\quad\quad\quad\quad\quad\quad\quad\quad\quad\quad  \boxtimes \left(\delta_{j,c^4}^2 Q(\widetilde{\X}_2^m,Z) \boxtimes \delta_{j,c^4}^1 Q(\X_2^m,Z)\right), 
	\end{align}
	and for $m=2k+1$,
	\begin{align}
		\label{xdkovse2}
		\Theta_{j,2k+1}(P\otimes Q) := &\left(\delta_{j}^2 P(\X_1^{2k+1},Z) \boxtimes \delta_{j}^1 P(\widetilde{\X}_1^{2k+1},Z)\right) \\ \nonumber
		&\boxtimes \left(\delta_{j}^2 Q(\widetilde{\X}_2^{2k+1},Z) \boxtimes \delta_{j}^1 Q(\X_2^{2k+1},Z)\right). 
	\end{align}
	where following the construction of Definition \ref{3biz},
	\begin{align*}
		&\X_1^m = \Big( \left(u_i^{c+c_k}\right)_{\dep^k(c)\leq m}, \left(v_i^{c+c_k}\right)_{\dep^k(c)\leq m} \Big)_{i\in [1,d]}, \\
		&\X_1^{2k+1} = \left( \left(u_i^{c+c_k}\right)_{\dep^k(c)\leq 2k}, u_i^{3c_k+2}, \left(v_i^{c+c_k} \right)_{\dep^k(c)\leq 2k}, v_i^{3c_k+2}\right)_{i\in [1,d]}, \\
		&\X_2^m = \Big( \left(u_i^{c+2c_k}\right)_{\dep^k(c)\leq m}, \left(v_i^{c+2c_k}\right)_{\dep^k(c)\leq m} \Big)_{i\in [1,d]}, \\
		&\X_2^{2k+1} = \left( \left(u_i^{c+2c_k}\right)_{\dep^k(c)\leq 2k}, u_i^{3c_k+3}, \left(v_i^{c+2c_k} \right)_{\dep^k(c)\leq 2k}, v_i^{3c_k+3} \right)_{i\in [1,d]},
	\end{align*}
	and similarly we define $\widetilde{\X}_1^m$ and $\widetilde{\X}_2^m$ by adding $3c_k+3$ to every integer. Besides, we set 
	\begin{equation}
		\label{wjbdcskdnc}
		y^{T_{k+1}}_m = \left( \left(u_{i,\widetilde{t}_l-\widetilde{t}_{l-1}}^{I_l}\right)_{l< m}, f_{\widetilde{t}_m-\widetilde{t}_{m-1}}\left(u^{I_{m}}_{i,\widetilde{t}_m-\widetilde{t}_{m-1}}\right) \right)_{i\in [1,d],I\in J_{i_0,\dots,i_{n-1},s}}.
	\end{equation}
	Then thanks to the formula for $L^{k,s}_{\rho_{k+1},\beta_{k+1},\gamma_{k+1},\delta_{k+1}}$ established in Lemma \ref{3apparition}, we have that for $s\leq 2k$,
	\begin{align*}
		&\tau_N\left(  L^{k,s}_{\rho_{k+1},\beta_{k+1},\gamma_{k+1},\delta_{k+1}}(R)(u^{N,T_{k+1}},Z^N)\right) \\
		&= \frac{1}{2} \sum_{1\leq i,j\leq d}\ \sum_{\substack{c,c^1,c^2\in [1,c_k], \\ \dep^k{c} = s,\  \dep^k{c^1} = \dep^k{c^2} = m}} \tau_N\left(\Theta_{j,c,m,s}\left[\delta_{i,c^1}\D_{i,c^2} S\right]\left( y^{T_{k+1}}_m , Z^N \right)\right).
	\end{align*}
	for $m\leq 2k, s=2k+1$,
	\begin{align*}
		&\tau_N\left(  L^{k,s}_{\rho_{k+1},\beta_{k+1},\gamma_{k+1},\delta_{k+1}}(R)(u^{N,T_{k+1}},Z^N)\right) \\
		&= \frac{1}{2} \sum_{1\leq i,j\leq d}\ \sum_{\substack{c^1,c^2\in [1,c_k], \\ \dep^k{c^1} = \dep^k{c^2} = m}} \tau_N\left(\Theta_{j,m,2k+1}\left[\delta_{i,c^1}\D_{i,c^2} S\right]\left( y^{T_{k+1}}_m , Z^N \right)\right).
	\end{align*}
	and for $m=2k+1$, 
	\begin{align*}
		&\tau_N\left(  L^{k,s}_{\rho_{k+1},\beta_{k+1},\gamma_{k+1},\delta_{k+1}}(R)(u^{N,T_{k+1}},Z^N)\right) =\frac{1}{2} \sum_{1\leq i,j\leq d} \tau_N\left(\Theta_{j,2k+1}\left[\delta_{i}\D_{i} S\right]\left( y^{T_{k+1}}_m , Z^N \right)\right).
	\end{align*}
	
	\noindent Thus, in order to conclude step one, one must show the following lemma.
	\begin{lemma}
		\label{sidvbs}
		With the notations introduced previously in the proof of Proposition \ref{pfffffffff}, and notably in Equations \eqref{xdkovse}, \eqref{xdkovse3}, \eqref{xdkovse2} and \eqref{wjbdcskdnc}, as well as Definition \ref{kdjvnssc}, one has,
		\begin{equation}
			\label{dfvdvs}
			\tau_N\left(\Theta_{j,c,m,s}\left[\delta_{i,c^1}\D_{i,c^2} S\right]\left( y^{T_{k+1}}_m , Z^N \right)\right)=0,
		\end{equation}
		\begin{equation}
			\tau_N\left(\Theta_{j,m,2k+1}\left[\delta_{i,c^1}\D_{i,c^2} S\right]\left( y^{T_{k+1}}_m , Z^N \right)\right)=0,
		\end{equation}
		\begin{equation}
			\tau_N\left(\Theta_{j,2k+1}\left[\delta_{i}\D_{i} S\right]\left( y^{T_{k+1}}_m , Z^N \right)\right)=0.
		\end{equation}
	\end{lemma}
	
	\begin{proof}
		
		The proofs of these equations are essentially the same, so we only prove Equation \eqref{dfvdvs}. To do so, by linearity, we assume that $S$ is a monomial. If there exist $R,T$ monomials and $l,d$ such that $S = R\ u^d_l v^d_l T$, then
		\begin{align*}
			\delta_{i,c^2} S &= \delta_{i,c^2} R \times 1\otimes u^d_l v^d_l T + Ru^d_l v^d_l\otimes 1 \times \delta_{i,c^2} T + \1_{i=l,c^2=d}\left(Ru^d_l \otimes v^d_lT - Ru^d_l\otimes v^d_lT\right)\\
			&= \delta_{i,c^2} R \times 1\otimes u^d_l v^d_l T + R u^d_l v^d_l \otimes 1 \times \delta_{i,c^2} T
		\end{align*}
		Hence,
		$$ \D_{i,c^2} S =  \delta_{i,c^2} R \widetilde{\#}(u^d_l v^d_l T) + \delta_{i,c^2} T \widetilde{\#} (Ru^d_l v^d_l) $$
		Consequently, $\delta_{i,c^1}\D_{i,c^2} S$ is a linear combination of terms of the form $A u^d_l v^d_l B\otimes C$ and $C\otimes A u^d_l v^d_l B$ with $A,B,C$ monomials. But then with $Q=A u^d_l v^d_l B$, we get that, 
		\begin{align*}
			\delta_{j,c^3}^2 Q(\X_1^s,Z) \boxtimes&\ \delta_{j,c^3}^1 Q(\widetilde{\X}_1^s,Z) = \left(\delta_{j,c^3}^2 A(\X_1^s,Z) u^{d+c_k}_l v^{d+c_k}_l B(\X_1^s,Z)\right) \boxtimes \delta_{j,c^3}^1 A(\widetilde{\X}_1^s,Z) \\
			&+ \delta_{j,c^3}^2 B(\X_1^s,Z) \boxtimes  \left(A(\widetilde{\X}_1^s,Z) u^{d+4c_k+3}_l v^{d+4c_k+3}_l \right) \delta_{j,c^3}^1 B(\widetilde{\X}_1^s,Z).
		\end{align*}
		Thus after evaluating in $y^{T_{k+1}}_m$, we get that
		\begin{align*}
			&\left(\delta_{j,c^3}^2 Q(\X_1^s,Z) \boxtimes \delta_{j,c^3}^1 Q(\widetilde{\X}_1^s,Z)\right)\left( y^{T_{k+1}}_m , Z^N \right) \\
			& = \left(\delta_{j,c^3}^2 (AB)(\X_1^s,Z) \boxtimes \delta_{j,c^3}^1 (AB)(\widetilde{\X}_1^s,Z)\right)\left( y^{T_{k+1}}_m , Z^N \right).
		\end{align*}
		The case of $C\otimes A u^d_l v^d_l B$ is the same, hence we get that
		\begin{equation}
			\label{dkjdsod}
			\Theta_{j,c,m,s}\left[\delta_{i,c^1}\D_{i,c^2} (R\ u_l^c v_l^c T)\right]\left( y^{T_{k+1}}_m, Z^N \right) = \Theta_{j,c,m,s}\left[\delta_{i,c^1}\D_{i,c^2} (RT)\right]\left( y^{T_{k+1}}_m , Z^N \right).
		\end{equation}
		Let us remind that we want to prove Equation \eqref{dfvdvs}. Given $(u_t)_{t\geq 0}$ a free unitary Brownian motions, by linearity and by induction on the degree of $S$ we can assume that $S$ is a monomial in the following quantities,
		\begin{itemize}
			\item $(u_i^c)^n,(v_i^c)^n$, for $ {1\leq i\leq d,\ c\in [1,c_k],\ \dep^k(c)=m}$,
			\item $(u_i^c)^n-\tau\left((u_{\widetilde{t}_l-\widetilde{t}_{l-1}})^n\right), (v_i^c)^n-\tau\left((u_{\widetilde{t}_l-\widetilde{t}_{l-1}}^*)^n\right)$, for $1\leq i\leq d$, $ c\in [1,c_k]$, $ \dep^k(c)=l<m $,	
			\item $R - \tau_N\left(R\left(Z^N\right)\right)\in \C\langle Z\rangle$.
		\end{itemize}
		
		\begin{rem}
			\label{ksjdvncksn}
			Note that after evaluating one of those elements in $\left( y^{T_{k+1}}_m , Z^N \right)$, one gets respectively a power of a free Haar unitary, a power of a free Brownian motion recentered, or a matrix of size $N$.
		\end{rem}
		
		Besides, thanks to Equations \eqref{dkjdsod} one can also assume that if 
		\begin{equation}
			\label{skdjncws}
			S=P_1\cdots P_r
		\end{equation}
		where every $P_i$ is an element of the list above, then we do not have two consecutive elements of $ \C\langle Z \rangle$, and we do not have two consecutive terms which both involves $u_i^c$ or $v_i^c$ for the same $c$ and $i$. Thus $\Theta_{j,c,m,s}\left[\delta_{i,c^1}\D_{i,c^2} S\right]\left( y^{T_{k+1}}_m , Z^N \right)$ is a linear combination of polynomials of the form
		$$ A_1 R_1 A_2 A_3\widetilde{R}_1 A_4 A_5 \widetilde{R}_2 A_6 A_7 R_2 A_8,$$ 
		
		\noindent that we evaluate in $\left( y^{T_{k+1}}_m , Z^N \right)$. The polynomials $A_i$ are terms which appears when differentiating one of the $P_l$, consequently after evaluating them, those are polynomials in one and only one of the variables, $f_{\widetilde{t}_m-\widetilde{t}_{m-1}}\left(u_{i,\widetilde{t}_m-\widetilde{t}_{m-1}}^{I_m}\right)$ or $f_{\widetilde{t}_m-\widetilde{t}_{m-1}}\left(u_{i,\widetilde{t}_m-\widetilde{t}_{m-1}}^{I_m} \right)^*$ for $I\in J_{k+1}$, $i\in [1,d]$. Indeed, with $c_1,c_2$ as in Equation \eqref{ijdncsincds}, $c_3,c_4$ as in the definition of $\Theta_{j,c,m,s}$ (see Equation \eqref{xdkovse}), one has that 
		$$\dep^k(c_1)=\dep^k(c_2)=\dep^k(c_3)=\dep^k(c_4)=m.$$
		And thus, for any $l$, either $\delta_{i,c_1}P_l=0\otimes 0$ or $P_l$ is equal to $(u_i^{c_1})^n$ or $(v_i^{c_1})^n$ for some $n$, hence $\delta_{i,c_1}P_l$  is equal to 
		$$\sum_{g=1}^n (u_i^{c_1})^g\otimes (u_i^{c_1})^{n-g}\quad \text{ or }\quad -\sum_{g=1}^n (v_i^{c_1})^{n-g}\otimes (v_i^{c_1})^{g}.$$
		The polynomials $R_1,R_2,R_3,R_4$ are of the form $P_u\dots P_v$ with the same polynomials $P_l$ as in Equation \eqref{skdjncws}. Let us assume for example that $R_1=P_u\dots P_v$, then since $A_1$ comes from differentiating $P_{u-1}$ and $A_2$ from $P_{v+1}$, we do not have two consecutive elements of $ \C\langle Z \rangle$, and we do not have two consecutive terms which both involves $u_i^c$ or $v_i^c$ for the same $c$ and $i$. Thus after evaluating $A_1R_1A_2$ in $\left( y^{T_{k+1}}_m , Z^N \right)$, this yields a product $a_1\dots a_b$ where for every $j$, $\tau(a_j)=0$ and $a_j$ is free from $a_{j+1}$. 
		
		Thus similarly we get that, $A_1 R_1 A_2$, $A_3\widetilde{R}_1 A_4$, $A_5 \widetilde{R}_2 A_6$ and $A_7 R_2 A_8$ after evaluation in $\left( y^{T_{k+1}}_m , Z^N \right)$ are all such products. We denote those variables by $a^1_j,\widetilde{a}^1_j,\widetilde{a}^2_j,a^2_j$ respectively.	Note that after evaluating 
		
		Since we assumed that $m\leq s$, as one can see from Equation \eqref{xdkovse}, the variables which appears in $A_1R_1A_2$, $A_3\widetilde{R}_1A_4$, $A_5\widetilde{R}_2A_6$ and $A_7R_2A_8$ are respectively of the form $u_i^{c+c_k}$, $u_i^{c+2c_k}$, $u_i^{c+4c_k+3}$ and $u_i^{c+5c_k+3}$. Thus after evaluating in $\left( y^{T_{k+1}}_m , Z^N \right)$, keeping Remark \ref{ksjdvncksn} in mind, with
		$$ \Delta^1=\{a_j\ |\ a_j\notin \M_N(\C) \},$$
		if we define similarly $\widetilde{\Delta}^1$, $\widetilde{\Delta}^2$ and $\Delta^2$, then those families of non-commutative random variables are free between each other.  
		
		By traciality one can consider $A_3\widetilde{R}_1A_4$, $A_5\widetilde{R}_2A_6$ or $A_7R_2A_8$ instead of $A_1R_1A_2$, hence one can assume that $\Delta^1$ is non-empty. Let $l$ be the maximal $j$ such that $a_i^1\notin\M_N(\C)$. Then $a_l^1$ is free from
		$$ X = a^1_{l+1}\dots a_b^1 (A_2 A_3\widetilde{R}_1 A_4 A_5 \widetilde{R}_2 A_6 A_7 R_2 A_8 A_1)\left( y^{T_{k+1}}_m , Z^N \right). $$
		Thus
		\begin{align*}
			&\tau_N\left((A_1 R_1 A_2 A_3\widetilde{R}_1 A_4 A_5 \widetilde{R}_2 A_6 A_7 R_2 A_8)\left( y^{T_{k+1}}_m , Z^N \right)\right) \\
			&= \tau_{N}(a_1^1 \dots a_l^1 (X-\tau(X)) ) + \tau_{N}(a_1^1 \dots a_l^1 ) \tau_{N}(X) \\
			&=0.
		\end{align*}
		
	\end{proof}
	
	\textbf{Step 2:} Let us now prove Equation \eqref{doijvsed} with $s<m$. This case is more difficult then the previous one for the following reason, if we try to adapt the proof of Step 1 and define once again $ \Delta^1$, $\widetilde{\Delta}^1$, $\widetilde{\Delta}^2$ and $\Delta^2$ as above, then those families of non-commutative random variables will not be free between each other and we won't be able to conclude. Thus it is necessary to find a different approach.
	
	Once again, as in Equation \eqref{ksncsa}, we have that
	\begin{equation*}
		\tau_N\left(  L^{k,s}_{\rho_{k+1},\beta_{k+1},\gamma_{k+1},\delta_{k+1}}(R)(u^{N,T_{k+1}},Z^N)\right) = \tau_N\left(  L^{k,s}_{\rho_{k+1},\beta_{k+1},\gamma_{k+1},\delta_{k+1}}(R)(v^{T_{k+1}},Z^N)\right).
	\end{equation*}
	
	\noindent As in the previous step, we use the notations introduced in Definition \ref{kdjvnssc}. And if $m\leq 2k+1$ we set 
	$$ S = R\left( \left(\prod_{l=1}^{m} u^{I_{l}}_{i}\right)_{i\in [1,d], I\in J_{i_0,\dots,i_{n-1}}} , Z \right),$$	
	and if $m=2k+2$, we set
	$$ S = R\left( \left(\prod_{l=1}^{2k} u^{I_{l}}_{i}\ U_i\right)_{i\in [1,d], I\in J_{i_0,\dots,i_{n-1}}} , Z \right).$$	
	
	\noindent Note that since $s<m$ if $m = \dep^{k+1}(c)$, then $\dep^{k}(c) = m-1$, hence if $m\leq 2k+1$,
	\begin{equation}
		\delta_i\D_i R\left( \left(\prod_{l=1}^{m} u^{I_{l}}_{i}\right)_{i\in [1,d], I\in J_{i_0,\dots,i_{n-1}}} , Z \right) = \sum_{\substack{c^1,c^2\in [1,c_k] \\ \dep^k(c^1)=\dep^k(c^2)=m-1}} \delta_{i,c^1}\D_{i,c^2} S,
	\end{equation}
	and if $m=2k+2$,
	\begin{equation}
		\delta_i\D_i R\left( \left(\prod_{l=1}^{2k} u^{I_{l}}_{i}\ U_i\right)_{i\in [1,d], I\in J_{i_0,\dots,i_{n-1}}} , Z \right) = \delta_{i}\D_{i} S.
	\end{equation}
	
	\noindent Then for $P,Q\in \A_{d(1+c_k), q}$ we set for $s\leq 2k$,
	\begin{align}
		\label{skmclskmc}
		\Theta_{j,c,m,s}(P\otimes Q) := &\left(\delta_{j,c}^2 P(\X_1^{m,s},Z) \boxtimes \delta_{j,c}^1 P(\widetilde{\X}_1^{m,s},Z)\right) \\ \nonumber
		&\boxtimes \left(\delta_{j,c}^2 Q(\widetilde{\X}_2^{m,s},Z) \boxtimes \delta_{j,c}^1 Q(\X_2^{m,s} ,Z)\right), 
	\end{align}
	and for $s=2k+1$,
	\begin{align}
		\label{skmclskmc2}
		\Theta_{j,2k+1}(P\otimes Q) := &\left(\delta_{j}^2 P(\X_1^{2k+2,2k+1},Z) \boxtimes \delta_{j}^1 P(\widetilde{\X}_1^{2k+2,2k+1},Z)\right) \\ \nonumber
		& \boxtimes \left(\delta_{j}^2 Q(\widetilde{\X}_2^{2k+2,2k+1},Z) \boxtimes \delta_{j}^1 Q(\X_2^{2k+2,2k+1},Z)\right). 
	\end{align}
	where following the construction of Definition \ref{3biz}, for $s,m\leq 2k$,
	\begin{align*}
		\X_1^{m,s} = \Big(& \left(u_i^{c+c_k}\right)_{\dep^k(c)<s}, \left(u_i^{c+c_k} u_i^{c}\right)_{\dep^k(c)=s}, \left(u_i^{c}\right)_{s< \dep^k(c)\leq m-1}, \\
		& \left(v_i^{c+c_k}\right)_{\dep^k(c)<s}, \left(v_i^{c} v_i^{c+c_k}\right)_{\dep^k(c)=s}, \left(v_i^{c}\right)_{s< \dep^k(c)\leq m-1} \Big)_{i\in [1,d]}, \\
		\X_1^{2k+1,s} = \Big(& \left(u_i^{c+c_k}\right)_{\dep^k(c)<s}, \left(u_i^{c+c_k} u_i^{c}\right)_{\dep^k(c)=s}, \left(u_i^{c}\right)_{s< \dep^k(c)\leq 2k}, u_i^{3c_k+2} \\
		\big(&v_i^{c+c_k}\big)_{\dep^k(c)<s}, \left(v_i^{c} v_i^{c+c_k}\right)_{\dep^k(c)=s}, \left(v_i^{c}\right)_{s< \dep^k(c)\leq 2k}, v_i^{3c_k+2} \Big)_{i\in [1,d]}, \\
		\X_1^{2k+2,2k+1} = \Big(& \left(u_i^{c+c_k}\right)_{\dep^k(c)\leq 2k}, u_i^{3c_k+2} u_i^{3c_k+1}, \\
		& \left(v_i^{c+c_k}\right)_{\dep^k(c)\leq 2k}, v_i^{3c_k+2}v_i^{3c_k+1} \Big)_{i\in [1,d]}.
	\end{align*}
	
	\begin{align*}
		\X_2^{m,s} = \Big(& \left(u_i^{c+2c_k}\right)_{\dep^k(c)<s}, \left(u_i^{c+2c_k} u_i^{c}\right)_{\dep^k(c)=s}, \left(u_i^{c}\right)_{s< \dep^k(c)\leq m-1}, \\
		& \left(v_i^{c+2c_k}\right)_{\dep^k(c)<s}, \left(v_i^{c} v_i^{c+2c_k}\right)_{\dep^k(c)=s}, \left(v_i^{c}\right)_{s< \dep^k(c)\leq m-1} \Big)_{i\in [1,d]}, \\
		\X_2^{2k+1,s} = \Big(& \left(u_i^{c+2c_k}\right)_{\dep^k(c)<s}, \left(u_i^{c+2c_k} u_i^{c}\right)_{\dep^k(c)=s}, \left(u_i^{c}\right)_{s< \dep^k(c)\leq 2k}, u_i^{3c_k+3} \\
		\big(&v_i^{c+2c_k}\big)_{\dep^k(c)<s}, \left(v_i^{c} v_i^{c+2c_k}\right)_{\dep^k(c)=s}, \left(v_i^{c}\right)_{s< \dep^k(c)\leq 2k}, v_i^{3c_k+3} \Big)_{i\in [1,d]}, \\
		\X_2^{2k+2,2k+1} = \Big(& \left(u_i^{c+2c_k}\right)_{\dep^k(c)\leq 2k}, u_i^{3c_k+3} u_i^{3c_k+1}, \\
		& \left(v_i^{c+2c_k}\right)_{\dep^k(c)\leq 2k}, v_i^{3c_k+3}v_i^{3c_k+1} \Big)_{i\in [1,d]}.
	\end{align*}
	
	\noindent We also define $\widetilde{\X}_1^{m,s}$ and $\widetilde{\X}_2^{m,s}$ by adding $3c_k+3$ to every integer. Besides, we set 
	$$ y^{T_{k+1}}_m = \left( \left(u_{i,\widetilde{t}_l-\widetilde{t}_{l-1}}^{I_l}\right)_{l< m}, f_{\widetilde{t}_m-\widetilde{t}_{m-1}}\left(u^{I_{m}}_{i,\widetilde{t}_m-\widetilde{t}_{m-1}}\right) \right)_{i\in [1,d],I\in J_{i_0,\dots,i_{n-1},s}}.$$
	Then thanks to the formula for $L^{k,s}_{\rho_{k+1},\beta_{k+1},\gamma_{k+1},\delta_{k+1}}$ established in Lemma \ref{3apparition}, as well as Equations \eqref{dojs}, \eqref{dojs2}, \eqref{dojs3} and \eqref{dojs4}, we have for $m\leq 2k+1$,
	\begin{align*}
		&\tau_N\left(  L^{k,s}_{\rho_{k+1},\beta_{k+1},\gamma_{k+1},\delta_{k+1}}(R)(u^{T_{k+1}}_m,Z^N)\right) \\
		&= \frac{1}{2} \sum_{1\leq i,j\leq d}\ \sum_{\substack{c,c^1,c^2\in [1,c_k], \\ \dep^k{c}=s,\ \dep^k(c^1)=\dep^k(c^2)=m-1}} \tau_N\left(\Theta_{j,c,m,s}\left[ \delta_{i,c^1}\D_{i,c^2} S\right]\left(y^{T_{k+1}}_m, Z^N \right)\right).
	\end{align*}
	For $m=2k+2$ and $s\leq 2k$,
	\begin{align*}
		&\tau_N\left(  L^{k,s}_{\rho_{k+1},\beta_{k+1},\gamma_{k+1},\delta_{k+1}}(R)(u^{T_{k+1}}_m,Z^N)\right) \\
		&= \frac{1}{2} \sum_{1\leq i,j\leq d}\ \sum_{\substack{c\in [1,c_k], \\ \dep^k{c}=s}} \tau_N\left(\Theta_{j,c,m,s}\left[\delta_i\D_i S\right]\left( y^{T_{k+1}}_m , Z^N \right)\right).
	\end{align*}
	And for $s=2k+1$, 
	\begin{align*}
		&\tau_N\left(  L^{k,s}_{\rho_{k+1},\beta_{k+1},\gamma_{k+1},\delta_{k+1}}(R)(u^{N,T_{k+1}},Z^N)\right) =\frac{1}{2} \sum_{1\leq i,j\leq d} \tau_N\left(\Theta_{j,2k+1}\left[\delta_{i}\D_{i} S\right]\left( y^{T_{k+1}}_m , Z^N \right)\right).
	\end{align*}
	
	\noindent Thus, in order to conclude step 2, one must show the following lemma.
	
	\begin{lemma}
		With the notations introduced previously in the proof of Proposition \ref{pfffffffff}, and notably Equation \eqref{skmclskmc} and \eqref{skmclskmc2} as well as Definition \ref{kdjvnssc}, one has,
		\begin{equation}
			\label{dfvdvsfx}
			\tau_N\left(\Theta_{j,c,m,s}\left[ \delta_{i,c^1}\D_{i,c^2} S\right]\left( y^{T_{k+1}}_m , Z^N \right)\right)=0,
		\end{equation}
		\begin{equation}
			\tau_N\left(\Theta_{j,c,m,s}\left[\delta_i\D_i S\right]\left( y^{T_{k+1}}_m , Z^N \right)\right)=0,
		\end{equation}
		\begin{equation}
			\tau_N\left(\Theta_{j,2k+1}\left[\delta_{i}\D_{i} S\right]\left( y^{T_{k+1}}_m , Z^N \right)\right)=0.
		\end{equation}
	\end{lemma}
	
	\begin{proof}
		
		The proof of those equations are essentially the same, so we only prove Equation \eqref{dfvdvsfx}. Once again, if $S$ is a monomial such that there exist $R,T$ monomials and $i,c$ such that $S = R\ u_l^d v_l^d T$ or $R\ u_l^d v_l^d T$, then with the same proof as in the one of Lemma \ref{sidvbs},
		\begin{equation}
			\label{sljclcelsc}
			\tau_N\left(\Theta_{j,c,m,s}\left[ \delta_{i,c^1}\D_{i,c^2} S\right]\left( y^{T_{k+1}}_m , Z^N \right) \right)=\tau_N\left(\Theta_{j,c,m,s}\left[ \delta_{i,c^1}\D_{i,c^2} (RT)\right]\left( y^{T_{k+1}}_m , Z^N \right)\right).
		\end{equation}
		Besides, one has that
		\begin{align}
			\label{slkcdlsmce}
			\D_{i,c^2} (RT) =&\ \sum_{R = R_1 u_i^{c^2} R_2} R_2TR_1u_i^{c^2} - \sum_{R = R_1 v_i^{c^2} R_2} v_i^{c^2}R_2TR_1 \nonumber \\
			&+ \sum_{T = T_1 u_i^{c^2} T_2} T_2RT_1u_i^{c^2} - \sum_{T= T_1 v_i^{c^2} T_2} v_i^{c^2}T_2RT_1 \\
			=&\ \D_{i,c^2}(TR). \nonumber
		\end{align}
		Hence
		\begin{equation}
			\label{sldvlsmv}
			\tau_N\left(\Theta_{j,c,m,s}\left[ \delta_{i,c^1}\D_{i,c^2} (TR)\right]\left( y^{T_{k+1}}_m , Z^N \right) \right)=\tau_N\left(\Theta_{j,c,m,s}\left[ \delta_{i,c^1}\D_{i,c^2} (RT)\right]\left( y^{T_{k+1}}_m , Z^N \right)\right).
		\end{equation}
		
		\noindent Thus given $(u_t)_{t\geq 0}$ a free unitary Brownian motion, by induction we can assume that $S$ is a monomial in the following quantities,
		\begin{enumerate}
			\item $(u_i^c)^n-\tau\left((u_{\widetilde{t}_l-\widetilde{t}_{l-1}})^n\right), (v_i^c)^n-\tau\left((u_{\widetilde{t}_l-\widetilde{t}_{l-1}}^*)^n\right),$ for $ 1\leq i\leq d,\ c\in [1,c_k],\ \dep^k(c)=l<s$,
			\item $(u_i^c)^n-\tau\left((u_{\widetilde{t}_{s+1}-\widetilde{t}_{s-1}})^n\right),(v_i^c)^n-\tau\left((u_{\widetilde{t}_{s+1}-\widetilde{t}_{s-1}}^*)^n\right),$ for $ 1\leq i\leq d,\ c\in [1,c_k],\\ \dep^k(c)=s$, if $s<m-1$,
			\item $(u_i^c)^n-\tau\left((u_{\widetilde{t}_{l+1}-\widetilde{t}_{l}})^n\right), (v_i^c)^n-\tau\left((u_{\widetilde{t}_{l+1}-\widetilde{t}_l}^*)^n\right)$, for $ 1\leq i\leq d,\ c\in [1,c_k],\ m-1> \dep^k(c)=l>s$,
			\item $u_i^c,\quad {1\leq i\leq d,\ c\in [1,c_k],\ \dep^k(c)=m-1}$,
			\item $v_i^c,\quad {1\leq i\leq d,\ c\in [1,c_k],\ \dep^k(c)=m-1}$,
			\item $R - \tau_N\left(R\left(Z^N\right)\right)\in \C\langle Z_1,\dots, Z_q\rangle$.
		\end{enumerate}
		We say that two elements of the list above are related if they are both elements of $\C\langle Z_1,\dots, Z_q\rangle$, if they are both of type 1,2, or 3 and that they both involve $u_i^c$ or $v_i^c$ for the same $c$ and $i$, or if finally one of them is of type 4, and the other one of type 5, and that they both involve the same $c$ and $i$. The intuition behind this definition is that given a product such that every element is unrelated with the following one, then after evaluating them we will get a product of elements of trace $0$ such that every element is free from the following one. Hence one can then use the second point of Definition \ref{3freeprob} to deduce that the trace of their product is equal to $0$. In particular the trace of a product of two unrelated elements after evaluating them will be $0$.
		
		Thanks to Equation \eqref{sljclcelsc}, one can assume that $S=P_1\cdots P_r$ where every $P_i$ is an element of the list above, and that two consecutive elements are always unrelated. Thanks to Equation \eqref{sldvlsmv}, we can also assume that $P_1$ and $P_r$ are unrelated. Thus one can assume that $\D_{i,c^2} S$ is a linear combination of terms of the form
		$$ \widetilde{S} u_i^{c^2}, \quad v_i^{c^2} \widetilde{S}. $$
		\noindent for some $\widetilde{S}$ which satisfies the same properties as $S$. Thus given that for any polynomials $S_1,S_2$, $\Theta_{j,c,m,s}\left[S_1\otimes 1\right]=\Theta_{j,c,m,s}\left[1\otimes S_2\right]=0$, one can consider that $\delta_{i,c^1}\D_{i,c^2} S$ is a linear combination of elements of the following form,
		\begin{align}
			\label{lskcdlsmclxs}
			&S_1 u_i^{c^1} \otimes S_2 u_i^{c^2},\quad v_i^{c^2} S_1 \otimes v_i^{c^1} S_2, \\ 
			&v_i^{c^2} S_1 u_i^{c^1} \otimes S_2,\quad S_1 \otimes v_i^{c^1} S_2 u_i^{c^2}, \nonumber
		\end{align}
		More precisely to prove Equation \eqref{dfvdvsfx}, we need to show that
		\begin{align}
			\label{slijcwlekcs}
			\tau_N\left(\Theta_{j,c,m,s}\left( \sum_{\widetilde{S} = S_1 u_i^{c^1} S_2} S_1 u_i^{c^1} \otimes S_2 u_i^{c^2} - \sum_{\widetilde{S} = S_1 v_i^{c^1} S_2} S_1 \otimes v_i^{c^1} S_2 u_i^{c^2} \right)\left( y^{T_{k+1}}_m , Z^N \right)\right)=0,
		\end{align}
		and
		\begin{align}
			\label{slkcmwlk}
			\tau_N\left(\Theta_{j,c,m,s}\left( \sum_{\widetilde{S} = S_1 u_i^{c^1} \otimes S_2} v_i^{c^2} S_1 u_i^{c^1} \otimes S_2 - \sum_{\widetilde{S} = S_1 v_i^{c^1} \otimes S_2} v_i^{c^2} S_1 \otimes v_i^{c^1} S_2 \right)\left( y^{T_{k+1}}_m , Z^N \right)\right)=0.
		\end{align}
		
		\noindent Let us study the first case since the other one is similar. $S_1$ and $S_2$ satisfies the same property as $S$, i.e. they are monomials in elements of the list above without related consecutive elements. Besides, if we view them as monomials in elements of the list above, then the last element of $S_1$ and the first of $S_2$ is unrelated with $u_i^{c^1}$ whereas the first element of $S_1$ and the last of $S_2$ is unrelated with $u_i^{c^2}$. First, let us assume that $s<m-1$, then one has, 
		\begin{align}
			\label{oicspa}
			&\Theta_{j,c,m,s}\left[ S_1 u_i^{c^1} \otimes S_2 u_i^{c^2} \right]\left( y^{T_{k+1}}_m , Z^N \right) \\ \nonumber
			&= \sum_{\substack{S_1= T_1 R_1 T_2,\ S_2= T_3 R_2 T_4 \\ R_1,R_2 \text{ of type 2}}} \Big(\Big(\delta_{j,c}^2R_1(\X_1^s)\Big) T_2(\X_1^s)\ u_i^{c^1}\ T_1(\widetilde{\X}_s^1) \left(\delta_{j,c}^1R_1(\widetilde{\X}_s^1)\right) \\
			&\quad\quad\quad\quad\quad\quad\quad\quad\quad \Big(\delta_{j,c}^2R_2(\widetilde{\X}_s^2)\Big) T_4(\widetilde{\X}_s^2)\ u_i^{c^2+3c_k+3}\ T_3(\X_s^2) \delta_{j,c}^1R_2(\X_s^2)\Big)\left( y^{T_{k+1}}_m, Z^N \right), \nonumber
		\end{align}
		
		\begin{align}
			\label{oicspa2}
			&\Theta_{j,c,m,s}\left[ S_1 \otimes v_i^{c^1} S_2 u_i^{c^2} \right]\left( y^{T_{k+1}}_m , Z^N \right)  \\ \nonumber
			&= \sum_{\substack{S_1= T_1 R_1 T_2,\ S_2= T_3 R_2 T_4 \\ R_1,R_2 \text{ of type 2}}} \Big(\Big(\delta_{j,c}^2R_1(\X_1^s)\Big) T_2(\X_1^s)\ T_1(\widetilde{\X}_s^1) \left(\delta_{j,c}^1R_1(\widetilde{\X}_s^1)\right) \\
			&\quad\quad\quad\quad\quad\quad\quad\quad \Big(\delta_{j,c}^2R_2(\widetilde{\X}_s^2)\Big) T_4(\widetilde{\X}_s^2)\ u_i^{c^2+3c_k+3}\ v_i^{c^1} T_3(\X_s^2) \delta_{j,c}^1R_2(\X_s^2)\Big)\left( y^{T_{k+1}}_m, Z^N \right). \nonumber
		\end{align}
		
		\noindent Let us study Equation \eqref{oicspa}, i.e.
		\begin{align*}
			\Big(V(\X_1^s) & T_2(\X_1^s)\ u_i^{c^1}\ T_1(\widetilde{\X}_s^1) V'(\widetilde{\X}_s^1) W^{'}(\widetilde{\X}_s^2) \\
			&T_4(\widetilde{\X}_s^2)\ u_i^{c^2+3c_k+3}\ T_3(\X_s^2) W(\X_s^2)\Big)\left( y^{T_{k+1}}_m, Z^N \right),
		\end{align*}
		
		\noindent where $V,V^{'},W$ and $W^{'}$ are powers of $u_j^{c}$ or $v_j^{c}$. Thus one can set 
		$$T_2(\X_1^s)\left(  y^{T_{k+1}}_m, Z^N\right) = b_1\cdots b_{r_2},$$
		$$T_1(\widetilde{\X}_s^1)\left(  y^{T_{k+1}}_m, Z^N \right) = a_{r_1}\cdots a_{1},$$
		$$T_4(\widetilde{\X}_s^4)\left(  y^{T_{k+1}}_m, Z^N \right) = d_1\cdots d_{r_4},$$
		$$T_3(\X_s^2)\left( y^{T_{k+1}}_m, Z^N \right) = c_{r_3}\cdots c_1,$$
		where one can assume that for every $j$, $\tau(a_j)=\tau(a_j)=0$ and $a_j$ is free from $a_{j+1}$ thanks to our assumption that there are no consecutive related terms. Besides, for the same reason, we also have that 
		\begin{itemize}
			\item $A=V'(\widetilde{\X}_1^s)(y^{T_{k+1}}_m)$, $B=V(\X_1^s)(y^{T_{k+1}}_m)$, $C=W(\X_s^2)(y^{T_{k+1}}_m)$ and $D=W'(\widetilde{\X}_s^2)(y^{T_{k+1}}_m)$ are free from $b_1,a_1,d_1, c_1$,
			\item $\mathcal{U} =  u_i^{c^1}(y^{T_{k+1}}_m)$ and  $\mathcal{V} = u_i^{c^2+3c_k+3}(y^{T_{k+1}}_m)$ are free between themselves as well as unrelated with $b_{r_2}, a_{r_1}, d_{r_4}$ and $c_{r_3}$ (i.e. they are either free or of the same type).
			\item Given that the elements of $\X_1^s$ and $\X_2^s$ are unrelated with those of $\widetilde{\X}_1^s$ and $\widetilde{\X}_2^s$, we have that the variables $(a_j)_{1\leq j\leq r_1},(d_j)_{1\leq j\leq r_4},\U$ are free from $(b_j)_{1\leq j\leq r_2},(c_j)_{1\leq j\leq r_3},\V$ unless they are of type $6$.
		\end{itemize}
		Thus with our new notations, we want to show that 
		\begin{align}
			\label{skcspkcd}
			\tau_N\left( B b_{1}\cdots b_{r_2} \U a_{r_1}\cdots a_1 A D d_1\cdots d_{r_4} \V c_{r_3}\cdots c_1 C\right) =0.
		\end{align}
		Then we use the following strategy, first we have that
		\begin{align*}
			&\tau_N\left( B b_1\cdots b_{r_2} \U a_{r_1}\cdots a_1  A D d_1\cdots d_{r_4} \V c_{r_3}\cdots c_1 C\right) \\
			&= \tau_N\left( (B-\tau_N(B))  b_1\cdots b_{r_2} \U a_{r_1}\cdots a_1  A D d_1\cdots d_{r_4} \V c_{r_3}\cdots c_1 (C-\tau_N(C))\right) \\
			&\quad +  \tau_N\left( (B-\tau_N(B)) b_1\cdots b_{r_2} \U a_{r_1}\cdots a_1  A D d_1\cdots d_{r_4} \V c_{r_3}\cdots c_1 \right) \tau_N(C) \\
			&\quad + \tau_N(B)\tau_N\left( b_1\cdots b_{r_2} \U a_{r_1}\cdots a_1  A D d_1\cdots d_{r_4} \V c_{r_3}\cdots c_1 (C-\tau_N(C))\right) \\
			&\quad + \tau_N(B) \tau_N(C) \tau_N\left( b_1\cdots b_{r_2} \U a_{r_1}\cdots a_1  A D d_1\cdots d_{r_4} \V c_{r_3}\cdots c_1\right).
		\end{align*}
		Thus one can assume that $B$ and $C$ either have trace $0$ or are equal to the identity. From there on, we want to use repeatedly the second point of Definition \ref{3freeprob}. Indeed, 
		\begin{align*}
			&\tau_N\left( B b_1\cdots b_{r_2} \U a_{r_1}\cdots a_1 A D d_1\cdots d_{r_4} \V c_{r_3}\cdots c_1 C\right) \\
			&= \tau_N\left( B b_1\cdots b_{r_2} \U a_{r_1}\cdots a_1 (A D-\tau_N(AD)) d_1\cdots d_{r_4} \V c_{r_3}\cdots c_1 C\right) \\
			&\quad + \tau_N(AD) \tau_N\left( B b_1\cdots b_{r_2} \U a_{r_1}\cdots a_1 d_1\cdots d_{r_4} \V c_{r_3}\cdots c_1 C\right) \\
			&= \tau_N(AD) \tau_N\left( B b_1\cdots b_{r_2} \U a_{r_1}\cdots a_1 d_1\cdots d_{r_4} \V c_{r_3}\cdots c_1 C\right).
		\end{align*}
		Let us remind that since we do not have consecutive related terms, then if $a_i$ and $d_i$ are related, then $(a_{i+1},d_{i+1})$ will be free from $a_id_i$. Hence, if there exists $j$ such that $a_j$ and $d_j$ are unrelated, but $a_i$ and $d_i$ are related for $i<j$, then 
		\begin{align*}
			&\tau_N\left( B b_1\cdots b_{r_2} \U a_{r_1}\cdots a_1 A D d_1\cdots d_{r_4} \V c_{r_3}\cdots c_1 C\right) \\
			&= \tau_N(AD) \tau_N(a_1 d_1)\cdots\tau_N(a_{j-1} d_{j-1}) \tau_N\left( B b_1\cdots b_{r_2} \U a_{r_1}\cdots a_j d_j\cdots d_{r_4} \V c_{r_3}\cdots c_1 C\right) \\
			&=0.
		\end{align*}
		Otherwise, if $r_1\leq r_4$, 
		\begin{align*}
			&\tau_N\left( B b_1\cdots b_{r_2} \U a_{r_1}\cdots a_1 A D d_1\cdots d_{r_4} \V c_{r_3}\cdots c_1 C\right) \\
			&= \tau_N(AD) \tau_N(a_1 d_1)\cdots\tau_N(a_{r_1} d_{r_1}) \tau_N\left( B b_1\cdots b_{r_2} \U d_{r_1+1}\cdots d_{r_4} \V c_{r_3}\cdots c_1 C\right) \\
			&=0.
		\end{align*}
		Otherwise,
		\begin{align*}
			&\tau_N\left( B b_1\cdots b_{r_2} \U a_{r_1}\cdots a_1 A D d_1\cdots d_{r_4} \V c_{r_3}\cdots c_1 C\right) \\
			&= \tau_N(AD) \tau_N(a_1 d_1)\cdots\tau_N(a_{r_4} d_{r_4}) \tau_N\left( B b_1\cdots b_{r_2} \U a_{r_1}\cdots a_{r_4+1} \V c_{r_3}\cdots c_1 C\right).
		\end{align*}
		If $a_{r_4+1}\neq \V^*$, then the above equation is equal to $0$. Otherwise
		\begin{align*}
			&\tau_N\left( B b_1\cdots b_{r_2} \U a_{r_1}\cdots a_1 A D d_1\cdots d_{r_4} \V c_{r_3}\cdots c_1 C\right) \\
			&= \tau_N(AD) \tau_N(a_1 d_1)\cdots\tau_N(a_{r_4} d_{r_4}) \tau_N\left( B b_1\cdots b_{r_2} \U a_{r_1}\cdots a_{r_4+2} c_{r_3}\cdots c_1 C\right) \\
			&= \tau_N(AD) \tau_N(a_1 d_1)\cdots\tau_N(a_{r_4} d_{r_4}) \tau_N\left( c_{r_3}\cdots c_1 C B b_1\cdots b_{r_2} \U a_{r_1}\cdots a_{r_4+2} \right) \\
			&= \tau_N(AD) \tau_N(a_1 d_1)\cdots\tau_N(a_{r_4} d_{r_4}) \tau_N(CB) \tau_N\left( c_{r_3}\cdots c_1 b_1\cdots b_{r_2} \U a_{r_1}\cdots a_{r_4+2} \right).
		\end{align*}
		Then with the same reasoning  we get that if there exists $j$ such that $c_j$ and $b_j$ are unrelated, but $c_i$ and $b_i$ are related for $i<j$, then the above quantity is equal to $0$. Similarly if $r_2\geq r_3$, or that $c_{r_2+1}\neq \U^*$ then it is also equal to $0$. Otherwise we have that
		\begin{align*}
			&\tau_N\left( B b_1\cdots b_{r_2} \U a_{r_1}\cdots a_1 A D d_1\cdots d_{r_4} \V c_{r_3}\cdots c_1 C\right) \\
			&= \tau_N(AD) \tau_N(a_1 d_1)\cdots\tau_N(a_{r_4} d_{r_4}) \tau_N(CB) \tau_N(c_1 b_1) \cdots \\
			&\quad\quad\quad\quad\quad\quad\quad \quad\quad\quad\quad\quad\quad\quad\quad\quad \cdots \tau_N(c_{r_2}b_{r_2}) \tau_N\left( c_{r_3}\cdots c_{r_2+2} a_{r_1}\cdots a_{r_4+2} \right).
		\end{align*}
		Let us remind that in the case where $a_{r_4+1}= \V^*$ and $c_{r_2+1}= \U^*$ we cannot have that $r_1 = r_4+ 1$ or $r_3=r_2+1$ since otherwise this would contradict $\U$ and $\V$ being unrelated with $b_{r_2}, a_{r_1}, d_{r_4}$ and $c_{r_3}$. Besides, if $c_{r_2+2}$ and $a_{r_1}$ are not of type 6, then they are free and the above quantity is equal to $0$. Finally, if  $r_3>r_2+2$ or $r_1 > r_4+ 2$, then since $c_{r_2+3}$ and $a_{r_1}$ cannot be of type 6, we have that
		\begin{align*}
			&\tau_N\left( B b_1\cdots b_{r_2} \U a_{r_1}\cdots a_1 A D d_1\cdots d_{r_4} \V c_{r_3}\cdots c_1 C\right) \\
			&= \tau_N(AD) \tau_N(a_1 d_1)\cdots\tau_N(a_{r_4} d_{r_4}) \tau_N(CB) \tau_N(c_1 b_1) \cdots \tau_N(c_{r_2}b_{r_2}) \\
			&\quad\quad \tau_N(c_{r_2+2} a_{r_1}) \tau_N\left( c_{r_3}\cdots c_{r_2+3} a_{r_1-1}\cdots a_{r_4+2} \right) \\
			&=0.
		\end{align*}
		Hence this leaves us with the following case,
		\begin{itemize}
			\item $r_3 = r_2+2$ and $r_1 = r_4+ 2$, 
			\item $\forall i\leq r_2$, $c_i$ and $b_i$ are related,
			\item $\forall i\leq r_4$, $a_i$ and $d_i$ are related,
			\item  $a_{r_4+1}= \V^*$ and $c_{r_2+1}= \U^*$,
			\item $a_{r_4+2}$ and $c_{r_2+2}$ are of type 6.
		\end{itemize}
		And in that case we have that
		\begin{align*}
			&\tau_N\left( B b_1\cdots b_{r_2} \U a_{r_1}\cdots a_1 A D d_1\cdots d_{r_4} \V c_{r_3}\cdots c_1 C\right) \\
			&= \tau_N(AD) \tau_N(a_1 d_1)\cdots\tau_N(a_{r_4} d_{r_4}) \tau_N(CB) \tau_N(c_1 b_1) \cdots \tau_N(c_{r_2}b_{r_2}) \tau_N(c_{r_2+2} a_{r_4 +2}).
		\end{align*}
		Let us set $A_j,B_j,C_j,D_j\in \A_{d(1+c_k), q}$ such that 
		$$ B_j(\X_1^s)\left( y^{T_{k+1}}_m , Z^N \right) = b_j,$$
		$$A_j(\widetilde{\X}_s^1)\left( y^{T_{k+1}}_m , Z^N \right) = a_j,$$
		$$D_j(\widetilde{\X}_s^4)\left( y^{T_{k+1}}_m , Z^N \right) = d_j,$$
		$$C_j(\X_s^2)\left( y^{T_{k+1}}_m , Z^N \right) = c_j,$$
		Then thanks to Equation \eqref{slkcdlsmce}, there exist $R_1,R_2\in \A_{d(1+c_k), q}$ which are of type 2, such that, 
		\begin{equation}
			\label{slcsoicdw}
			\widetilde{S} = A_{r_4+2} v_i^{c^2} A_{r_4} \cdots A_1 R_1 B_1\cdots B_{r_2} u_i^{c^1} C_{r_2+2} v_i^{c^1} C_{r_2} \dots C_1 R_2 D_1 \dots D_{r_4}.
		\end{equation}
		Let us remind that our goal is to prove Equation \eqref{slijcwlekcs}. After adapting the previous computations to the case of $S_1 \otimes v_i^{c^1} S_2 u_i^{c^2} $, we have that there exist terms in 
		$$\Theta_{j,c,s}\left[S_1 u_i^{c^1} \otimes S_2 u_i^{c^2} \right]\left( y^{T_{k+1}}_m , Z^N \right) \text{ and }\Theta_{j,c,s}\left[S_1 \otimes v_i^{c^1} S_2 u_i^{c^2} \right]\left( y^{T_{k+1}}_m , Z^N \right)$$
		whose trace are not $0$ only if $\widetilde{S}$ is as in Equation \eqref{slcsoicdw}. But then, following the notations of Equation \eqref{oicspa} one can set $S_1 = T_1R_1T_2,$ $ S_2 =T_3R_2T_4$ with
		$$ T_1 =  A_{r_4+2} v_i^{c^2} A_{r_4} \cdots A_1,\quad T_2 = B_1\cdots B_{r_2},$$
		$$ T_3 =  C_{r_2+2} v_i^{c^1} C_{r_2} \dots C_1,\quad T_4 = D_1 \dots D_{r_4}.$$
		And after taking the trace this yields the term 
		\begin{align*}
			&\tau_N\Big( \Big(V(\X_1^s) T_2(\X_1^s)\ u_i^{c^1}\ T_1(\widetilde{\X}_s^1) V'(\widetilde{\X}_s^1) W^{'}(\widetilde{\X}_s^2) \\
			&\quad\quad\quad\quad T_4(\widetilde{\X}_s^2)\ u_i^{c^2+3c_k+3}\ T_3(\X_s^2) W(\X_s^2)\Big)\left( y^{T_{k+1}}_m , Z^N \right)\Big) \\
			&= \tau_N(AD) \tau_N(a_1 d_1)\cdots\tau_N(a_{r_4} d_{r_4}) \tau_N(CB) \tau_N(c_1 b_1) \cdots \tau_N(c_{r_2}b_{r_2}) \tau_N(c_{r_2+2} a_{r_4 +2}).
		\end{align*}
		However, if we take the notations of Equation \eqref{oicspa2} with
		$$ T_1 =  A_{r_4+2} v_i^{c^2} A_{r_4} \cdots A_1,\quad T_2 = B_1\cdots B_{r_2} u_i^{c^1} C_{r_2+2},$$
		$$ T_3 =  C_{r_2} \dots C_1,\quad T_4 = D_1 \dots D_{r_4},$$
		then after taking the trace we get the term
		\begin{align*}
			&\tau_N\Big( \Big(V(\X_1^s) T_2(\X_1^s)\ T_1(\widetilde{\X}_s^1) V'(\widetilde{\X}_s^1) W^{'}(\widetilde{\X}_s^2) \\
			&\quad\quad\quad\quad T_4(\widetilde{\X}_s^2)\ u_i^{c^2+3c_k+3} v_i^{c^1}\ T_3(\X_s^2) W(\X_s^2)\Big)\left( y^{T_{k+1}}_m , Z^N \right)\Big) \\
			&= \tau_N\left( B b_1\dots b_{r_2} \U c_{r_2+2} a_{r_4+2} \V^* a_{r_4}\dots a_1 AD d_1 \dots d_{r_4} \V \U^* c_{r_2}\dots c_1 C\right) \\
			&= \tau_N(AD) \tau_N(a_1 d_1)\cdots\tau_N(a_{r_4} d_{r_4}) \tau_N(CB) \tau_N(c_1 b_1) \cdots \tau_N(c_{r_2}b_{r_2}) \tau_N(c_{r_2+2} a_{r_4 +2}).
		\end{align*}
		Thus those terms cancel each other, hence the proof of Equation \eqref{slijcwlekcs}.
		
		In the case where $s=m-1$, then in Equations \eqref{oicspa} and \eqref{oicspa2} one has to take $R_1,R_2$ of type 4 or 5 instead of 2, one also has to replace $u_i^{c^1}, u_i^{c^2+3c_k+3}, v_i^{c^1}$ by $u_i^{c^1+c_k}u_i^{c^1}$, $ u_i^{c^2+5c_k+3}u_i^{c^2+3c_k+3}$, $ v_i^{c^1}v_i^{c^1+2c_k}$. Besides, there are additional terms which comes from differentiating with respect to $ u_i^{c^1}$ and $u_i^{c^2}$. However, one can handle those terms exactly the same way. Additionally, if $\dep^k(c)= s = m-1$, then 
		$$ u_i^c(\X_1^s)(y^{T_{k+1}}_m) = u_{i,\widetilde{t}_s-\widetilde{t}_{s-1}}^{c+c_k} f_{\widetilde{t}_m-\widetilde{t}_{m-1}}\left(u^{c}_{i,\widetilde{t}_m-\widetilde{t}_{m-1}}\right),$$
		and similarly for $\widetilde{\X}_1^s, \X_2^s, \widetilde{\X}_2^s$. Consequently we still have that
		\begin{itemize}
			\item $ \tau_{N}\left( \left(u_i^c(\X_1^s)(y^{T_{k+1}}_m)\right)^k \right) =0$, if and only if $k\neq 0$.
			\item $ \tau_{N}\left( \left(u_i^c(\X_1^s)(y^{T_{k+1}}_m)\right)^k \left(v_i^c(\X_2^s)(y^{T_{k+1}}_m)\right)^l \right) =0$, if and only if $k\neq l$.
		\end{itemize}
		Hence the rest of the proof remains unchanged.
	\end{proof}
	\noindent This allows us to conclude the proof of Proposition \ref{pfffffffff}.
\end{proof}

We can now prove Theorem \ref{3TTheo}.

\begin{proof}[Proof of Theorem \ref{3TTheo}]
	
	The proof will be divided in two parts, first we prove Equation \eqref{3mainresu}, then we will prove the properties of the coefficients $\alpha_P^i(f,Z^N)$ that we listed in Theorem \ref{3TTheo}.
	
	\textbf{Part 1:} Thanks to Proposition \ref{3intercoef}, we immediately get that :
	\begin{align*}
		\E\Big[ \tau_{N}&\Big(f(P(U^N,Z^N))\Big)\Big] = \sum_{0\leq i\leq k} \frac{1}{N^{2i}} \alpha_i^P(f,Z^N) \\
		&+ \frac{1}{N^{2(k+1)}} \int_{\R} \int_{A_{k+1}} \int_{[0,1]^{4(k+1)}} \E\Big[\tau_N\Big( \Big(L^{{T}_{k+1}}_{\alpha_{k+1},\beta_{k+1},\gamma_{k+1},\delta_{k+1}} \dots \\
		&\quad\quad\quad\quad\quad \dots L^{{T}_1}_{\alpha_1,\beta_1,\gamma_1,\delta_1}\Big)(Q) (U^{N,T_{k+1}},Z^N) \Big)\Big] d\alpha\ d\beta\ d\gamma\ d\delta\  dt_1\dots dt_{2(k+1)}\  d\mu(y) .
	\end{align*}
	
	\noindent All we need to do from now on is to get an estimate on the last term. Let $Q\in \G_{d,q}^n$,  we say that $M\in\G^n_{d,q}$ is a monomial if it is a monomial in $U_{i,I},V_{i,I},Z_j$ and $\left\{e^{R}\ |\ R \text{ polynomial} \right\}$, we denote $\deg M$ the length of $M$ as a word in $U_{i,I},V_{i,I},Z_j$ and $e^{R}$. Then we can write 
	$$ Q = \sum_{1\leq i\leq \Nb(Q)} c_i M_i $$
	where $c_i\in\C$ and $M_i\in \G_{d,q}^n$ are monomials (not necessarily distinct). We also define $C_{\max}(Q) = \max \{1, \sup_i |c_i|\}$. Since for any $I\in J_n$, $\norm{U^{N,T_n}_{i,I}} =1$, given 
	$$ D_N = \max \left(1 , \max \left\{{\norm{Z_j^N}} \right\}_{1\leq j\leq q}\right),$$
	we get that 
	\begin{equation}
		\label{3majorgross}
		\norm{Q(U^{N,T_{n}},Z^N)} \leq \Nb(Q) \times  C_{\max}(Q) \times K_N^{\deg(Q)} .
	\end{equation}
	
	\noindent It is worth noting that this upper bound is not optimal and heavily dependent on the decomposition chosen. We also consider $\widetilde{\G}_{d,q}^n$ the subspace of $\G_{d,q}^n$ whose every element is a polynomial in the variables $U_{i,I},V_{i,I},Z_j$ and  
	$$\left\{e^{\i \lambda y P((U_{i,I})_{1\leq i\leq d},Z)}\ |\ I\in J_n, \lambda\in[0,1]\right\} .$$
	
	\noindent Then ${L}^{T_{n+1}}_{\alpha_n,\beta_n,\gamma_n,\delta_n}$ maps $\widetilde{\G}_{d,q}^n$ to $\widetilde{\G}_{d,q}^{n+1}$. Let $Q\in \widetilde{\G}_{d,q}^n$, we have that 
	$$\deg\left({L}^{T_{n+1}}_{\alpha_n,\beta_n,\gamma_n,\delta_n}(Q)\right) \leq \deg Q + 4\deg P+4, $$
	$$C_{\max}\left({L}^{T_{n+1}}_{\alpha_n,\beta_n,\gamma_n,\delta_n}(Q)\right) \leq (1+|y|)^4\ C_{\max}(P)^4\ C_{\max}(Q), $$
	\begin{align*}
		\Nb\left({L}^{T_{n+1}}_{\alpha_n,\beta_n,\gamma_n,\delta_n}(Q)\right) \leq &\deg(Q) (\deg Q + \deg P+1)(\deg Q + 2\deg P+2) \\
		&\times (\deg Q + 3\deg P +3) \times (\Nb(P)\deg P)^4 \times \Nb(Q) .
	\end{align*}
	
	\noindent Thus if we define by induction $Q_0 = e^{\i y P}$, and $Q_{n+1} = {L}^{{T}_{n+1}}_{\alpha_n,\beta_n,\gamma_n,\delta_n} Q_n $, since $\deg Q_0 = C_{\max}(Q_0) = \Nb(Q_0)=1$, by a straightforward induction we get that
	\begin{equation}
		\deg Q_n \leq 1+4n \left(\deg P +1\right)
	\end{equation}
	\begin{equation}
		\label{3trucamodif}
		C_{\max}(Q_n) \leq  (1+|y|)^{4n}\ C_{\max}(P)^{4n}
	\end{equation}
	\begin{align}
		\Nb(Q_n) &\leq \Big(\Nb(P)\deg P\Big)^{4n} \prod\limits_{j= 0}^{4n-1} (1+j(\deg P +1)) \\
		&\leq \Big(\Nb(P)(\deg P)(1+\deg P)\Big)^{4n} (4n)! \nonumber
	\end{align}	
	
	\noindent Actually since we have $\D_{\delta_1,i} e^{\i y P} = \i y\ \delta_{\delta_1,i} P \widetilde{\#} e^{\i y P} $, one can replace $(1+|y|)^{4n}$ in Equation \eqref{3trucamodif} by $|y|(1+|y|)^{4n-1}$. Thus thanks to Equation \eqref{3majorgross} as well as Propositions \ref{2Browniantounitary} and \ref{pfffffffff}, we get that for $\widetilde{t}_m-\widetilde{t}_{m-1}\geq 5$,
	\begin{align*}
		&\left| \int_{[0,1]^{4(k+1)}} \tau_N\left({L}^{T_{k+1}}_{\alpha_{k+1},\beta_{k+1},\gamma_{k+1},\delta_{k+1}}\dots {L}^{T_1}_{\alpha_1,\beta_1,\gamma_1,\delta_1} Q(U^{N,T_{k+1}},Z^N)\right) d\alpha d\beta d\gamma d\delta\right| \\
		&=\Bigg| \int_{[0,1]^{4(k+1)}} \tau_N\Big({L}^{T_{k+1}}_{\alpha_{k+1},\beta_{k+1},\gamma_{k+1},\delta_{k+1}}\dots {L}^{T_1}_{\alpha_1,\beta_1,\gamma_1,\delta_1} Q(U^{N,T_{k+1}},Z^N) \\
		&\quad\quad\quad\quad\quad\quad\quad\quad -  L^{{T}_{k+1}}_{\alpha_{k+1},\beta_{k+1},\gamma_{k+1},\delta_{k+1}} \dots L^{{T}_1}_{\alpha_1,\beta_1,\gamma_1,\delta_1} Q(u^{N,T_{k+1}},Z^N)\Big) d\alpha d\beta d\gamma d\delta\Bigg| \\
		&\leq \Big\Vert {L}^{T_{k+1}}_{\alpha_{k+1},\beta_{k+1},\gamma_{k+1},\delta_{k+1}}\dots {L}^{T_1}_{\alpha_1,\beta_1,\gamma_1,\delta_1} Q(U^{N,T_{k+1}},Z^N) \\
		&\quad\quad\quad-  L^{{T}_{k+1}}_{\alpha_{k+1},\beta_{k+1},\gamma_{k+1},\delta_{k+1}} \dots L^{{T}_1}_{\alpha_1,\beta_1,\gamma_1,\delta_1} Q(u^{N,T_{k+1}},Z^N) \Big\Vert \\
		&\leq 4e^2 \pi e^{- \widetilde{t}_m + \widetilde{t}_{m-1}} (1+|y|) C_{\max}(P)C_{\max}(Q_{k+1}) \\
		&\quad \times \Nb(Q_{k+1}) \deg(Q_{k+1}) \Nb(P)\deg P \times K_N^{\deg Q_{k+1} + \deg P}\\
		&\leq 4e^2 \pi e^{- \widetilde{t}_m + \widetilde{t}_{m-1}}\times \frac{|y|}{1+|y|} \\
		&\quad \times \Big((1+|y|) C_{\max}(P) \Nb(P)(\deg P)(\deg P+1) K_N^{1+\deg P}\Big)^{4k+5} (4k+5)!.
	\end{align*}
	Note that in the second before last line, we have $K_N^{\deg Q_{k+1} + \deg P}$ instead of $K_N^{\deg Q_{k+1} + \deg P+1}$ since for any self-adjoint element $a$, $\norm{e^{\i a}}=1$. Consequently, after taking the minimum over $m$ and integrating, we get that
	\begin{align*}
		&\Bigg| \int_{\R} \int_{A_{k+1}} \int_{[0,1]^{4(k+1)}} \E\left[\tau_N\Big( \left(L^{{T}_{k+1}}_{\alpha_{k+1},\beta_{k+1},\gamma_{k+1},\delta_{k+1}} \dots L^{{T}_1}_{\alpha_1,\beta_1,\gamma_1,\delta_1}\right)(Q) (U^{N,T_{k+1}},Z^N) \Big)\right] \\
		&\quad\quad\quad\quad\quad\quad\quad\quad\quad\quad\quad\quad\quad\quad\quad\quad\quad\quad\quad\quad\quad\quad\quad d\alpha d\beta d\gamma d\delta\  dt_1\dots dt_{2(k+1)}\  d\mu(y) \Bigg| \\
		&\leq  4e^2 \pi \int_{A_{k+1}} e^{- \max_{1\leq r\leq 2(k+1)} \widetilde{t}_{r} - \widetilde{t}_{r-1}} dt_1\dots dt_{2k+2} \times \int_{\R} |y|(1+|y|)^{4k+4} d|\mu|(y) \\
		&\quad\quad  \times \Big( C_{\max}(P) \Nb(P)(\deg P)(\deg P+1) K_N^{\deg P +1}\Big)^{4k+5} (4k+5)!.
	\end{align*}
	
	\noindent Besides,
	\begin{align*}
		&\int_{A_{k+1}} e^{- \max_{1\leq r\leq 2(k+1)} \widetilde{t}_{r} - \widetilde{t}_{r-1}} dt_1\dots dt_{2k+2} \\
		&\leq \int_{A_{k+1}} e^{-\frac{1}{2k+2} \sum_{1\leq r\leq 2(k+1)} \widetilde{t}_{r} - \widetilde{t}_{r-1}} dt_1\dots dt_{2k+2} \\
		&= \int_{A_{k+1}} e^{-\frac{ t_{2k+2}}{2k+2}} dt_1\dots dt_{2k+2} \\
		&=  (2k+2)^{2k+2} \int_{A_{k+1}} e^{- t_{2k+2}} dt_1\dots dt_{2k+2} \\
		&= (2k+2)^{2k+2}\ ,
	\end{align*}
	
	\noindent and
	\begin{align*}
		\int_{\R} |y|(1+|y|)^{4k+4} d|\mu|(y) \leq 2^{4k+4} \int_{\R} (|y|+ |y|^{4k+5}) d|\mu|(y) .
	\end{align*}
	
	\noindent Thus thanks to Stirling's formula, there exists a constant $C$ such that 
	\begin{align*}
		&\Bigg| \int_{\R} \int_{A_{k+1}} \int_{[0,1]^{4(k+1)}} \E\left[\tau_N\Big( \left(L^{{T}_{k+1}}_{\alpha_{k+1},\beta_{k+1},\gamma_{k+1},\delta_{k+1}} \dots L^{{T}_1}_{\alpha_1,\beta_1,\gamma_1,\delta_1}\right)(Q) (U^{N,T_{k+1}},Z^N) \Big)\right] \\
		&\quad\quad\quad\quad\quad\quad\quad\quad\quad\quad\quad\quad\quad\quad\quad\quad\quad\quad\quad\quad\quad\quad\quad d\alpha d\beta d\gamma d\delta\  dt_1\dots dt_{2(k+1)}\  d\mu(y) \Bigg| \\
		&\leq \int_{\R} (|y|+ |y|^{4k+5} d|\mu|(y) \\
		&\quad \times \Big(C\times K_N^{\deg P+1} C_{\max}(P) \Nb(P)(\deg P)(\deg P+1)\Big)^{4k+5}\times k^{6k}  .
	\end{align*}
	
	\noindent Hence we get Equation \eqref{3mainresu}. We get Equation \eqref{3mainresu2} very similarly. 
	
	\textbf{Part 2:} To prove the last assertion, we only need to consider a function $f$ which takes the value $0$ on a neighborhood of the spectrum of $P(u,Z^N)$. Let $U^{lN}$ be independent Haar unitary matrices of size $lN$, then we get that for any $k$ such that $f$ is smooth enough, thanks to Equation \eqref{3mainresu}, 
	$$ \E\left[ \tau_{lN}\Big(f(P(U^{lN},Z^N\otimes I_l))\Big)\right] = \sum_{0\leq i\leq k} \frac{1}{(lN)^{2i}} \alpha_i^P(f,Z^N\otimes I_l) + \mathcal{O}(l^{-2(k+1)}) .$$
	
	\noindent But in the sense of Definition \ref{3freeprob}, for any $i$, $(u^{T_i},Z^N\otimes I_l)$ and $(u^{T_i},Z^N)$ have the same distribution, hence
	$$ \E\left[ \tau_{lN}\Big(f(P(U^{lN},Z^N\otimes I_l))\Big)\right] = \sum_{0\leq i\leq k} \frac{1}{(lN)^{2i}} \alpha_i^P(f,Z^N) + \mathcal{O}(l^{-2(k+1)}) .$$
	
	\noindent Consequently, if there exists $i$ such that $\alpha_i^P(f,Z^N)\neq 0$, then we can find constants $c$ and $k$ (dependent on $N$) such that 
	\begin{equation}
		\label{3contradhypo}
		\E\left[ \tau_{lN}\Big(f(P(U^{lN},Z^N\otimes I_l))\Big)\right] \sim_{l\to\infty} c\times l^{-2k} . 
	\end{equation}
	
	\noindent We are going to show that the left hand side decays exponentially fast in $l$, hence proving a contradiction. Now if we set $E$ the support of $f$, then
	$$ \left| \E\left[ \tau_{lN}\Big(f(P(U^{lN},Z^N\otimes I_l))\Big)\right] \right| \leq \norm{f}_{\CC^0} \P\left( \sigma\left( P(U^{lN},Z^N\otimes I_l) \right) \cap E \neq \emptyset \right) .$$
	However, there exists a constant $A$ such that for any $l$, $\norm{P(U^{lN},Z^N\otimes I_l)} \leq A $. Thus,
	$$ \left| \E\left[ \tau_{lN}\Big(f(P(U^{lN},Z^N\otimes I_l))\Big)\right] \right| \leq \norm{f}_{\CC^0} \P\left( \sigma\left( P(U^{lN},Z^N\otimes I_l) \right) \cap E \cap [-A,A] \neq \emptyset \right).$$
	Let $g$ be a $\mathcal{C}^{\infty}$-function, with compact support disjoint from the spectrum of $P(u,Z^N)$ such that $g|_{E\cap[-A,A]}=1$. Then,
	\begin{equation}
		\label{sconsodsa}
		\left| \E\left[ \tau_{lN}\Big(f(P(U^{lN},Z^N\otimes I_l))\Big)\right] \right| \leq \norm{f}_{\CC^0} \P\left( \norm{g\left( P(U^{lN},Z^N\otimes I_l) \right)} \geq 1 \right).
	\end{equation}
	
	\noindent Since $g$ is $\mathcal{C}^{\infty}$ and has compact support, thanks to the Fourier inversion formula, we have with $ \hat{g}(y) = \frac{1}{2\pi} \int_{\R} g(x) e^{-\i xy} dx$, that 
	$$ g(x) = \int_{\R} e^{\i x y}\ \hat{g}(y)\  dy, $$
	and besides, $\int |y\hat{g}(y)| dy<\infty$. Thus for any self-adjoint matrices $X$ and $Y$,
	\begin{align*}
		\norm{g(X)-g(Y)} &= \norm{ \int y \int_0^1 e^{\i y U \alpha} (X-Y) e^{\i y V (1-\alpha)} \hat{g}(y) d\alpha dy } \\
		&\leq \norm{X-Y} \int  |y\hat{g}(y)| dy .
	\end{align*}
	
	\noindent Hence there is a constant $C_B$ such that for any unitary matrices $U_i,V_i\in\M_{lN}(\C)$,
	\begin{align*}
		\norm{g(P(U,Z^N)) - g(P(V,Z^N))} \leq C_B \sum_i \norm{U_i - V_i} .
	\end{align*}
	\noindent Consequently, thanks to Theorem 5.17 of \cite{tyrev}, one can find a constant $C$ such that  any $\delta>0$,
	\begin{equation}
		\label{slmcspo}
		\P\left( \norm{g\left( P(U^{lN},Z^N\otimes I_l)\right)} \geq \E\left[\norm{g\left( P(U^{lN},Z^N\otimes I_l)\right)}\right] + \delta \right) \leq e^{-C \delta^2 (lN-2)}.
	\end{equation}
	Besides, by using Equation \eqref{3mainresu} with $k=0$ one has that
	\begin{align*}
		&\E\left[\norm{g\left( P(U^{lN},Z^N\otimes I_l)\right)}\right] \\
		&\leq \E\left[\tr_{lN}\left(g\left( P\left(U^{lN},Z^N\otimes I_l\right)\right)\right)\right] \\
		&= lN\ \tau_{lN}\left(g\left( P\left(u,Z^N\otimes I_l\right)\right)\right) + \mathcal{O}(l^{-1})\\
		&= lN\ \tau_N\left(g\left( P\left(u,Z^N\right)\right)\right) + \mathcal{O}(l^{-1})\\
		&= \mathcal{O}(l^{-1}).
	\end{align*}
	Hence combined with Equations \eqref{sconsodsa} and \eqref{slmcspo}, we get that there exists a constant $K$ such that
	$$\left|\E\left[ \tau_{lN}\Big(f(P(U^{lN},Z^N\otimes I_l))\Big)\right]\right| \leq e^{-Kl},$$
	which is in contradiction with Equation \eqref{3contradhypo}. Hence the conclusion.
	
\end{proof}

We can now prove Theorem \ref{3lessopti}, the only difficulty of the proof is to use the hypothesis of smoothness to replace our function $f$ by a function which satisfies \eqref{3hypoth} without losing too much on the constants.

\begin{proof}[Proof of Theorem \ref{3lessopti}]
	
	To begin with, let
	\begin{equation}
		\label{3fctplateau}
		h:x\to \left\{
		\begin{array}{ll}
			e^{-x^{-4} - (1-x)^{-4}} & \mbox{if } x\in (0,1), \\
			0 & \mbox{else.}
		\end{array}
		\right.
	\end{equation}
	
	\noindent Let $H$ be the primitive of $h$ which takes the value $0$ on $\R^-$, normalized so that it takes the value $1$ for $x\geq 1$. Then given a constant $m$ one can define the function $g : x\to H(m+1-x)H(m+1+x)$ which takes the value $1$ on $[-m,m]$ and $0$ outside of $(-m-1,m+1)$. One has that $\norm{P(U^N,Z^N)} \leq \mathbf{m} C_{\max} K_N^n$, hence we fix $m = \mathbf{m} C_{\max} K_N^n$, thus the spectrum of $P(U^N,Z^N)$ is contained in $[-m,m]$. Consequently,
	\begin{equation}
		\label{3alinf}
		\E\left[ \tau_{N}\Big(f(P(U^N,Z^N))\Big)\right] = \E\left[ \tau_{N}\Big((fg)(P(U^N,Z^N))\Big)\right].
	\end{equation}
	
	\noindent Since $fg$ has compact support and is a function of class $\mathcal{C}^{4(k+1)+3}$, we can take its Fourier transform and then invert it so that with the convention $ \hat{h}(y) = \frac{1}{2\pi} \int_{\R} h(x) e^{-\i xy} dx$, we have
	
	$$ \forall x\in\R,\quad (fg)(x) = \int_{\R} e^{\i xy} \widehat{fg}(y)\ dy . $$
	
	\noindent Besides, since if $h$ has compact support bounded by $m+1$ then $\norm{\hat{h}}_{\CC^0} \leq \frac{1}{\pi}(m+1) \norm{h}_{\CC^0} $, we have
	
	\begin{align*}
		\int_{\R} (|y|+ |y|^{4k+5}) \left| \widehat{fg}(y) \right|\ dy &\leq \int_{\R} \frac{\sum_{i=0}^{4k+7} |y|^i }{1+y^2}\ \left| \widehat{fg}(y) \right|\ dy \\
		&\leq \bigintss_{\R} \frac{\sum_{i=0}^{4k+7} \left| \widehat{(fg)^{(i)}}(y) \right| }{1+y^2}\ dy \\
		&\leq \frac{1}{\pi} \left( m + 1\right) \norm{fg}_{\mathcal{C}^{4k+7}} \int_{\R} \frac{1}{1+y^2}\ dy \\
		&\leq \left( m + 1\right) \norm{fg}_{\mathcal{C}^{4k+7}} ,
	\end{align*}
	
	\noindent Hence $fg$ satisfies the hypothesis of Theorem \ref{3TTheo} with $\mu(dy) = \widehat{fg}(y) dy$. 
	Therefore,  combining with Equation \eqref{3alinf}, by adjusting the constant $C$, we get that
	\begin{align*}
		&\left| \E\left[ \tau_{N}\Big(f(P(U^N,Z^N))\Big)\right] - \sum_{0\leq i\leq k} \frac{1}{N^{2i}} \alpha_i^P(fg,Z^N) \right| \\
		&\leq \frac{1}{N^{2k+2}} \norm{fg}_{\mathcal{C}^{4k+7}} \times \Big(C\times K_N^{n+1} C_{\max} \mathbf{m}\times n(n+1)\Big)^{4k+6}\times k^{6k} . \nonumber
	\end{align*}
	
	\noindent Then one sets $\alpha_i^P(f,Z^N) = \alpha_i^P(fg,Z^N)$. Besides, if $f_1$ and $f_2$ are functions of class $\mathcal{C}^{4k+7}$ equal on a neighborhood of the spectrum of $P(u,Z^N)$, where $u$ is a $d$-tuple of free Haar unitaries free from $\M_N(\C)$, then with the same proof as in the one of Theorem \ref{3TTheo}, one has that for any $i\leq k$, $\alpha_i^P(f_1,Z^N) = \alpha_i^P(f_2,Z^N)$.
	
	Finally, one can write the $j$-th derivative of $x\to e^{-x^{-4}} $ on $\R^+$ as $x\to Q_j(x^{-1})e^{-x^{-4}} $ for some polynomial $Q_j$. By studying $\Nb(Q_j), C_{\max}(Q_j) $ and $\deg(Q_j)$, as in the proof of Theorem \ref{3TTheo}, we get that the infinity norm of the $j$-th derivative of this function is smaller than $20^j j! (5j/4)^{5j/4} $. Hence by adjusting $C$ and using Stirling's formula,
	\begin{align*}
		&\left| \E\left[ \ts_{N}\Big(f(P(U^N,Z^N))\Big)\right] - \sum_{0\leq i\leq k} \frac{1}{N^{2i}} \alpha_i^P(fg,Z^N) \right| \\
		&\leq \frac{1}{N^{2k+2}} \norm{f}_{\mathcal{C}^{4k+7}} \times \Big(C\times K_N^{n+1} C_{\max} \mathbf{m}\times n(n+1)\Big)^{4k+6}\times k^{15k} . \nonumber
	\end{align*}
	
	\noindent The other points of the theorem are a direct consequence of Theorem \ref{3TTheo}.
	
\end{proof}

\subsection{Continuity properties of the coefficients of the asymptotic expansion}

The aim of this subsection is to give some details on the continuity of the coefficients $\alpha_i^P(f,Z^N)$ with respect to their parameters. Indeed, one has the following corollary of Theorem \ref{3TTheo} and more specifically Formula \eqref{exprdescoeff}.

\begin{cor}
	\label{continucoeff}
	With notations and assumptions as in Theorem \ref{3TTheo}, given the following objects,
	\begin{itemize}
		\item $f,g:\R\to\R\in\CC^{4i+3}$,
		\item $P,Q\in \PP_{d,q}$ polynomials of degree at most $n$ and largest coefficient $c_{\max}$,
		\item $Z^N$ and $\widetilde{Z}^N$ tuples of matrices such that for every $i$, $\norm{Z_i^N} \leq K$ and $\norm{\widetilde{Z}_i^N} \leq K$,
	\end{itemize}
	Then there exist non-negative constants $C_i(n,c_{\max},K), C_i^1(n,c_{\max},K,\norm{f}_{4i+3}), C_i^2(n,c_{\max},\\ K,\norm{f}_{4i+3})$ such that with $c_M(\cdot)$ defined as in Equation \eqref{3normA},
	\begin{equation}
		\label{result1}
		\left| \alpha_i^P(f,Z^N) - \alpha_i^P(g,Z^N) \right| \leq C_i(n,c_{\max},K)\ \norm{f-g}_{\CC^{i+3}},
	\end{equation}
	\begin{equation}
		\label{result2}
		\left| \alpha_i^P(f,Z^N) - \alpha_i^Q(f,Z^N) \right| \leq C_i^1(n,c_{\max},K,\norm{f}_{4i+3})\ \sup_{M \text{ monomial}} |c_M(P)-c_M(Q)|,
	\end{equation}
	\begin{equation}
		\label{result3}
		\left| \alpha_i^P(f,Z^N) - \alpha_i^P(f,\widetilde{Z}^N) \right| \leq C_i^2(n,c_{\max},K,\norm{f}_{4i+3})\ \max_i \norm{Z_i^N-\widetilde{Z}_i^N}.
	\end{equation}
	Besides, if $Z^N$ converges in distribution (as defined in Definition \ref{3freeprob}) towards a family $z$, then $\alpha_i^P(f,Z^N)$ converges towards $\alpha_i^P(f,z)$.
\end{cor}

Note that one could estimate the constants $C_i^P(Z^N), C_i^{m,c_{\max}}(Z^N,f)$ and $C_i^P(K,f)$ with respect to their parameters, similarly to how we obtain Equations \eqref{3mainresu} and \eqref{3mainresu2}. However, we do not do it here in order to keep the computations short.

\begin{proof}
	The uniqueness of the coefficients $\alpha_i^P(f,Z^N)$ coupled with the linearity of the map $$f\mapsto \E\left[ \ts_{N}\Big(f(P(U^N,Z^N))\Big)\right],$$ implies the linearity of the map $f\mapsto \alpha_i^P(f,Z^N)$. Hence Equation \eqref{3mainresu02} implies Equation \eqref{result1}. Besides, with $P$ and $Q$ defined as previously, with the same proof than the one of Proposition \eqref{3duhamel},
	\begin{align*}
		&\E\left[ \ts_N\left( e^{\i y P(U^{N},Z^N)} \right) \right] - \E\left[ \ts_N\left( e^{\i y Q(U^{N},Z^N)} \right) \right] \\
		&=\i y \int_{0}^1 \E\left[ \ts_N\left( e^{\i y u P(U^{N},Z^N)} \left(P(U^{N},Z^N) - Q(U^{N},Z^N)\right) e^{\i y (1-u) Q(U^{N},Z^N)} \right) \right]\ du.
	\end{align*}
	Hence thanks to Proposition \ref{3intercoef}, we get that
	\begin{align*}
		&\int_{A_i}\int_{[0,1]^{4i}} \tau_N\Big( \left(L^{{T}_i}_{\rho_i,\beta_i,\gamma_i,\delta_i} \dots L^{{T}_1}_{\rho_1,\beta_1,\gamma_1,\delta_1}\right)(e^{\i y P}-e^{\i y Q}) (u^{{T}_i},Z^N) \Big)\  d\rho\, d\beta\, d\gamma\, d\delta\ dt\\
		&= \i y \int_{0}^1 \int_{A_i} \int_{[0,1]^{4i}} \tau_N\Big( \Big(L^{{T}_i}_{\rho_i,\beta_i,\gamma_i,\delta_i} \dots \\
		&\quad\quad\quad\quad\quad\quad \dots L^{{T}_1}_{\rho_1,\beta_1,\gamma_1,\delta_1}\Big)\left(e^{\i y u P} \left(P - Q\right) e^{\i y (1-u) Q}\right) (u^{{T}_i},Z^N) \Big)  d\rho\ d\beta\ d\gamma\ d\delta\ dt du,
	\end{align*}
	Consequently after integrating over $y$, we have
	\begin{align*}
		&\alpha_i^P(f,Z^N) - \alpha_i^Q(f,Z^N) \\
		&= \i\int_{0}^1  \int_{\R} y \int_{A_i} \int_{[0,1]^{4i}} \tau_N\Big( \Big(L^{{T}_i}_{\rho_i,\beta_i,\gamma_i,\delta_i} \dots \\
		&\quad\quad\quad\quad \dots L^{{T}_1}_{\rho_1,\beta_1,\gamma_1,\delta_1}\Big)\left(e^{\i y u P} \left(P - Q\right) e^{\i y (1-u) Q}\right) (u^{{T}_i},Z^N) \Big) d\rho\, d\beta\, d\gamma\, d\delta\ dt\ d\mu(y)\ du
	\end{align*}
	Since one can write 
	$$ P-Q = \sum_{M\text{ monomial}} (c_M(P)-c_M(Q)) M,$$
	we get Equation \eqref{result2} with the help of Equation \eqref{kdncksmcnlks}. Similarly we have that
	\begin{align*}
		P(X,Z) - P(X,\widetilde{Z}) &= \sum_{M\text{ monomial}} c_M(P) \left(M(X,Z) - M(X,\widetilde{Z})\right) \\
		&= \sum_{M\text{ monomial}} c_M(P) \sum_{M=AZ_iB} A(X,Z)(Z_i-\widetilde{Z}_i)B(X,\widetilde{Z}),
	\end{align*}
	hence we get Equation \eqref{result3}. Finally, if $Z^N$ converges in distribution towards a family $z$, then the family $(u^{{T}_i},Z^N)$ converges in distribution towards $(u^{{T}_i},z)$ where $z$ is free from $u^{{T}_i}$. Indeed, thanks to Equation \ref{kddkdxkfl}, the trace of a polynomial $L$ evaluated in $(u^{{T}_i},Z^N)$ can be expressed into a linear combination of product of traces of polynomials in either $u^{{T}_i}$ or $Z^N$. Then the convergence in distribution of the family $Z^N$ implies that this formula converges towards the same linear combination but whose polynomials are evaluated into $u^{{T}_i}$ or $z$ instead of $u^{{T}_i}$ or $Z^N$, that is the trace of the polynomial $L$ evaluated into $(u^{{T}_i},z)$ where the family $u^{{T}_i}$ and $z$ are free. Thus, thanks to the dominated convergence theorem, Formula \eqref{exprdescoeff} coupled with Equation \eqref{kdncksmcnlks} implies the convergence of $\alpha_i^P(f,Z^N)$ towards $\alpha_i^P(f,z)$.

\end{proof}

\section{Consequences of Theorem \ref{3TTheo}}
\label{conswlf}

\subsection{Proof of Corollary \ref{3voisinage}}

Let $g$ be a non-negative $\mathcal{C}^{\infty}$-function which takes the value $0$ on $(-\infty,1/2]$, $1$ on $[1,\infty)$ and in $[0,1]$ elsewhere. For any $a,b\in \R\cup \{\infty, -\infty\}$, we define $h_{(a,b)}^{\varepsilon} : x\mapsto g(\varepsilon^{-1}(x-a)) g(-\varepsilon^{-1}(x-b) )$ with convention $g(\infty)=1$. Then let $\mathcal{I}_N$ be the collection of connected components of the complementary set of $\sigma(P(u,Z^N))$. Then we define
$$ h^{\varepsilon}_N = \sum_{I\in \mathcal{I}_N} h_I^{\varepsilon} .$$
This function is well-defined since the spectrum of $P(x,Z^N)$ is compact, hence its complementary set has a finite number of connected components of measure larger than $\varepsilon$. And since if $b-a\leq \varepsilon$, $h_{(a,b)}^{\varepsilon} = 0$, the sum over $I\in\mathcal{I}_N$ is actually a finite sum. Besides, we have that
\begin{align*}
	\P\left( \sigma(P(U^N,Z^N)) \not\subset \sigma(P(u,Z^N)) + \varepsilon  \right)\ &\leq\ \P\left( \norm{h^{\varepsilon}(P(U^N,Z^N))} \geq 1 \right) \\
	&\leq\ \E\left[ \tr_N\left( h^{\varepsilon}(P(U^N,Z^N)) \right) \right] .
\end{align*}

\noindent Besides, since $\norm{h_I^{\varepsilon}}_{\mathcal{C}^{4(k+1) +2}}$ is bounded by $C_k\varepsilon ^{-4k-6}$ for $\varepsilon$ small enough where $C_k$ is a constant independent of $N$, and that the supports of the functions $h_I^{\varepsilon}$ are disjoint for $I\in\mathcal{I}_N$, we have that $\norm{h_N^{\varepsilon}}_{\mathcal{C}^{4(k+1) +2}}$ is also bounded by $C_k\varepsilon ^{-4k-6}$. Then thanks to Theorem \ref{3lessopti} since the spectrum of $P(u,Z^N)$ and the support of $h^{\varepsilon}_N$ are disjoint, in combination with the assumption that the operator norm of the matrices $Z^N$ is uniformly bounded over $N$, for any $k\in\N$, we get that there is a constant $C_k$ such that for any $\varepsilon$ and for $N$ large enough, 
$$ \E\left[ \tr_N\left( h^{\varepsilon}(P(U^N,Z^N)) \right) \right] \leq C_k \frac{\varepsilon ^{-4k-6}}{N^{2k+1}} .$$
Thus if we set $\varepsilon = N^{-\alpha}$ with $\alpha<1/2$, then by fixing $k$ large enough we get that
$$ \P\left( \sigma(P(U^N,Z^N)) \not\subset \sigma(P(u,Z^N)) + (-N^{-\alpha},N^{-\alpha})  \right) = \mathcal{O}(N^{-2}) .$$
Hence the conclusion follows by the Borel-Cantelli lemma.

\subsection{Proof of Corollary \ref{3boundednormrenm}}

To begin with let us explain how to handle tensor of matrices. We mainly rely on the following two lemmas.
\begin{lemma}
	\label{dfjivndinv}
	Let $A,B,C,D\in \M_M(\C)$ and $W_1,W_2$ be independent Haar unitary matrices. Then
	\begin{equation}
		\tr_M\left(ABCD\right) = M^2 \E\left[\tr_M\left(B W_1 A W_2 D W_1^* C W_2^*\right)\right].
	\end{equation}
\end{lemma}

\begin{proof}
	For any unitary matrix $V$,
	$$\E\left[ B W_1 A W_2 D W_1^* C W_2^*\right] = \E\left[ B W_1 A V W_2 D W_1^*  C W_2^* V^*\right].$$
	Then let $H$ be a skew-Hermitian matrix, for any $s\in \R$, $e^{sH}\in \mathbb{U}_N$, thus by taking $V=e^{sH}$ and differentiating with respect to $s$ the equality above, we get that,
	$$ \E\left[ B W_1 A H W_2 D W_1^* C W_2^*\right] = \E\left[ B W_1 A W_2 D W_1^* C W_2^* H\right].$$
	
	\noindent Since every matrix is a linear combination of skew-Hermitian matrices (indeed, if $A\in\M_N(\C)$, then $2A = (A-A^*) + \i \times (-\i)(A^*+A)\ $), the previous equality is true for any matrix $H\in \M_N(\C)$. Thus, with $(e_i)_{1\leq i\leq M}$ the canonical basis of $\C^M$, with $H=e_ie_j^*$,
	\begin{align*}
		\sum _{1\leq i,j\leq M} e_i^*\E\left[B W_1 A e_ie_j^* W_2 D W_1^* C W_2^*\right] e_j = 	\sum _{1\leq i,j\leq M} e_i^*\E\left[B W_1 A W_2 D W_1^* C W_2^* e_ie_j^*\right] e_j
	\end{align*}
	Hence,
	\begin{align*}
		M \E\left[\tr_M\left(B W_1 A W_2 D W_1^* C W_2^*\right)\right] &= \E\left[\tr_M\left(B W_1 A\right)\tr_M\left( W_2 D W_1^* C W_2^*\right) \right] \\
		&= \E\left[\tr_M\left(B W_1 A\right)\tr_M\left( D W_1^* C\right) \right] \\
		&= \E\left[\tr_M\left(AB W_1\right)\tr_M\left(W_1^* CD\right) \right].
	\end{align*}
	Besides, with the same reasoning applied to $ABW_1\tr_M\left(W_1^*CD\right)$ we have that,
	\begin{align*}
		\sum _{1\leq i,j\leq M} e_i^*\E\left[ABW_1 e_ie_j^* \tr_M\left( W_1^*CD\right)\right] e_j = \sum _{1\leq i,j\leq M} e_i^*\E\left[ABW_1 \tr_M\left( e_ie_j^* W_1^*CD\right)\right] e_j,
	\end{align*}
	consequently,
	\begin{align*}
		M\E\left[\tr_M\left(AB W_1\right)\tr_M\left(W_1^* CD\right) \right] &= \sum _{1\leq i,j\leq M} e_i^*\E\left[ABW_1 e_ie_j^* \tr_M\left( W_1^*CD\right)\right] e_j\\
		&= \sum _{1\leq i,j\leq M} e_i^*\E\left[ABW_1 \tr_M\left( e_ie_j^* W_1^*CD\right)\right] e_j \\
		&= \sum _{1\leq i,j\leq M} \E\left[\left(ABW_1\right)_{i,j} \left( W_1^*CD\right)_{j,i} \right] \\
		&= \tr_M\left(ABCD\right).
	\end{align*}
	Thus we have that
	\begin{equation*}
		\tr_M\left(ABCD\right) = M^2 \E\left[\tr_M\left(B W_1 A W_2 D W_1^* C W_2^*\right)\right].
	\end{equation*}
\end{proof}

We immediately use the previous lemma to deduce the following.

\begin{lemma}
	With the same notations as in Corollary \ref{3boundednormrenm}, and $(u_t)_{t\geq 0}$ a $d$-tuple of free unitary Brownian motions, if we set
	$$ S_{N,T}^n = \sup_{0\leq t \leq T} \E\left[ \tau_N\otimes\ts_N\Big((P^*P)^{4n}\left( (u_t U^N)\otimes I_M,Z^N\otimes I_M, I_N\otimes Y^M \right)\Big)\right], $$
	then there exists constants $C_P$ and $c_P$ which only depends on $P$ and $\sup_{i,N}\norm{Z^N_i} + \sup_{j,M} \norm{Y^M_j}$ such that
	$$ S_{N,T}^n \leq \left(\norm{P\left( u\otimes I_M,Z^N\otimes I_M, I_N\otimes Y^M \right)}+c_P\ e^{-T/2}\right)^{8n} +  C_P \left(\frac{MTn^2}{N}\right)^2 S_{N,T}^{n-1}. $$	
\end{lemma}

\begin{proof}
	\noindent Let $r$ be the cardinality of the family $Z^N$, $s$ the one of the family $Y^M$. We set $q=r+s$. Given $A\in \A_{d,q}$ a monomial, one can find monomials $A_1\in \A_{d,r}$ and $A_2\in \A_{0,s}$ such that with $(u_t)_{t\geq 0}$ a $d$-tuple of free unitary Brownian motions,
	$$ A\left((u_tU^N)\otimes I_M, Z^N\otimes I_M,I_N\otimes Y^M\right) = A_1(u_tU^N, Z^N) \otimes A_2(Y^M).$$
	Then thanks to Lemma \ref{3apparition}, for $T\geq t\geq 0$,
	\begin{align*}
		&\E\left[\tau_N\otimes\ts_M\Big(A\left((u_tU^N)\otimes I_M, Z^N\otimes I_M,I_N\otimes Y^M\right)\Big)\right] \\
		&= \E\left[\tau_N\otimes\ts_M\Big(A\left((u_{T}U^N)\otimes I_M, Z^N\otimes I_M,I_N\otimes Y^M\right)\Big)\right]  \\
		&\quad + \frac{1}{N^2}\int_{t}^{T} \int_0^{s} \int_{[0,1]^4} \E\left[ \tau_N\left( L^{\{r,s\}}_{\rho,\beta,\gamma,\delta}(A_1)\left( U^{N,\{r,s\}},Z^N \right) \right) \ts_M\left(A_2(Y^M)\right)\right] \\
		&\quad\quad\quad\quad\quad\quad\quad\quad\quad\quad\quad\quad\quad\quad\quad\quad\quad\quad\quad\quad\quad\quad\quad\quad\quad\quad\quad\quad d\rho\ d\beta\ d\gamma\ d\delta\ dr ds.
	\end{align*}
	
	\noindent Consequently with the notations of Lemma \ref{3apparition}, we set for $Q\in \A_{d,q}$
	\begin{align*}
		L(Q) &:= \frac{1}{2} \sum_{1\leq i,j\leq d}\ \Big(\delta_{j}^2\left( \delta_{i}^1 \D_{i} Q\right)\left((U_{i,\{2,1\}})_i,Z,Y\right) R_1 \boxtimes \delta_{j}^1 \left( \delta_{i}^1 \D_{i} Q\right)\left((U_{i,\{5,4\}})_i,Z,Y\right) R_2\Big) \\
		&\quad\quad\quad\quad \boxtimes \Big(\delta_{j}^2\left( \delta_{i}^2 \D_{i} Q\right)\left((U_{i,\{6,4\}})_i,Z,Y\right) R_1^* \boxtimes \delta_{j}^1 \left( \delta_{i}^2 \D_{i} Q\right)\left((U_{i,\{3,1\}})_i,Z,Y\right) R_2^*\Big).
	\end{align*}
	Then if we evaluate $R_1,R_2$ in $I_N\otimes W^1,I_N\otimes W^2$, thanks to Lemma \ref{dfjivndinv},
	\begin{align*}
		&\E\left[\tau_N\otimes\ts_M\Big(A\left((u_tU^N)\otimes I_M, Z^N\otimes I_M,I_N\otimes Y^M\right)\Big)\right] \\
		&= \E\left[\tau_N\otimes\ts_M\Big(A\left((u_{T}U^N)\otimes I_M, Z^N\otimes I_M,I_N\otimes Y^M\right)\Big)\right]  \\
		&\quad + \left(\frac{M}{N}\right)^2\int_{t}^{T} \int_0^{s} \E\Big[\tau_N\otimes\ts_M\Big( L(A)\Big( U^{N,\{r,s\}}\otimes I_M,Z^N\otimes I_M, I_N\otimes Y^M, \\
		&\quad\quad\quad\quad\quad\quad\quad\quad\quad\quad\quad\quad\quad\quad\quad\quad\quad\quad\quad\quad\quad\quad\quad\quad I_N\otimes W^1, I_N\otimes W^2 \Big) \Big)\Big] dr ds.
	\end{align*}
	And by linearity this equality remains true for any polynomial $A$. Thus one sets $A=(P^*P)^{4n}$, and one can view $L(A)$ as a linear combination of terms of the following form
	\begin{align*}
		&\sum_{n_1+n_2+n_3+n_4=l} \left(Q_1 (P^*P)^{n_1} Q_2 \right)\left((U_{i,\{2,1\}})_i,Z,Y\right) R_1  \left(Q_3 (P^*P)^{n_2} Q_4 \right)\left((U_{i,\{5,4\}})_i,Z,Y\right) R_2 \\
		&\quad\quad\quad\times \left(Q_5 (P^*P)^{n_3} Q_6 \right)\left((U_{i,\{6,4\}})_i,Z,Y\right) R_1^* \left(Q_7 (P^*P)^{n_4} Q_8 \right)\left((U_{i,\{3,1\}})_i,Z,Y\right) R_2^*,
	\end{align*}
	evaluated in $\left( U^{N,\{r,s\}}\otimes I_M,Z^N\otimes I_M, I_N\otimes Y^M, I_N\otimes W^1,I_N\otimes W^2 \right)$, and where $l\in [n-4,n-1]$, $Q_1,\dots,Q_8$ are polynomials which do not depends on $n$. Besides, this linear combination does not depend on $n$. 
	
	Note that with $a\in\A_N\otimes \M_M(\C)$, if one sets $\norm{a}_p := \E\left[\tau_{N}\otimes\ts_M(|a|^p)\right]^{1/p}$, one has the following inequalities (see Theorem 2.1.5 of \cite{refreq}),
	\begin{itemize}
		\item $\norm{a b c d}_1 \leq \norm{a}_{\frac{n_1+\dots+n_4}{n_1}} \norm{b}_{\frac{n_1+\dots+n_4}{n_2}} \norm{c}_{\frac{n_1+\dots+n_4}{n_3}} \norm{d}_{\frac{n_1+\dots+n_4}{n_4}}$,
		\item $\norm{a b c}_p \leq \norm{a}_{\infty} \norm{b}_p \norm{c}_{\infty}$.
	\end{itemize}
	Thus there exists a constant $C_P$ which only depends on $P$ and $\sup_{i,N}\norm{Z^N_i} + \sup_{j,M} \norm{Y^M_j}$ such that 
	\begin{align*}
		&\Bigg|\int_{t}^{T} \int_0^{s} \E\Big[\tau_N\otimes\ts_M\Big( L((P^*P)^{4n})\Big( U^{N,\{r,s\}}\otimes I_M,Z^N\otimes I_M, I_N\otimes Y^M,\\
		&\quad\quad\quad\quad\quad\quad\quad\quad\quad\quad\quad\quad\quad\quad\quad\quad\quad\quad\quad\quad\quad\quad\quad\quad I_N\otimes W^1,I_N\otimes W^2 \Big) \Big)\Big] dr ds\Bigg| \\
		& \leq C_P T^2 n^4\ \sup_{0\leq t \leq T} \E\left[\tau_{N}\otimes\ts_M\Big( (P^*P)^{4(n-1)}\left( (u_t U^N)\otimes I_M,Z^N\otimes I_M, I_N\otimes Y^M \right) \Big)\right].
	\end{align*}
	where we used that $P^*P$ is self-adjoint, hence 
	\begin{align*}
		&\left|(P^*P)^{4(n-1)}\left( (u_t U^N)\otimes I_M,Z^N\otimes I_M, I_N\otimes Y^M \right)\right| \\
		&=(P^*P)^{4(n-1)}\left( (u_t U^N)\otimes I_M,Z^N\otimes I_M, I_N\otimes Y^M \right),
	\end{align*}
	as well as the fact that after evaluating in $U^{N,\{r,s\}}$, $U_{i,\{5,4\}},U_{i,\{6,4\}},U_{i,\{3,1\}}$ and $U_{i,\{2,1\}}$ all have the same distribution (in the sense of Definition \ref{3freeprob}) as $u_{t,i} U^N_i$. Thus we get that
	$$ S_{N,T}^n \leq \E\left[ \norm{P\left( (u_T U^N)\otimes I_M,Z^N\otimes I_M, I_N\otimes Y^M \right)}^{8n} \right] +  C_P \left(\frac{MTn^2}{N}\right)^2 S_{N,T}^{n-1}. $$
	Besides, thanks to Proposition \ref{2Browniantounitary}, given $u$ a $d$-tuple of free Haar unitaries, there exists a constant $c_P$ which only depends on $P$ and $\sup_{i,N}\norm{Z^N_i} + \sup_{j,M} \norm{Y^M_j}$ such that
	\begin{align*}
		&\left| \norm{P\left( (u_T U^N)\otimes I_M,Z^N\otimes I_M, I_N\otimes Y^M \right)} - \norm{P\left( u\otimes I_M,Z^N\otimes I_M, I_N\otimes Y^M \right)} \right| \\
		&= \Big| \norm{P\left( (u_T U^N)\otimes I_M,Z^N\otimes I_M, I_N\otimes Y^M \right)} \\
		&\quad\quad\quad\quad\quad\quad\quad\quad\quad\quad\quad\quad\quad\quad\quad\quad - \norm{P\left( (f_T(u)U^N)\otimes I_M,Z^N\otimes I_M, I_N\otimes Y^M \right)} \Big| \\
		&\leq \norm{P\left( u_T U^N\otimes I_M,Z^N\otimes I_M, I_N\otimes Y^M \right) - P\left( f_T(u)U^N\otimes I_M,Z^N\otimes I_M, I_N\otimes Y^M \right)} \\
		&\leq c_P\ e^{-T/2}
	\end{align*}
	Hence the conclusion.
\end{proof}

We immediately get by induction that
\begin{align}
	\label{ldsijvpeaw0}
	S_{N,T}^n \leq &\ \left(\norm{P\left( u\otimes I_M,Z^N\otimes I_M, I_N\otimes Y^M \right)}+c_P\ e^{-T/2}\right)^{8n} \\
	&\quad\quad\times \sum_{k=0}^{n} \left(\frac{\sqrt{C_P} MTn^2}{N \left(\norm{P\left( u\otimes I_M,Z^N\otimes I_M, I_N\otimes Y^M \right)}+c_P\ e^{-T/2}\right)^4}  \right)^{2k}. \nonumber
\end{align}
Thus by taking $T=2\ln(N)$ and $n$ to be the integer part of
\begin{equation}
	\label{sodncskncd}
	\frac{1}{2^{1/4}} \sqrt{\frac{N \norm{P\left( u\otimes I_M,Z^N\otimes I_M, I_N\otimes Y^M \right)}^4}{\sqrt{C_P} M\times 2\ln(N)}},
\end{equation}
one has
\begin{align}
	\label{ldsijvpeaw}
	S_{N,T}^n \leq 2 \left(\norm{P\left( u\otimes I_M,Z^N\otimes I_M, I_N\otimes Y^M \right)}+\frac{c_P}{N} \right)^{8n}.
\end{align}
Thus
\begin{align*}
	&\E\left[ \norm{P\left( U^N\otimes I_M,Z^N\otimes I_M, I_N\otimes Y^M \right)} \right] \\
	&\leq\E\left[ \norm{(P^*P)^{4n}\left( U^N\otimes I_M,Z^N\otimes I_M, I_N\otimes Y^M \right)} \right]^{\frac{1}{8n}} \\
	&\leq (MN)^{\frac{1}{8n}} \E\left[\frac{1}{MN}\tr_{MN}\Big( (P^*P)^{4n}\left( U^N\otimes I_M,Z^N\otimes I_M, I_N\otimes Y^M \right) \Big)\right]^{\frac{1}{8n}} \\
	&\leq (MN)^{\frac{1}{8n}} \left(S_{N,T}^n\right)^{\frac{1}{8n}} \\
	&\leq (2MN)^{\frac{1}{8n}} \left(\norm{P\left( u\otimes I_M,Z^N\otimes I_M, I_N\otimes Y^M \right)}+\frac{c_P}{N}\right).
\end{align*}
Let us first assume that
\begin{equation}
	\label{dijvnskj}
	n\geq \ln(MN)-1.
\end{equation}
One can find a constant $C_P$ which only depends on $P$ and $\sup_{i,N}\norm{Z^N_i} + \sup_{j,M} \norm{Y^M_j}$ such that
\begin{align*}
	&\E\left[ \norm{P\left( U^N\otimes I_M,Z^N\otimes I_M, I_N\otimes Y^M \right)} \right] \\
	&\leq \left(1+C_P\sqrt{\frac{\ln^3(MN)M}{N\norm{P\left( u\otimes I_M,Z^N\otimes I_M, I_N\otimes Y^M \right)}^4}}\right) \\
	&\quad \times \left(\norm{P\left( u\otimes I_M,Z^N\otimes I_M, I_N\otimes Y^M \right)}+\frac{c_P}{N}\right).
\end{align*}
Thus, one can find a constant $C_P$ which only depends on $P$ and $\sup_{i,N}\norm{Z^N_i} + \sup_{j,M} \norm{Y^M_j}$ such that
\begin{align}
	\label{skbdcsjbdc}
	&\E\left[ \norm{P\left( U^N\otimes I_M,Z^N\otimes I_M, I_N\otimes Y^M \right)} \right] \\
	&\leq \norm{P\left( u\otimes I_M,Z^N\otimes I_M, I_N\otimes Y^M \right)} \nonumber \\
	&\quad\quad\quad\quad\quad\quad +C_P\left( \frac{1}{N} + \sqrt{\frac{\ln(N)\ln^2(MN)M}{N\norm{P\left( u\otimes I_M,Z^N\otimes I_M, I_N\otimes Y^M \right)}^4}} \right) \nonumber \\
	&\leq \norm{P\left( u\otimes I_M,Z^N\otimes I_M, I_N\otimes Y^M \right)} \nonumber \\
	&\quad\quad\quad\quad\quad\quad +C_P\sqrt{\frac{M}{N}} \ln^{3/2}(MN) \left(1+\frac{1}{\norm{P\left( u\otimes I_M,Z^N\otimes I_M, I_N\otimes Y^M \right)}^2}\right). \nonumber
\end{align}

Then thanks to Theorem 5.17 of \cite{tyrev}, one can find a constant  $K_P$ which only depends on $P$ and $\sup_{i,N}\norm{Z^N_i} + \sup_{j,M}\norm{Y^M_j}$ such that  any $\delta>0$,
\begin{align*}
	&\P\Big( \norm{P\left( U^N\otimes I_M,Z^N\otimes I_M, I_N\otimes Y^M\right)} \\
	&\quad\quad\quad\quad \geq \E\left[\norm{P\left( U^N\otimes I_M,Z^N\otimes I_M, I_N\otimes Y^M\right)}\right] + \delta \Big) \leq e^{-K_P \delta^2 (N-2)}. \nonumber
\end{align*}
And by combining this equation with Equation \eqref{skbdcsjbdc}, one has
\begin{align*}
	&\P\Bigg(\norm{P\left( U^N\otimes I_M,Z^N\otimes I_M, I_N\otimes Y^M\right)} \geq  \norm{P\left( u\otimes I_M,Z^N\otimes I_M, I_N\otimes Y^M \right)} + \delta \\
	&\quad +C_P\sqrt{\frac{M}{N}} \ln^{3/2}(MN) \left(1+\frac{1}{\norm{P\left( u\otimes I_M,Z^N\otimes I_M, I_N\otimes Y^M \right)}^2}\right) \Bigg) \leq  e^{-K_P \delta^2 (N-2)}.\nonumber
\end{align*}

Finally, keeping in mind Equation \eqref{sodncskncd}, if Equation \eqref{dijvnskj} is not satisfied, then one can find a constant $c_P>0$ which only depends on $P$ and $\sup_{i,N}\norm{Z^N_i} + \sup_{j,M}\norm{Y^M_j}$ such that
$$\frac{\left(\frac{M}{N}\right)^{1/2}\ln^{3/2}(NM)}{\norm{P\left( u\otimes I_M,Z^N\otimes I_M, I_N\otimes Y^M\right)}^2} \geq c_P.$$
Hence one can pick $C_P$ which only depends on $P$ and $\sup_{i,N}\norm{Z^N_i} + \sup_{j,M}\norm{Y^M_j}$ such that
\begin{align*}
	&\P\Bigg(\norm{P\left( U^N\otimes I_M,Z^N\otimes I_M, I_N\otimes Y^M\right)} \geq  \norm{P\left( u\otimes I_M,Z^N\otimes I_M, I_N\otimes Y^M \right)} + \delta \\
	&\quad\quad\quad\quad\quad\quad\quad\quad\quad\quad\quad\quad\quad\quad\quad\quad + C_P\frac{\left(\frac{M}{N}\right)^{1/2}\ln^{3/2}(NM)}{\norm{P\left( u\otimes I_M,Z^N\otimes I_M, I_N\otimes Y^M\right)}^2} \Bigg) =0.\nonumber
\end{align*}
Hence the conclusion remains true.

Besides, if the family $Y^M$ converges strongly in distribution towards $y$, then thanks to Lemma 5.2 of \cite{deux} and Corollary 17.10 from \cite{exact}, we have that $(u\otimes I_{M},1\otimes Y^{M})_{M\geq 1}$ converges strongly in distribution towards $(u\otimes 1, 1\otimes y)$. Consequently thanks to the inequality above, by Borel-Cantelli we get that almost surely, if $M \ll N\ln^{-3}(N)$, then
$$ \limsup_{N\to\infty} \norm{P\left( U^N\otimes I_M, I_N\otimes Y^M\right)} \leq \norm{P\left( u\otimes 1, 1\otimes y \right)}.$$
\noindent Besides, with $Z^{N,M}=I_N\otimes Y^M$, $Z=1\otimes y$ we know thanks to Theorem 5.4.10 of \cite{alice} that if $h$ is a continuous function taking positive values on $\left(\norm{\widetilde{P}\widetilde{P}^*(u\otimes 1, 1\otimes y)}-\varepsilon, \infty \right)$ and taking value $0$ elsewhere, then 
$$\frac{1}{MN}\tr_{MN}(h(P^*P(U^N\otimes I_M, I_N\otimes Y^M)))$$
converges almost surely towards $\tau_{\A}\otimes_{\min}\tau_{\B} (h(P^*P(u\otimes 1, 1\otimes y)))$. If this quantity is positive, then almost surely for $N$ large enough so is $\frac{1}{MN}\tr_{MN}(h(P^*P(U^N\otimes I_M, I_N\otimes Y^M)))$, thus

$$ \norm{P^*P(U^N\otimes I_M, I_N\otimes Y^M)} \geq \norm{P^*P(u\otimes 1, 1\otimes y)} - \varepsilon .$$

\noindent Since $h$ is non-negative and the intersection of the support of $h$ with the spectrum of $P^*P(u\otimes 1, 1\otimes y)$ is non-empty, we have that $h(P^*P(u\otimes 1, 1\otimes y)) \geq 0$ and is not $0$. Besides, we know that the trace on the space where $z$ is defined is faithful, and so is the trace on the $\mathcal{C}^*$-algebra generated by a free Haar unitary, hence by Theorem \ref{3freesum}, so is $\tau_{\A}$. Thus, since both $\tau_{\A}$ and $\tau_{\B}$ are faithful, by Lemma \ref{1faith}, so is $\tau_{\A}\otimes_{\min}\tau_{\B}$ and $\tau_{\A}\otimes_{\min}\tau_{\B} (h(\widetilde{P}\widetilde{P}^*(u\otimes 1, 1\otimes y)))>0$. As a consequence,  almost surely,
$$ \liminf_{N\to \infty} \norm{P^*P\left(U^N\otimes I_M, I_N\otimes Y^M\right)} \geq \norm{P^*P(u\otimes 1, 1\otimes y)} . $$

\noindent We finally conclude  thanks to the fact that for any $z$ in a $\CC^*$-algebra, $\norm{z z^*} = \norm{z}^2$.

\subsection{Proof of Corollary \ref{units}}

We set $y=\max_i |y_{i+1}^N-y_i^N|$. Then thanks to Proposition 5.1 of \cite{deux}, coupled with Theorem 5.17 of \cite{tyrev}, we get that for any polynomial $Q$, there is a constant $C$ such that
\begin{align*}
	&\P\Big(\left|\ts_{N}\left( Q(a^N)\right) - \E\left[\ts_{N}\left( Q(a^N)\right)\right]\right| \geq \delta + \mathcal{O}\left( \frac{y}{N} \right)\Big) \leq 2\ e^{-C \delta^2 N(N-2)/y^2}.
\end{align*}
Thus thanks to Borel-Cantelli lemma, we get that almost surely for any $\varepsilon>0$, for $N$ large enough,
\begin{equation}
	\label{sodcos}
	\ts_{N}\left( Q(a^N)\right) = \E\left[\ts_{N}\left( Q(a^N)\right)\right]+ \mathcal{O}\left( N^{\varepsilon}\frac{y}{N} \right).
\end{equation}
Then thanks to Lemma \ref{3apparition}, with $a_T^N$ defined similarly to $a^N$ but with $u_TU^N$ instead of $U^N$ where $u_T$ is a $d$-tuple of free unitary Brownian motions we get that
\begin{align*}
	\E\left[\ts_{N}\left( Q(a^N)\right)\right] = &\ \E\left[\ts_{N}\left( Q(a^N_T)\right)\right]  \\
	&+ \frac{1}{N^2}\int_{t}^{T} \int_0^{s} \int_{[0,1]^4} \E\left[ \tau_N\left( R_{\alpha,\beta,\delta,\gamma,r,t} \right) \right] d\rho d\beta d\gamma d\delta\ dr ds,
\end{align*}
where $R_{\alpha,\beta,\delta,\gamma,r,t}$ is such that for some constant $C$ independent of  $\alpha,\beta,\delta,\gamma,r,t,y$ and $N$,
$$\E\left[\norm{R_{\alpha,\beta,\delta,\gamma,r,t}}\right] \leq C y^4.$$
Hence we have that
\begin{equation*}
	\E\left[\ts_{N}\left( Q(a^N)\right)\right] = \E\left[\ts_{N}\left( Q(a^N_T)\right)\right]  + \mathcal{O}\left( \frac{T^2y^4}{N^2} \right). 
\end{equation*}
And by defining $a^{\infty}$ like $a^N$ but with $u$ a $d$-tuple of free Haar unitary instead of $U^N$, we get thanks to Proposition \ref{2Browniantounitary} that
\begin{equation*}
	\E\left[\ts_{N}\left( Q(a^N)\right)\right] = \ts_{N}\left( Q(a^{\infty})\right)  + \mathcal{O}\left( y e^{-T/2} + \frac{T^2y^4}{N^2} \right). 
\end{equation*}
Hence by fixing $T=4\ln(N)$, we get that
\begin{equation*}
	\E\left[\ts_{N}\left( Q(a^N)\right)\right] = \ts_{N}\left( Q(a^{\infty})\right)  + \mathcal{O}\left( \frac{\ln(N)^2y^4}{N^2} \right). 
\end{equation*}
Thus combined with Equation \eqref{sodcos}, almost surely
\begin{equation}
	\label{sldc}
	\ts_{N}\left( Q(a^N)\right) = \ts_{N}\left( Q(a^{\infty})\right)+ \mathcal{O}\left( y\sqrt{\frac{\ln(N)}{N}} \right) = \ts_{N}\left( Q(a^{\infty})\right) + o(1).
\end{equation}

Let us now fix $Q_1$ to $Q_p$ be non-commutative polynomials, $i_1,\dots,i_p\in[\![1,k]\!]$ such that for every $j$, $\tau(Q_j(a_{i_j}))=0$ and  if $j<p$, $i_j\neq i_{j+1}$. Then with exactly the same proof as the one of Theorem 1.2 of \cite{FK}, we have that
$$ \lim\limits_{N\to\infty} \tau_N\left( Q_1(a_{i_1}^N) \dots Q_p(a_{i_p}^N)\right) =0.$$
Hence almost surely $(a_1^N,\dots,a_k^N)$ converges in distribution towards the free family $(a_1,\dots,a_k)$.

\begin{acks}[Acknowledgments]
	The author would like to thank Kevin Schnelli, his postdoc supervisor at the time where this paper was first released, as well as Alice Guionnet and Beno\^it Collins his former PhD supervisors under whom he started this project. He would also like to thank the reviewers for providing an especially thorough review of this paper which significantly improved the readability.
\end{acks}




\end{document}